\documentclass[11pt]{amsart}
\usepackage{amsmath}
\usepackage{amssymb, verbatim}
\usepackage{pstricks,pst-node,pst-plot}
\usepackage{comment}
\usepackage{color}

\title[]{Moment maps, nonlinear PDE, and stability in Mirror Symmetry}
\author[T. C. Collins]{Tristan C. Collins}
  \email{tristanc@mit.edu}
  \address{Department of Mathematics, Massachusetts Institute of Technology, 77 Massachusetts Avenue, Cambridge, MA 02139}
 \thanks{T.C.C is supported in part by NSF grant DMS-1506652, DMS-1810924 and an Alfred P. Sloan Fellowship. }

 \author[S.-T Yau]{Shing-Tung Yau}
  \email{yau@math.harvard.edu}
  \address{Department of Mathematics, Harvard University, 1 Oxford Street, Cambridge, MA 02138}

\theoremstyle{plain}
\newtheorem{thm}{Theorem}[section]
\newtheorem{prop}[thm]{Proposition}
\newtheorem{defn}[thm]{Definition}
\newtheorem{lem}[thm]{Lemma}
\newtheorem{cor}[thm]{Corollary}
\newtheorem{conj}[thm]{Conjecture}
\newtheorem{que}[thm]{Question}

\theoremstyle{definition}
\newtheorem{ex}[thm]{Example}
\newtheorem{rk}[thm]{Remark}

\numberwithin{equation}{section}

\newcommand{\Tr}{\textrm{Tr}}

\newcommand{\del}{\partial}

\newcommand{\dbar}{\overline{\del}}
\newcommand{\ddb}{\sqrt{-1}\del\dbar}

\newcommand{\Hom}{\text{Hom}}

\newcommand{\DDb}{\sqrt{-1}D \overline{D}}
\newcommand{\E}{\mathcal{E}}
\newcommand{\osc}{\mathrm{osc}}
\newcommand{\cP}{\mathcal{P}}

\renewcommand{\leq}{\leqslant}
\renewcommand{\geq}{\geqslant}

\renewcommand{\epsilon}{\varepsilon}
\renewcommand{\phi}{\varphi}

\begin{document}

\maketitle

\begin{abstract}
We study the deformed Hermitian-Yang-Mills (dHYM) equation, which is mirror to the special Lagrangian equation, from the variational point of view via an infinite dimensional GIT problem mirror to Thomas' GIT picture for special Lagrangians. This gives rise to infinite dimensional manifold $\mathcal{H}$ mirror to Solomon's space of positive Lagrangians.  In the hypercritical phase case we prove the existence of smooth approximate geodesics, and weak geodesics with $C^{1,\alpha}$ regularity.  This is accomplished by proving sharp with respect to scale estimates for the Lagrangian phase operator on collapsing manifolds with boundary.  We apply these results to the infinite dimensional GIT problem for deformed Hermitian-Yang-Mills.  We associate algebraic invariants to certain birational models of $X\times \Delta$, where $\Delta \subset \mathbb{C}$ is a disk.  Using the existence of regular weak geodesics we prove that these invariants give rise to obstructions to the existence of solutions to the dHYM equation.  Furthermore, we show that these invariants fit into a stability framework closely related to Bridgeland stability.  Finally, we use a Fourier-Mukai transform on toric K\"ahler manifolds to describe degenerations of Lagrangian sections of SYZ torus fibrations of Landau-Ginzburg models $(Y,W)$.  We speculate on the resulting algebraic invariants, and discuss the implications for relating Bridgeland stability to the existence of special Lagrangian sections of $(Y,W)$.
\end{abstract}

\section{Introduction}

Mirror symmetry predicts that Calabi-Yau manifolds come in pairs $(X,\Omega, \omega)$, $(\check{X}, \check{\Omega}, \check{\omega})$ with the property that symplectic geometry on $\check{X}$ is related to complex geometry on $X$ and vice versa.  The physical mechanism underlying mirror symmetry is a duality between type IIA string theory compactified on $X$ and type IIB string theory compactified on $\check{X}$.  Within this duality there is a correspondence between the D-branes in each theory; we refer the reader to \cite{Asp} for a nice discussion of $D$-branes on Calabi-Yau manifolds.  In both the type IIA and type IIB theory the physically realistic D-branes are minimizers of some energy functional, and are referred to as BPS, or supersymmetric D-branes. 

 On the $A$-model the $D$-branes, often referred to as $A$-branes, are known to be Lagrangian submanifolds of $(\check{X}, \check{\Omega}, \check{\omega})$ equipped with flat unitary bundles (as well as certain extended versions of these).  On the $B$-model, the complex geometric side, $D$-branes, often called $B$-branes, can be thought of as holomorphic vector bundles, possibly supported on analytic subsets $V\subset X$. In this language Kontsevich's homological mirror symmetry proposal \cite{Kont} predicts a correspondence between $A$-branes on $\check{X}$, and $B$-branes on $X$
\[
D^{b}{\rm Fuk}(\check{X}) \sim D^{b}Coh(X).
\]
where the left hand side is the derived Fukaya category of $\check{X}$.  Mirror symmetry furthermore predicts a duality between the supersymmetric branes on each side of the correspondence.  On the symplectic side, the supersymmetric $A$-branes are known to be special Lagrangian (sLag) submanifolds of $\check{X}$, together flat unitary bundles.  On the $B$-model, however, the supersymmetry constraint is more mysterious.  Around 2000, three separate approaches to understanding supersymmetric $B$-branes were introduced.  One approach, by Mari\~no-Minasian-Moore-Strominger \cite{MMMS}, used the Dirac-Born-Infeld$+$Chern-Simons functional to compute the equations of motion in the case of abelian gauge group. A second approach by Leung-Yau-Zaslow \cite{LYZ} was to use the Strominger-Yau-Zaslow \cite{SYZ} proposal, and a Fourier-Mukai transform, to compute the mirror object to a special Lagrangian in the setting of semi-flat mirror symmetry \cite{Leung}.  Each of these found that the equations of motion corresponded to a holomorphic line bundle $L \rightarrow X$ with a hermitian metric $h$ solving the equation
\begin{equation}\label{eq: introDHYM}
\begin{aligned}
{\rm Im}(e^{-\sqrt{-1}\hat{\theta}}(\omega +F(h))^{n}) &=0\\
{\rm Re}(e^{-\sqrt{-1}\hat{\theta}}(\omega +F(h))^{n}) &>0
\end{aligned}
\end{equation}
where $\hat{\theta}$ is a constant.  This equation became known as the {\em deformed Hermitian-Yang-Mills} (dHYM) equation.  A third approach, initiated by Douglas-Fiol-R\"omelsberger \cite{DFR} (see also \cite{Doug}), was inspired by the Donaldson-Uhlenbeck-Yau theorem \cite{Do3, UY}.  They bypassed the equations of motion for BPS $B$-branes, and instead proposed an algebro-geometric notion called $\Pi$-stability.  Their proposal can then be summarized as ``an object in $D^{b}Coh(X)$ is a supersymmetric $B$-brane if it is $\Pi$-stable".  This idea was taken up by Bridgeland \cite{Br} who developed the notion of categorical stability conditions in great generality.  Since Bridgeland's pioneering work the subject of stability conditions on categories, and particularly $D^{b}Coh(X)$ and $D^{b}{\rm Fuk}(\check{X})$, has generated a tremendous amount of interest.  However, despite a great deal of progress (see, for example \cite{AB, Bay, BMT, MS} and the references therein), there is no general construction of a Bridgeland stability condition on $D^{b}Coh(X)$ for $X$ Calabi-Yau of dimension larger than $2$ \cite{AB}.

The goal of this paper is to begin to unite these three approaches to supersymmetric $B$-branes.  Indeed, the correspondence between the algebraic notion of supersymmetric $A/B$- branes, and the geometric notion of supersymmetric $A/B$-branes has played an important role in the development of mirror symmetry.  Even before the introduction of $\Pi$-stability, Thomas \cite{Th} and Thomas-Yau \cite{ThY} proposed a notion of stability for Lagrangians and predicted that the Lagrangian $L$ could be deformed by Hamiltonian deformations to a special Lagrangian if and only if $L$ is stable.  This proposal was based in part on a moment map formalism for special Lagrangians discovered by Thomas \cite{Th}, and in part on the analogy with the Donaldson-Uhlenbeck-Yau theorem, motivated by mirror symmetry.  More recently Joyce \cite{J} proposed a very broad update to the Thomas-Yau conjecture in the framework of Bridgeland stability and the mean curvature flow.  Broadly, the folklore conjecture is

\begin{conj}[Folklore]\label{conj: folklore}
There is a Bridgeland stability condition on $D^{b}{\rm Fuk}(\check{X})$ (resp. $D^{b}Coh(X)$) so that the isomorphism class of a Lagrangian $L$ (resp. holomorphic vector bundle $E$) is stable if and only if it contains a special Lagrangian (resp. $E$ admits a metric solving the deformed Hermitian-Yang-Mills equation).
\end{conj}

This conjecture is really two conjectures, the first involving the existence of a Bridgeland stability condition, and the second relating Bridgeland stability and the existence of a solution to a certain nonlinear PDE.

On either side of this conjecture there has been little progress.  Haiden-Katzarkov-Kontsevich-Pandit showed that gradient flows of metrics on semi-stable quiver representations \cite{HKKP}, and the Yang-Mills flow on a holomorphic bundle over a Riemann surface \cite{HKKP1}, give rise to canonically defined filtrations associated with Bridgeland stability conditions, giving evidence for Conjecture~\ref{conj: folklore}.  On the symplectic side Joyce \cite{J} has outlined a program for approaching to Conjecture~\ref{conj: folklore}, based on understanding the singularity formation and surgery of the Lagrangian mean curvature flow (LMCF).  We remark that  Neves \cite{Nev} has shown that finite time singularities of the LMCF are essentially unavoidable, and hence the problem of understanding the long-time behavior of the LMCF is extremely difficult.  At the same time, Imagi-Joyce-Oliveira dos Santos \cite{IJO} have shown how ideas from Floer theory and the Fukaya category can be used to study the singularity formation of the LMCF.

On the holomorphic side, the deformed Hermitian-Yang-Mills equation has recently been studied by Jacob-Yau \cite{JY} and the authors and Jacob \cite{CJY}.  In \cite{CJY} a necessary and sufficient {\em analytic} condition was given for the existence of solutions to dHYM in the critical phase case. It was observed that these conditions gave rise to algebraic obstructions of ``Bridgeland type".  However, outside of this result, and for the case of higher rank vector bundles, essentially nothing is known.

This paper takes up the above folklore conjecture, primarily on the $B$-model.  In particular, we develop the algebro-geometric obstruction theory for the deformed Hermitian-Yang-Mills equation on a line bundle in the hypercritical phase case.  We compare our results with the ``expected" Bridgeland stability condition, and use mirror symmetry to deduce similar results for Landau-Ginzburg models mirror to toric Fano varieties.

Our approach to this problem is to study the mirror of an infinite dimensional GIT framework for special Lagrangians due to Thomas \cite{Th}, and Solomon \cite{Sol}.  To put things in context, let us briefly recall the basic idea of finite dimensional GIT; we refer the reader to \cite{ThNo, GIT} for a thorough discussion.  Suppose $(X,\omega)$ is a projective K\"ahler manifold acted on by a group $G$, which is the complexification of compact real Lie group $K$, acting on $(X,\omega)$ by symplectomorphisms.  By the Hilbert-Mumford criterion, a point $p$ with finite stabilizer is GIT stable if and only if the orbit of $p$ is closed under all $1$-parameter subgroups, which we think of as infinite geodesics in $G$.  The Kempf-Ness theorem makes the the connection with symplectic geometry by associating to $G, p$ a certain $K$-invariant function $f_p$, called the Kempf-Ness functional, constructed out of the moment map for the $K$-action.  The Kempf-Ness functional has the following two properties:  (i) $f_p$ is convex along all one-parameter subgroups of $G/K$, and (ii) $p$ is stable if and only if $f_p$ is proper on $G/K$.  Since $f_p$ is convex, the properness can be checked by evaluating the limit slope of $f_p$ along infinite geodesic rays in $G/K$, a calculation which gives rise to algebraic invariants of the $G$-action whose sign determine the stability of $p$.  Furthermore, the construction of $f_p$ shows that $p$ is stable if and only if the $K$-orbit of $p$ contains a zero of the moment map, which is the usual statement of the Kempf-Ness theorem.

On the $A$-model, \cite{Sol} Solomon introduces a Riemannian structure, and geodesic equation on the infinite dimensional space of {\em positive} (or almost calibrated) Lagrangians.  These geodesics are the one parameter subgroups in the complexified symplectomorphism group in Thomas' infinite dimensional GIT picture \cite{Th,ThY}.   Solomon also introduces two functionals $\mathcal{C}, \mathcal{J}$,  the latter of which is the the Kempf-Ness functional of the GIT problem, and is therefore convex along putative smooth geodesics.

  The existence problem for geodesics in the space of positive Lagrangians has recently generated a great deal of interest.  Solomon-Yuval \cite{SolYu} demonstrated the existence of smooth geodesics between positive Lagrangians in Milnor fibers.  Rubinstein-Solomon~\cite{RuSol} studied the existence of geodesics between graphical positive Lagrangians. They prove that if $f_0, f_1$ are two functions defined on a smoothly bounded domain $D\subset \mathbb{R}^n$  so that
\[
x\longmapsto (x,\nabla f_i(x)) \qquad i=0,1
\]
define positive Lagrangians in $\mathbb{R}^{2n} = \mathbb{C}^{n}$, then there exists a {\em continuous function} $F(x,t)$ having $F(x,0)= f_0$, $F(x,1) = f_1$ and which solves Solomon's geodesic equation in the weak sense of Harvey-Lawson's Dirichlet duality theory.  These weak geodesics were described in terms of envelopes in the hypercritical phase case by Darvas-Rubinstein~\cite{DarRu}.  The Rubinstein-Solomon approach was extended to compact Riemannian manifolds by Dellatorre \cite{Del}, and to the deformed Hermitian-Yang-Mills setting, as developed here, by Jacob \cite{JacPr}. 

In this paper we will develop the mirror of the Thomas-Solomon infinite dimensional GIT picture for the deformed Hermitian-Yang-Mills equation.  We describe a infinite dimensional symplectic manifold, admitting an action by a group of symplectomorphisms, together with a space $\mathcal{H}\subset C^{\infty}(X,\mathbb{R})$ and a Riemannian structure on $\mathcal{H}$, which can be thought of as analogous to $G/K$ in the finite dimensional GIT.  We compute the geodesic equation, and introduce a notion of $\epsilon$-geodesics, which solve an approximate version of the geodesic equation.  We also introduce the complexified Calabi-Yau functional (see Definition~\ref{defn: AY}), and extract from this functional analogues of the $\mathcal{C}$, and  $\mathcal{J}$ functionals, as well as a $\mathbb{C}$-valued functional $Z$.  The $\mathcal{J}$ functional is the Kempf-Ness functional for the GIT problem; it has critical points at solutions of the dHYM equation, and is convex along smooth geodesics.  A fundamental issue in the analogy with finite dimensional GIT is that smooth geodesics need not exist.  Thus, our main analytic contribution is to prove, in the hypercritical phase case, the existence of weak geodesics connecting points in $\mathcal{H}$, with $C^{1,\alpha}$ regularity.  With this much regularity, we can show that the functionals $\mathcal{J}, \mathcal{C}, Z$ are well-defined and we prove that they are convex/concave along these generalized geodesics.

 In order to study the existence of regular geodesics we study the manifolds
\[
(\mathcal{X}_{\epsilon}, \hat{\omega}) = \left(X \times \{ t\in \mathbb{C} : \epsilon e^{-1} < |t| < \epsilon \}, \omega: =\pi_{X}^{*}\omega + \sqrt{-1}dt\wedge d\bar{t} \right)
\]
and solutions of the specified Lagrangian phase equation
\begin{equation}\label{eq: introLagPhase}
\Theta_{\hat{\omega}}(\pi_{X}^{*}\alpha + \DDb\phi): = \sum_{i=0}^{n} \arctan(\mu_{i}) = h
\end{equation}
on $\mathcal{X}_{\epsilon}$.  Here $\mu_{0},\ldots, \mu_n$ are eigenvalues of $\pi_{X}^{*}\alpha + \DDb\phi$ with respect to $\hat{\omega}$.  We prove sharp (with respect to scale) estimates for this equation.  As a PDE question, this seems to be of independent interest.  Namely, suppose we have a domain of the form $M \times [-\epsilon, \epsilon] \subset M \times \mathbb{R}$, where $M$ is a Riemannian manifold, possible with boundary.  Suppose $\phi$ solves a fully non-linear elliptic equation $F(D^{2}\phi)=0$.  By the mean-value theorem, the gradient of $\nabla \phi$ is of order $\frac{1}{\epsilon}$ in the thin directions, and by the comparison principle we expect $D^{2}\phi \approx \epsilon^{-2}$ in directions parallel to the $\mathbb{R}$-factor.  The question then is whether this lack of regularity in the ``thin" directions propagates to directions tangent to $M$.  In the present setting we prove that the lack of regularity {\em does not} propagate.

\begin{thm}
Suppose $\phi(x,t)$ is a smooth $S^1$ invariant function on $(\mathcal{X}_{\epsilon}, \hat{\omega})$ solving~\eqref{eq: introLagPhase}, with $h:\mathcal{X}_{\epsilon} \rightarrow ((n-1)\frac{\pi}{2}, (n+1)\frac{\pi}{2})$, and with $\phi\big|_{\del \mathcal{X}_{\epsilon}} \in \mathcal{H}$. Then there is a constant $C$ independent of $\epsilon$ so that following estimates hold
\[
\osc_{\mathcal{X}_{\epsilon}}\phi + |\nabla^{X}\phi|_{\hat{\omega}} + |\nabla^{X} \overline{\nabla^{X}}\phi|_{\hat{\omega}} \leq C
\]
\[
|\nabla_{\bar{t}}\nabla^{X}\phi|_{\hat{\omega}}+ |\nabla_{t} \phi|_{\hat{\omega}} \leq \frac{C}{\epsilon}
\]
\[
|\nabla_{\bar{t}}\nabla_{t}\phi|_{\hat{\omega}} \leq \frac{C}{\epsilon^2}
\]
\end{thm}

The proof of these estimates depends on exploiting the geometry of $(\mathcal{X}_{\epsilon}, \hat{\omega})$ together with rather subtle concavity properties of the Lagrangian phase operator.   Altogether, these estimates imply the existence of smooth $\epsilon$-geodesics connecting any two potentials in $\mathcal{H}$, and the existence of weak geodesics in $\mathcal{H}$ with $C^{1,\alpha}$ regularity for any $\alpha \in (0,1)$.   We show that these weak geodesics have enough regularity to define the complexified Calabi-Yau functional, and prove that $\mathcal{C}$ is affine, $\mathcal{J}$ is convex, and the real and imaginary parts of $Z$ are concave along these weak geodesics.  Furthermore, using the existence result for geodesics we show that $\mathcal{H}$ has a well-defined metric structure.

With these results in hand, we are in the setting of an infinite dimensional GIT problem with geodesics playing the role of one-parameter subgroups. Using algebraic geometry we construct model infinite rays, analogous to one-parameter subgroups in the space $\mathcal{H}$, and evaluate the limit slope of the Calabi-Yau functional along these model curves in terms of algebraic data.  Using the existence of regular geodesics these model curves give rise to algebro-geometric obstructions to the existence of solutions to dHYM.  For example, we prove that in the hypercritical phase case (see Section~\ref{sec: varFrame}) we have

\begin{thm}
Let $\mathfrak{J}_{1} \subset \mathfrak{J}_{2} \cdots \subset \mathfrak{J}_{r-1} \subset \mathfrak{J}_{r} = \mathcal{O}_{X}$ be a sequence of ideal sheaves, and define
\[
\mathfrak{I} = \mathfrak{J}_{1} +t\cdot \mathfrak{J}_{2} + \cdots t^{r-1}\cdot \mathfrak{J}_{r-1} +(t^r) \subset \mathcal{O}_{X}\otimes \mathbb{C}[t].
\]
Let $\mu: \mathcal{X} \rightarrow X\times \Delta$ be a log resolution of $\mathfrak{I}$, so that $\mu^{-1}\mathfrak{I} = \mathcal{O}_{\mathcal{X}}(-E)$ for an s.n.c divisor $E$.  If $[\alpha]$ admits a solution of the deformed Hermitian-Yang-Mills equation then
\[
E.{\rm Im}\left[\frac{\left(\mu^{*}[\omega] + \sqrt{-1}\left(\mu^{*}[\alpha] -\delta E\right)\right)^{n}}{(\omega+\sqrt{-1}\alpha)^{n}.[X]}\right] \geq 0
\]
for all $\delta>0$ sufficiently small.  Moreover, equality holds if and only of $\mathfrak{I} = (t^\ell)$ for some $0< \ell \leq r$.
\end{thm}

These obstructions, fit into a framework which is closely related to Bridgeland stability \cite{Br}.  Namely, for a holomorphic line bundle we consider the central charge
\[
Z_{X}(L) = -\int_{X}e^{-\sqrt{-1}\omega}ch(L).
\]
We show that if $L$ has a solution of the deformed Hermitian-Yang-Mills equation with hypercritical phase, then $Z_{X}(L)$ lies in the upper half-plane.  Furthermore, we show that for any $V\subset X$ irreducible analytic subset we have that
\[
Z_{V}(L) = -\int_{V}e^{-\sqrt{-1}\omega}ch(L)
\]
lies in the upper half-plane and ${\rm Arg}Z_{V}(L) > {\rm Arg}Z_{X}(L)$.  We furthermore explain the role of Chern number inequalities and their the relation to the phase lifting problem and the notion of a slicing of $D^{b}Coh(X)$.  By the Hilbert-Mumford criterion, it is therefore natural to think of Bridgeland stability for a holomorphic line bundle $L$ as formally predicting the (infinite dimensional) GIT stability of $L$.

Finally, we apply SYZ mirror symmetry, and the Fourier-Mukai transform to study the existence of special Lagrangians on Landau-Ginzburg models mirror to toric K\"ahler manifolds.  We explain how our results imply similar results for Lagrangians, including the existence geodesics in the space of positive Lagrangians, and the construction of certain algebraic degenerations.  We speculate on the relationship with Bridgeland stability conditions on the derived Fukaya category.
 
As suggested by finite dimensional GIT, we expect this theory will have applications in proving the existence of solutions to dHYM.  In the case of K\"ahler-Einstein metrics,  Berman-Boucksom-Jonsson \cite{BBJ} recently gave a variational proof of the Chen-Donaldson-Sun theorem \cite{CDS1, CDS2, CDS3} establishing the existence of K\"ahler-Einstein metrics on $K$-stable Fano manifolds.   Some of the key ingredients in the approach of \cite{BBJ} are: (i) the existence of $C^{1,\alpha}$ geodesics in the space of K\"ahler metrics \cite{Chen}. (ii) the relationship between existence of K\"ahler-Einstein metrics and the properness of a certain Kempf-Ness functional $D$ (in an appropriate sense) \cite{DRu}, using the metric structure on the space of K\"ahler metrics. (iii) the extension of $D$ to a function on non-Archimedean metrics, which allows one to study the limit slopes along infinite geodesic rays \cite{BBJ, BHJ, BHJ2}.  One could hope for a similar approach to proving the existence of solutions to dHYM, using the existence of regular geodesics we establish, together with the convexity/concavity properties of the functionals $\mathcal{J}$ and $Z$.  However, for the dHYM equation there are significant new difficulties related to the phase lifting problem, and related Chern number inequalities; see Section~\ref{sec: StabCond}.

The layout of this paper is as follows.  In Section~\ref{sec: varFrame} we explain the mirror of Thomas' moment map framework and the Thomas-Solomon GIT/variational framework.  We compute the geodesic equation, introduce $\epsilon$-geodesics and study the complexified Calabi-Yau functional.  We also outline the approach to existence of regular geodesic.  In Section~\ref{sec: AnalPrelim} we discuss the properties of the Lagrangian phase operator, and construct barrier functions which play an important role in later estimates.  In Section~\ref{sec: C1est} we prove a priori $C^1$ estimates for solutions to~\eqref{eq: introLagPhase}.  In Section~\ref{sec: C2est} we prove interior $C^2$ estimates for solutions of~\eqref{eq: introLagPhase}.  We also explain how our estimates can be applied to give a streamlined proof of the existence of geodesics in the space of K\"ahler metrics \cite{Chen}.  In Section~\ref{sec: bndryEst} we prove boundary $C^2$ estimates, and combine our work to prove the existence of smooth $\epsilon$-geodesics, $C^{1,\alpha}$ weak geodesics, and prove various convexity statements for relevant functionals along weak geodesics.  
  In Section~\ref{sec: AlgObstr} we construct model curves from algebraic geometry, and use weak geodesics to produce alebro-geometric obstructions to the existence of solutions to dHYM.   In Section~\ref{sec: StabCond} we explain how the results in Section~\ref{sec: AlgObstr} can be put into a coherent framework closely related to Bridgeland stability, including the role of Chern number inequalities in the theory.  We also make some conjectures about relations between algebraic invariants, the existence of solutions to dHYM, and the non-emptiness of $\mathcal{H}$.  Finally, in Section~\ref{sec: Amod} we use the SYZ proposal to translate our theorem from toric K\"ahler manifolds to results about degenerations of Lagrangian sections in Landau-Ginzburg models, and relations with stability conditions on the derived Fukaya category.
\\
\\

{\bf Acknowledgements}: T.C.C would like to thank A. Jacob, J. Ross, and B. Berndtsson for many helpful conversations, as well as  J. Solomon, A. Hanlon, and P. Seidel for helpful conversations concerning special Lagrangians and the Fukaya category.  T.C.C is grateful to the European Research Council, and the Knut and Alice Wallenberg Foundation who supported a visiting semester at Chalmers University, where this work was initiated.  T.C.C would also like to thank Robert Berman, Daniel Persson, David Witt Nystr\"om and the rest of the complex geometry group at Chalmers for providing a stimulating research environment.

\section{The Variational Framework, Geodesics, and Approximate Geodesics}\label{sec: varFrame}

Let $(X,\omega)$ be a compact K\"ahler manifold, and fix a class $[\alpha]\in H^{1,1}(X,\mathbb{R})$.  By the $\del\dbar$-lemma, the $(1,1)$ forms $\alpha'$ lying in the cohomology class $[\alpha]$ are parametrized by functions $\alpha ' = \alpha +\ddb \phi$.  Consider the integral $\int_{X}(\omega+\sqrt{-1}\alpha)^{n} \in \mathbb{C}$.  We will always assume that this integral lies in $\mathbb{C}^{*}$.  Then we  define a unit complex number $e^{\sqrt{-1}\hat{\theta}}$ depending only on $[\omega],[\alpha]$ by
\[
\int_{X}(\omega+\sqrt{-1}\alpha)^{n} \in \mathbb{R}_{>0}e^{\sqrt{-1}\hat{\theta}}.
\]
Motivated by mirror symmetry we introduce the deformed Hermitian-Yang-Mills (dHYM) equation, which seeks a function $\phi \in C^{\infty}(X, \mathbb{R})$ so that $\alpha_{\phi} := \alpha+\ddb \phi$ satisfies
\[
{\rm Im}\left(e^{-\sqrt{-1}\hat{\theta}}(\omega+\sqrt{-1}\alpha_{\phi})^{n}\right)=0.
\]
We refer the reader to \cite{CXY} for a brief introduction to the dHYM equation.  Note that with our present convention, if $\alpha= c_1(L)$, then we are studying~\eqref{eq: introDHYM} on $L^{-1}$.  This convention plays no role, apart from avoiding an abundance of minus signs, until Section~\ref{sec: Amod}.  To write the dHYM equation more concretely, fix a point $p \in X$, and local holomorphic coordinates $(z_1,\ldots,z_n)$ near $p$ so that $\omega(p)_{\bar{j}i} = \delta_{\bar{j}i}$, and $(\alpha_{\phi})_{\bar{j}i} = \lambda_i \delta_{\bar{j}i}$.  Then $\lambda_i$ are the eigenvalues of the Hermitian endomorphism $\omega^{-1}\alpha_{\phi}$, and
\[
\begin{aligned}
\frac{(\omega + \sqrt{-1}\alpha_{\phi})^n}{\omega^{n}} &= \prod_{i=1}^{n}(1+\sqrt{-1}\lambda_i)\\
&= r(\alpha_{\phi})e^{\sqrt{-1}\Theta_{\omega}(\alpha_{\phi})}
\end{aligned}
\]
where 
\begin{equation}\label{eq: defnOperators}
r(\alpha_{\phi}) = \sqrt{\left( \prod_{i=1}^{n}(1+\lambda_i^2)\right)},\qquad \Theta_{\omega}(\alpha_{\phi}) = \sum_{i=1}^{n}\arctan(\lambda_i).
\end{equation}
The function $r(\alpha_{\phi})$ is called the radius function, while $\Theta_{\omega}(\alpha_{\phi})$ is called the Lagrangian phase.  The deformed Hermitian-Yang-Mills equation seeks $\alpha_{\phi}$ so that
\[
\sum_{i=1}^{n}\arctan(\lambda_i) = \hat{\theta} {\rm \mod 2\pi}.
\]
Note that if there is a solution of the deformed Hermitian-Yang-Mills equation then there is a well-defined lift of $\hat{\theta}$ to $\mathbb{R}$.  Furthermore, with this formulation it is clear that the dHYM equation is the complex analogue of the special Lagrangian graph equation.  The following lemma is due to Jacob-Yau \cite{JY}.

\begin{lem}[Jacob-Yau \cite{JY}]\label{lem: BPS}
 Solutions of the deformed Hermitian-Yang-Mills equation minimize the functional
\[
C^{\infty}(X,\mathbb{R})\phi \longrightarrow V(\phi) :=  \int_{X}r(\alpha_{\phi})\omega^{n}.
\]
If $\phi$ is a solution of the deformed Hermitian-Yang-Mills equation then
\[
0< V(\phi) =  \bigg| \int_{X}(\omega+ \sqrt{-1}\alpha)^{n}\bigg|.
 \]
 \end{lem}
 
We define the Lagrangian phase operator by
\begin{equation}\label{eq: LagPhaseOp}
\Theta_{\omega}(\alpha_{\phi}) = \sum_{i=1}^{n}\arctan(\lambda_i).
\end{equation}
In our previous work with Jacob \cite{CJY} we gave necessary and sufficient analytic conditions for the existence of solutions to the deformed Hermitian-Yang-Mills equation.  We proved
\begin{thm}[Collins-Jacob-Yau, \cite{CJY}]
For $\hat{\theta} > (n-2)\frac{\pi}{2}$, there exists a function $\phi$ satisfying
\[
\Theta_{\omega}(\alpha_{\phi})  = \hat{\theta}
\]
if and only if there exists a function $\underline{\phi}: X \rightarrow \mathbb{R}$ so that $\Theta_{\omega}(\alpha_{\underline{\phi}}) >(n-2)\frac{\pi}{2}$, and, for all $1\leq j\leq n$
\begin{equation}\label{eq: solvThmCSub}
\sum_{i\ne j} \arctan(\underline{\lambda}_{i}) > \hat{\theta}-\frac{\pi}{2}
\end{equation}
where $\underline{\lambda}_{i}$ are the eigenvalues of $\omega^{-1}\alpha_{\underline{\phi}}$.
\end{thm}

A primary motivation of this paper is to understand the implications of existence of solutions to the dHYM equation with the goal of replacing~\eqref{eq: solvThmCSub} with an algebro-geometric condition.  In the best case scenario this could give a checkable criterion equivalent to the existence of solutions to dHYM. 

Let us briefly discuss the moment map picture, though our primary interest will be the variational framework of the associated GIT problem.  Let $L\rightarrow (X,\omega)$ be a holomorphic line bundle, and fix a hermitian metric $h$ on $L$, which induces a unitary structure.  We consider the affine space $\mathcal{A}^{1,1}$ of $h$-unitary connections inducing integrable complex structures on $L$.  Since $L$ is a line bundle, this is just the same as the set of unitary connections $\nabla = d+A$ such that $\dbar A^{0,1}=0$.  The group $\mathcal{G}$ of gauge transformations acts on $\mathcal{A}^{1,1}$ by in the standard way.  If $g=e^{\phi}$ is a gauge transformation for some $\phi:X\rightarrow \mathbb{C}$, then
\[
A^{0,1} \mapsto A^{0,1}+ \dbar\phi \qquad A^{1,0} \mapsto A^{1,0}-\del\overline{\phi}.
\]
We can identify $T_{A}\mathcal{A}^{1,1} = {\rm Ker}\{ \dbar: \bigwedge^{0,1}\rightarrow \bigwedge^{0,2}\}$, and hence define a hermitian inner product by
\[
T_{A}\mathcal{A}^{1,1} \ni a,b \longmapsto \langle a,b\rangle_{A} :=-\sqrt{-1} \int_{X}a\wedge \bar{b} \wedge {\rm Re}\left(e^{-\sqrt{-1}\hat{\theta}}(\omega -F(A))^{n-1}\right).
\]
In general, this inner product is degenerate, but one can check that it is non-degenerate in an open neighbourhood of a solution to the dHYM equation.  The natural complex structure on $T_{A}\mathcal{A}^{1,1}$ acts by $a \mapsto \sqrt{-1}a$, and we get a symplectic form on $\mathcal{A}^{1,1}$ by taking
\[
 \Omega_{A}(a,b) := {\rm Im}\left(\langle a,b\rangle_{A}\right). 
 \]
 Let $\mathcal{G}_U$ be the Lie group of unitary gauge transformations of $(L,h)$, and identify $Lie(\mathcal{G}_{U}) = C^{\infty}(X,\sqrt{-1}\mathbb{R})$.  Let $\sqrt{-1}\phi \in Lie(\mathcal{G}_U)$.  Identify $Lie(\mathcal{G}_U)^{*}$ with the space of imaginary $2n$-forms on $X$ by the non-degenerate pairing
 \[
( \sqrt{-1}\phi, \sqrt{-1}\beta) \longmapsto \int_{X}\phi \beta.
 \]
 The function $\sqrt{-1}\phi$ generates the vector field $\sqrt{-1}\,\dbar{\phi}$ on $\mathcal{A}$.  For any $b\in T_{A}\mathcal{A}^{1,1}$ we consider
\[
\begin{aligned}
\langle \sqrt{-1}\,\dbar{\phi},b\rangle_{A} &= \int_{X}\dbar\phi \wedge \bar{b} \wedge {\rm Re}\left(e^{-\sqrt{-1}\hat{\theta}}(\omega -F(A))^{n-1}\right)\\
&=-\int_{X}\phi \cdot  \overline{\del b} \wedge {\rm Re}\left(e^{-\sqrt{-1}\hat{\theta}}(\omega -F(A))^{n-1}\right)
\end{aligned}
\]
and so
\[
 \Omega_{A}(\sqrt{-1}\,\dbar{\phi},b) = \frac{\sqrt{-1}}{2}\int_{X}\phi \cdot (\overline{\del b} - \del b)\wedge{\rm Re}\left(e^{-\sqrt{-1}\hat{\theta}}(\omega -F(A))^{n-1}\right).
 \]
Now, if we consider $A^{0,1} \mapsto A^{0,1}+tb$, then
\[
\frac{d}{dt}\bigg|_{t=0}F(A+tb) = -(\overline{\del b} - \del b)
\]
and this form is purely imaginary.  Therefore
\[
\begin{aligned}
&\frac{d}{dt}\bigg|_{t=0}{\rm Im}\left(e^{-\sqrt{-1}\hat{\theta}}\left(\omega-F(A+tb)\right)^n\right)\\
&= n{\rm Im}\left(e^{-\sqrt{-1}\hat{\theta}}\left(\omega-F(A)\right)^{n-1}\wedge(\overline{\del b} - \del b) \right)\\
&= n(\overline{\del b} - \del b)\wedge{\rm Re}\left(e^{-\sqrt{-1}\hat{\theta}}\left(\omega-F(A)\right)^{n-1}\right).
\end{aligned}
\]
It follows immediately that the moment map for the $\mathcal{G}_U$ action is
\[
A \mapsto -\frac{\sqrt{-1}}{2n}{\rm Im}\left(e^{-\sqrt{-1}\hat{\theta}}(\omega- F(A))^n\right),
\]
and hence solutions of dHYM correspond exactly to zeroes of the moment map.  Despite the fact the the symplectic form is degenerate, we can still hope to use ideas from GIT to study the existence of zeroes of the moment map.

Infinite dimensional GIT frameworks have appeared in several contexts in the study of nonlinear PDE on K\"ahler manifolds, including the study of the Hermitian-Yang-Mills equation on a compact K\"ahler manifold \cite{AtBo}, certain nonlinear generalizations thereof \cite{Conan}, and the K\"ahler-Einstein and constant scalar curvature equation \cite{Do4, Do5}, among others \cite{Do1}.  More recently, the study of the space of K\"ahler metrics has drawn a great deal of attention.  Since our setting is formally analogous to this latter topic, let us briefly recall this framework.  Following Donaldson \cite{Do1}, Mabuchi \cite{Mab} and Semmes \cite{Sem}, we fix a Kahler class $[\omega]$ on $X$, and consider
\[
\mathcal{H}_{PSH} := \{ \phi \in C^{\infty}(X,\mathbb{R}) : \omega_{\phi} := \omega +\ddb \phi >0 \}
\]
The tangent space to $\mathcal{H}_{PSH}$ at a function $\phi$ is $C^{\infty}(X,\mathbb{R})$, and we can introduce a Riemannian metric by
\[
\langle \psi_1, \psi_2 \rangle_{\phi} = \int_{X} \psi_1\psi_2 \omega_{\phi}^{n}.
\]
This makes $\mathcal{H}_{PSH}$ into an infinite dimensional Riemannian manifold.  One can then study the geodesic equation on this manifold, which is equivalent to the Homogeneous complex Monge-Amp\`ere equation \cite{Do1, Mab, Sem}.  For any $\phi_1, \phi_2 \in \mathcal{H}$ one can find a curve $\phi(t)$ of potentials for which $\phi(0) = \phi_0, \phi(1)=\phi_1$, $\phi(t)$ is $C^{1,1}$ in space and time, satisfies $\omega+\ddb \phi(t) \geq 0$, and solves the geodesic equation in a weak sense \cite{Chen} (see also \cite{CNS2, Guan, Bl, CTW, Y} for related work).  Furthermore, it is known that $\phi(t)$ cannot be $C^{2}$ in general \cite{Dar, DarL, LemV}.  Even without better regularity, the existence of weak geodesics plays an important role in linking the existence of solutions to certain nonlinear PDE, including the K\"ahler-Einstein equation on a Fano manifold \cite{Bo, BBJ}, and Donaldson \cite{Do} and Chen's J-equation \cite{Chen04}.

The variational framework studied here is mirror to a variational framework for positive Lagrangians introduced by Solomon \cite{Sol}.  When $L$ is an ample line bundle, this variational structure can be regarded as interpolating between the Riemannian structure for the Hermitian-Yang-Mills equation, and the Donaldson-Mabuchi-Semmes \cite{Do1, Mab, Sem} Riemannian structure on the space of K\"ahler metrics, as will be discussed below.

\begin{defn}
Define the space
\[
\mathcal{H} := \{ \phi \in C^{\infty}(X,\mathbb{R}) : {\rm Re}(e^{-\sqrt{-1}\hat{\theta}}(\omega + \sqrt{-1}\alpha_{\phi})^n)>0 \}
\]
\end{defn}

A slightly more concrete definition of the space $\mathcal{H}$ in terms of the Lagrangian phase operator~\eqref{eq: LagPhaseOp} is

\begin{lem}
The space $\mathcal{H}$ can be defined as
\[
\mathcal{H} := \{ \phi \in C^{\infty}(X,\mathbb{R}) : \Theta_{\omega}(\alpha_{\phi}) \in (\hat{\phi}-\frac{\pi}{2}, \hat{\phi}+\frac{\pi}{2}) \mod 2\pi\}.
\]
\end{lem}

\begin{rk}
It is clear that if one changes Kahler forms within the class $[\omega]$, then the space $\mathcal{H}$ will change as well.  Thus we should really be writing $\mathcal{H}_{\omega}$ to indicate the dependence on $\omega$, but we will refrain from doing so and hope this causes no confusion.
\end{rk}

Recall that the angle $\hat{\theta}$ is a priori only defined modulo $2\pi$.

\begin{lem}\label{lem: angLift}
Assume that $\mathcal{H} \ne \emptyset$.  Then there exists a unique lift of $\hat{\theta}$ to $\mathbb{R}$.
\end{lem}
This lemma is an application of the maximum principle;  we refer the reader to \cite{JY, CXY} for a proof.

It is a simple consequence of the Schur-Horn Theorem \cite{H1} and Lemma~\ref{lem: AngleBasicProps} (7) that the space $\mathcal{H}$ is convex when $|\hat{\theta}| \geq (n-1)\frac{\pi}{2}$.  However, in the lower branches it is not even clear that $\mathcal{H}$ is connected.  For our purposes this will not be a significant issue, since $\mathcal{H}$ is embedded in the vector space $C^{\infty}(X,\mathbb{R})$.

There is a natural Riemannian structure on $\mathcal{H}$ defined in the following way.  The tangent space at a point $\phi \in \mathcal{H}$ is $T_{\phi}\mathcal{H} = C^{\infty}(X,\mathbb{R})$ and we define a non-trivial Riemannian structure by
\[
\langle \psi_1, \psi_2 \rangle_{\phi} = \int_{X} \psi_1 \psi_2 {\rm Re}\left(e^{-\sqrt{-1}\hat{\theta}}(\omega+\sqrt{-1}\alpha_{\phi})^{n}\right).
\]
The Riemannian metric gives rise to a notion of geodesics.
\begin{prop}
A smooth curve $\phi(t) \in \mathcal{H}$ with $\phi(0)=\phi_0, \phi(1)=\phi_1$ is a geodesic if it solves the equation
\begin{equation}\label{eq: GeoEq1}
\ddot{\phi}{\rm Re}\left(e^{-\sqrt{-1}\hat{\theta}}(\omega+ \sqrt{-1}\alpha_{\phi})^{n}\right)+ n\sqrt{-1}\del\dot{\phi}\wedge\dbar\dot{\phi} \wedge{\rm Im}\left(e^{-\sqrt{-1}\hat{\theta}}(\omega+\sqrt{-1}\alpha_{\phi})^{n-1}\right)=0
\end{equation}
\end{prop}
\begin{proof}
Let $\phi(t)$ be a curve in $\mathcal{H}$ with constant speed.  Suppose that $\phi(t,s)$ is surface in $\mathcal{H}$ such that $\phi(0,s) = \phi_0$, $\phi(1,s)= \phi_1$, and $\phi(t,0)= \phi(t)$.  We may write $\phi(t,s) = \phi(t) + s\psi(t,s)$ with $\psi(0,s) = \psi(1,s)=0$.  We will compute the variation of arc-length around $\phi(t)$.  The length is
\[
L(s,t) = \int_{0}^{1}dt\sqrt{\int_{X} (\dot{\phi})^2{\rm Re}\left(e^{-\sqrt{-1}\hat{\theta}}(\omega+\sqrt{-1}\alpha_{\phi})^{n}\right)}
\]
Taking a derivative gives
\[
\begin{aligned}
\frac{d}{ds}\bigg|_{s=0}L &=  \int_{0}^{1}dt\frac{1}{\|\dot\phi\|}\int_{X} \dot{\psi}\dot{\phi}{\rm Re}\left(e^{-\sqrt{-1}\hat{\theta}}(\omega+\sqrt{-1}\alpha_{\phi})^{n}\right)\\
&\quad + \int_{0}^{1}\frac{1}{2\|\dot\phi\|}\int_{X} (\dot{\phi})^2{\rm Re}\left(n\sqrt{-1}e^{-\sqrt{-1}\hat{\theta}}(\omega+\sqrt{-1}\alpha_{\phi})^{n-1}\wedge \ddb \psi \right)
\end{aligned}
\]
Integration by parts on the second term yields
\begin{equation}\label{eq: ArcLengthVar1}
\begin{aligned}
\frac{d}{ds}\bigg|_{s=0}L &=  \int_{0}^{1}dt\frac{1}{\|\dot\phi\|}\int_{X} \dot{\psi}\dot{\phi}{\rm Re}\left(e^{-\sqrt{-1}\hat{\theta}}(\omega+\sqrt{-1}\alpha_{\phi})^{n}\right)\\
&\quad - \int_{0}^{1}dt\frac{n}{\|\dot\phi\|}\int_{X} \psi \sqrt{-1}\del\dot{\phi} \wedge \dbar\dot{\phi} \wedge {\rm Im}\left(e^{-\sqrt{-1}\hat{\theta}}(\omega+\sqrt{-1}\alpha_{\phi})^{n-1} \right)\\
&\quad-\int_{0}^{1}dt\frac{n}{\|\dot\phi\|}\int_{X} \psi \dot{\phi}  {\rm Im}\left(e^{-\sqrt{-1}\hat{\theta}}(\omega+\sqrt{-1}\alpha_{\phi})^{n-1}\wedge \ddb \dot{\phi} \right)
\end{aligned}
\end{equation}
where we also used that $\ddb$ is a real operator, and ${\rm Re}(\sqrt{-1}z) =-{\rm Im}(z)$.  We now integrate by parts in time on the first term.
\begin{equation}\label{eq: ArcLengthVar2}
\begin{aligned}
&\int_{0}^{1}dt\frac{1}{\|\dot\phi\|}\int_{X} \dot{\psi}\dot{\phi}{\rm Re}\left(e^{-\sqrt{-1}\hat{\theta}}(\omega+\sqrt{-1}\alpha_{\phi})^{n}\right)\\
&= \int_{0}^{1} dt \frac{d}{dt}\left(\frac{1}{\|\dot{\phi}\|}\int_{X} \psi\dot{\phi}{\rm Re}\left(e^{-\sqrt{-1}\hat{\theta}}(\omega+\sqrt{-1}\alpha_{\phi})^{n}\right)\right)\\
&\quad-\int_{0}^{1}dt\frac{1}{\|\dot\phi\|}\int_{X} \psi\ddot{\phi}{\rm Re}\left(e^{-\sqrt{-1}\hat{\theta}}(\omega+\sqrt{-1}\alpha_{\phi})^{n}\right)\\
&\quad -\int_{0}^{1}dt\frac{1}{\|\dot\phi\|}\int_{X} \psi\dot{\phi}{\rm Re}\left(n\sqrt{-1}e^{-\sqrt{-1}\hat{\theta}}(\omega+\sqrt{-1}\alpha_{\phi})^{n-1}\wedge \ddb\dot{\phi}\right)
\end{aligned}
\end{equation}
The last term from~\eqref{eq: ArcLengthVar2} cancels the last term from~\eqref{eq: ArcLengthVar1}.  Furthermore, since $\psi(0,s) = \psi(1,s)=0$ the first integral vanishes.  Thus
\[
\begin{aligned}
\frac{d}{ds}\bigg|_{s=0}L &= -\int_{0}^{1}dt\frac{1}{\|\dot\phi\|}\int_{X} \psi\ddot{\phi}{\rm Re}\left(e^{-\sqrt{-1}\hat{\theta}}(\omega+\sqrt{-1}\alpha_{\phi})^{n}\right)\\
&\quad -\int_{0}^{1}dt\frac{n}{\|\dot\phi\|}\int_{X} \psi \sqrt{-1}\del\dot{\phi} \wedge \dbar\dot{\phi} \wedge {\rm Im}\left(e^{-\sqrt{-1}\hat{\theta}}(\omega+\sqrt{-1}\alpha_{\phi})^{n-1} \right)
\end{aligned}
\]
which is what we wanted to prove.
\end{proof}
As in the case of geodesics in the space of K\"ahler metrics, we can reformulate this equation as a degenerate elliptic equation over the product manifold
\[
\mathcal{X} := X \times \mathcal{A} = X\times \{e^{-1} < |t|<1\}
\]
where now $t$ is a coordinate on $\mathbb{C}$.  We denote the projections for $\mathcal{X}$ to $X, \mathcal{A}$ by $\pi_{X}, \pi_{\mathcal{A}}$ respectively, and let $\DDb$ denote the $\del\dbar$ operator on the $n+1$ dimensional manifold $\mathcal{X}$, and $\ddb$ denote the $\del\dbar$ operator on $X$.

\begin{lem}\label{lem: geoEqAnn}
Suppose $\phi_{0}, \phi_1 \in \mathcal{H}$.  A solution $\hat{\phi}(x,s) \in \mathcal{H}$ of the geodesic equation with $\phi(x,0)=\phi_0, \phi(x,1) = \phi_1$ is equivalent to a function $\phi : \mathcal{X} \rightarrow \mathbb{R}$ which is $S^1$ invariant (ie. $\phi(x,t) = \phi(x,|t|)$) and solving
\begin{equation}\label{eq: geoAnn}
{\rm Im}\left[e^{-\sqrt{-1}\hat{\theta}}\left(\pi_{X}^{*}\omega + \sqrt{-1}\left(\pi_{X}^{*}\alpha + \DDb \phi\right)\right)^{n+1}\right]=0
\end{equation}
\[
{\rm Re}\left[e^{-\sqrt{-1}\hat{\theta}}\left(\omega + \sqrt{-1}\left(\alpha + \ddb \phi\right)\right)^{n}\right] >0
\]
and with $\phi(x,1) = \phi_0(x), \phi(x,e^{-1}) = \phi_1$.
\end{lem}
\begin{proof} 
Let $s = -\log|t|$, and set $\phi(x,t) = \hat{\phi}(x,-\log|t|)$, and let $\pi = \pi_{X}$ for simplicity.  The second line just expresses that $\phi(x,t) = \hat{\phi}(x,s) \in \mathcal{H}$, so it suffices to check~\eqref{eq: geoAnn} is equivalent to~\eqref{eq: GeoEq1}.  For simplicity denote $\dot{\hat{\phi}} = \del_s \hat{\phi}$, and similarly for higher derivatives.  Then
\[
\del_{t}\phi = -\frac{1}{t}\dot{\hat{\phi}}, \quad \del_{\bar{t}}\phi = -\frac{1}{\bar{t}}\dot{\hat{\phi}}, \quad \del_{t}\del_{\bar{t}}\phi = \frac{1}{|t|^2} \ddot{\hat{\phi}}
\]
Now it is a matter of linear algebra.  We expand
\[
\begin{aligned}
&\left(\pi^{*}\omega + \sqrt{-1}\left(\pi^{*}\alpha + \DDb \phi\right)\right)^{n+1}\\
& = \sqrt{-1}(n+1)\left(\pi^{*}\omega + \sqrt{-1}(\pi^{*}(\alpha_{\hat{\phi}})\right)^{n}\wedge\sqrt{-1}\del_t\dbar_{t}\phi \\
&\quad -\frac{n(n+1)}{2}(\pi^{*}\omega + \sqrt{-1}\pi^{*}(\alpha_{\hat{\phi}}))^{n-1}\wedge(\del_t\del_{\bar{t}}\phi + \del_{X}\del_{\bar{t}}\phi + \del_{t}\dbar_{X}\phi)^{2}\\
\end{aligned}
\]
The third line becomes
\[
\begin{aligned}
&-n(n+1)\del_{X}\del_{\bar{t}}\phi \wedge \del_{t}\dbar_{X}\phi\wedge(\pi^{*}\omega + \sqrt{-1}\pi^{*}(\alpha_{\hat{\phi}}))^{n-1}\\
&= n(n+1)\del_{\bar{t}}\del_{X}\phi \wedge \del_{t}\dbar_{X}\phi\wedge (\pi^{*}\omega + \sqrt{-1}\pi^{*}(\alpha_{\hat{\phi}}))^{n-1}
\end{aligned}
\]
After some straightforward algebra we get
\[
\frac{(n+1)\sqrt{-1}dt\wedge d\bar{t}}{|t|^2}\wedge\left(\ddot{\hat{\phi}}\sqrt{-1}\pi^{*}(\omega + \sqrt{-1}\alpha_{\hat{\phi}})^{n} + n\sqrt{-1}\del\dot{\hat{\phi}} \wedge\dbar \dot{\hat{\phi}}\wedge (\pi^{*}(\omega + \sqrt{-1}\alpha_{\hat{\phi}}))^{n-1}\right)
\]
Multiplying by $e^{-\sqrt{-1}\hat{\theta}}$ and taking the imaginary part yields equivalence between the geodesic equation and~\eqref{eq: geoAnn}.
\end{proof}

Equation~\eqref{eq: geoAnn} is a degenerate equation, and hence the existence and regularity of solutions is not guaranteed.  In fact, following work of Rubinstein-Solomon~\cite{RuSol}, Jacob~\cite{JacPr} showed that~\eqref{eq: geoAnn} is a degenerate elliptic equation which fits the Dirichlet duality theory of Harvey-Lawson~\cite{HL}.  Let us briefly recall how this is done.   On $\mathcal{X}$ define metrics
\begin{equation}\label{eq: epsMetric}
\hat{\omega}_{\epsilon} := \pi_{X}^{*}\omega + \epsilon^{2}\sqrt{-1}dt\wedge d\bar{t}.
\end{equation}
Then,

\begin{defn}\label{def: sldlp}
The (space-time lifted) degenerate Lagrangian phase operator  is defined by
\[
\widetilde{\Theta}_{\omega}(\pi_{X}^{*}\alpha + \DDb \phi ) = \lim_{\epsilon \rightarrow 0} \Theta_{\hat{\omega}_{\epsilon}}(\pi_{X}^{*}\alpha + \DDb \phi ).
\]
\end{defn}

As shown in \cite{RuSol}, the operator~$\widetilde{\Theta}_{\omega}(\cdot)$ defines a degenerate elliptic operator, and the geodesic equation can be rewritten as
\begin{equation}\label{eq: SpaceTimeAngle}
\widetilde{\Theta}_{\omega}(\alpha_{\phi}) = \hat{\theta}
\end{equation}
When $|\hat{\theta}|>(n-1)\frac{\pi}{2}$, Jacob \cite{JacPr} proves existence of continuous geodesics building on work of Rubinstein-Solomon \cite{RuSol}.  In this context, a continuous geodesic is a viscosity solution of~\eqref{eq: SpaceTimeAngle}, in the sense of Harvey-Lawson \cite{HL}.

Our approach here is different, necessitated by the need for geodesics with with better regularity for geometric applications.  We will therefore obtain the existence of regular geodesics as limits of smooth solutions to a regularized version of~\eqref{eq: geoAnn} which is elliptic.

\begin{defn}
Suppose $\phi_{0}, \phi_1 \in \mathcal{H}$. An $S^1$ invariant function $\phi : \mathcal{X} \rightarrow \mathbb{R}$ is said to be an $\epsilon$-regularized geodesic in the space $\mathcal{H}$ joining $\phi_0$ and $\phi_1$ if $\phi$ solves
\begin{equation}\label{eq: eReg_geoAnn}
{\rm Im}\left[e^{-\sqrt{-1}\hat{\theta}}\left(\pi^{*}\omega + \epsilon^2\sqrt{-1}dt\wedge d\bar{t} + \sqrt{-1}\left(\pi^{*}\alpha + \DDb \phi\right)\right)^{n+1}\right]=0
\end{equation}
\[
{\rm Re}\left[e^{-\sqrt{-1}\hat{\theta}}\left(\omega + \sqrt{-1}\left(\alpha + \ddb \phi\right)\right)^{n}\right] >0
\]
and with $\phi(x,1) = \phi_0(x), \phi(x,e^{-1}) = \phi_1$.  Here, as before, $\DDb$ denotes the $\ddb$-operator on the $n+1$ dimensional manifold $\mathcal{X}$, while $\ddb$ denotes the operator on $X$
\end{defn}
\begin{rk}
To streamline some statements, we will view geodesics as $\epsilon$-regularized geodesics with $\epsilon=0$.  We will also refer to $\epsilon$-regularized geodesics as $\epsilon$-geodesics.
\end{rk}

Note that~\eqref{eq: eReg_geoAnn} is just the deformed Hermitian-Yang-Mills equation on the manifold with boundary, defined with respect to the degenerating metric $\hat{\omega}_{\epsilon}$.

\begin{lem}
If $\phi(x,t)$ is an $S^{1}$ invariant solution of the deformed Hermitian-Yang-Mills equation on $(\mathcal{X}, \hat{\omega}_{\epsilon})$ with $\hat{\theta} >(n-1)\frac{\pi}{2}$, then $\phi(x,t)$ is an $\epsilon$-geodesic.
\end{lem}
\begin{proof}
This is almost a tautology, except to show that $\phi(x,t)\in \mathcal{H}$ for all $t\in \mathcal{A}$.  That is, we need to show that
\[
n\frac{\pi}{2} > \Theta_{\omega}(\alpha+\ddb \phi(x,t)) > \hat{\theta}-\frac{\pi}{2}
\]
for all $t$.  The upper bound is trivial while the lower bound follows from the Schur-Horn Theorem \cite{H1} together with Lemma~\ref{lem: AngleBasicProps} (7).  
\end{proof}

For our purposes, the most important point of this Riemannian structure is that there exist real valued functions on $\mathcal{X}$ which are convex/concave/linear along geodesics, and which can be used to study the existence problem for the dHYM equation on $X$.

\begin{defn}\label{defn: AY}
The complexified Calabi-Yau functional is defined by its differential on $\mathcal{H}$ as follows.  Suppose $\phi \in \mathcal{H}$, and $\psi \in T_{\phi}\mathcal{H}$, then
\[
dCY_{\mathbb{C}}(\phi)(\psi) = \int_{X} \psi(\omega +\sqrt{-1}\alpha_{\phi})^n.
\]
\end{defn}

Recall that on the space of Kahler metrics in the class $[\omega]$ the Calabi-Yau functional is defined by its differential at a K\"ahler potential $\phi$ by
\[
dCY(\phi)(\psi)= \int_{X}\psi \omega_{\phi}^n.
\]

Hence we can view the functional $CY_{\mathbb{C}}$ as the extension of the Calabi-Yau functional to {\em complexified} K\"ahler forms in $H^{1,1}(X,\mathbb{C})$.  The next proposition shows that $CY_{\mathbb{C}}$ integrates to a well-defined function.

\begin{prop}\label{prop: AYwellDefn}
The complexified Calabi-Yau functional integrates to a well-defined functional $CY_{\mathbb{C}}: \mathcal{H} \rightarrow \mathbb{C}$.  If $0 \in \mathcal{H}$, the we can $CY_{\mathbb{C}}$ explicitly as
\[
\frac{1}{n+1}\sum_{j=0}^{n} \int_{X}\phi(\omega+\sqrt{-1}\alpha_{\phi})^j\wedge (\omega+\sqrt{-1}\alpha)^{n-j}.
\]
\end{prop}
\begin{proof}
Since the space $\mathcal{H}$ may not be connected, we will instead show that $CY_{\mathbb{C}}$ integrates to a well-defined functional on $C^{\infty}(X,\mathbb{R})$, which we then restrict to $\mathcal{H}$.  Fix a base point in $\mathcal{H}$, which, by changing the background form $\alpha$ we may always take to be $0 \in \mathcal{H}$.  Let $\phi_1 \in C^{\infty}(X,\mathbb{R})$ be another potential, and suppose that $\phi(t), \hat{\phi}(t)$ are two paths in $C^{\infty}(X,\mathbb{R})$ such that $\phi(0)=\hat{\phi}(0)=0$ and $\phi(1)=\hat{\phi}(1) = \phi_1$.  Let $\psi(t) = \hat{\phi}(t)- \phi(t)$, and consider $\phi(t)+s\psi(t)$ for $s\in[0,1]$.  We compute
\begin{equation}\label{eq: AYwellDef}
\begin{aligned}
&\frac{d}{ds}\int_{0}^{1}dt \int_{X}(\dot{\phi} + s\dot{\psi}){\rm Im}\left(e^{-\sqrt{-1}\hat{\theta}}(\omega+\sqrt{-1}\alpha_{\phi+s\psi})^n\right)\\
&= \int_{0}^{1}dt \int_{X}\dot{\psi}{\rm Im}\left(e^{-\sqrt{-1}\hat{\theta}}(\omega+\sqrt{-1}\alpha_{\phi+s\psi})^n\right)\\
&\quad + \int_{0}^{1}dt \int_{X}(\dot{\phi} + s\dot{\psi}){\rm Im}\left(n\sqrt{-1}e^{-\sqrt{-1}\hat{\theta}}(\omega+\sqrt{-1}\alpha_{\phi+s\psi})^{n-1} \wedge \ddb\psi\right)
\end{aligned}
\end{equation}
Since $\psi(0)=\psi(1)=0$, we can integrate by parts in $t$ in the first term on the right hand side of~\eqref{eq: AYwellDef}.
\[
\begin{aligned}
&\int_{0}^{1}dt \int_{X}\dot{\psi}{\rm Im}\left(e^{-\sqrt{-1}\hat{\theta}}(\omega+\sqrt{-1}\alpha_{\phi+s\psi})^n\right)\\
&= - \int_{0}^{1}dt \int_{X}\psi{\rm Im}\left(n\sqrt{-1}e^{-\sqrt{-1}\hat{\theta}}(\omega+\sqrt{-1}\alpha_{\phi+s\psi})^{n-1}\wedge \ddb(\dot{\phi}+s\dot{\psi})\right)
\end{aligned}
\]
Finally, since $\ddb$ is a real operator, integration by parts cancels the second term on the right in~\eqref{eq: AYwellDef}.  To obtain the closed form take $\phi \in \mathcal{H}$, and consider the path $t\phi$.  Then we have
\[
CY_{\mathbb{C}}(\phi)(\psi) = \int_{0}^{1} dt \int_{X} \phi(\omega +\sqrt{-1}\alpha_{t\phi})^n
\]
Writing $\omega+\sqrt{-1}\alpha_{t\phi} = t(\omega+\sqrt{-1}\alpha_{\phi}) + (1-t)(\omega+\sqrt{-1}\alpha)$ we get
\[
CY_{\mathbb{C}}(\phi)(\psi) = \sum_{j=0}^{n}\int_{0}^{1}\binom{n}{j}t^{j}(1-t)^{n-j}  dt \int_{X}\phi(\omega +\sqrt{-1}\alpha_{\phi})^j\wedge (\omega+\sqrt{-1}\alpha)^{n-j}
\]
and the $t$ integral is easily evaluated by induction to be $\frac{1}{n+1}$ independent of $j$.
\end{proof}
\begin{rk}
The above proof shows that $CY_{\mathbb{C}}$ integrates to a well defined functional on $C^{\infty}(X,\mathbb{R})$, which we can then restrict to $\mathcal{H}$.  This avoids obvious technical difficulties in case $\mathcal{H}$ has more than one connected component.
\end{rk}

Following Solomon \cite{Sol} we can extract two particularly useful real valued functions from $CY_{\mathbb{C}}$.  Define
\[
\begin{aligned}
\mathcal{J}(\phi) &:= -{\rm Im}\left(e^{-\sqrt{-1}\hat{\theta}}CY_{\mathbb{C}}\right)\\
\mathcal{C}(\phi) &:= {\rm Re}\left(e^{-\sqrt{-1}\hat{\theta}}CY_{\mathbb{C}}\right).
\end{aligned}
\]
Clearly the $\mathcal{J}$ functional is precisely the Kempf-Ness functional for our infinite dimensional GIT problem.  We have the following Corollary of Proposition~\ref{prop: AYwellDefn}.

\begin{cor}
Fix a base point $\phi_0 \in \mathcal{H}$ and let $\phi(t)$ be a path in $C^{\infty}(X,\mathbb{R})$ connecting $\phi_0$ to $\phi_1 \in \mathcal{H}$.  Then
\[
\begin{aligned}
\mathcal{J}(\phi_1) &:= -\int_{0}^{1} dt \int_{X}\dot{\phi}{\rm Im}\left(e^{-\sqrt{-1}\hat{\theta}}(\omega+\sqrt{-1}\alpha_{\phi(t)})^n\right)\\
\mathcal{C}(\phi_1) &:= \int_{0}^{1} dt \int_{X}\dot{\phi}{\rm Re}\left(e^{-\sqrt{-1}\hat{\theta}}(\omega+\sqrt{-1}\alpha_{\phi(t)})^n\right)
\end{aligned}
\]
In particular $d\mathcal{J}(\phi) =0$ at a point $\phi \in \mathcal{H}$ if and only if $\phi$ solves the deformed Hermitian-Yang-Mills equation.
\end{cor}

The next proposition makes the connection with infinite dimensional GIT.

\begin{prop}\label{prop: convexityOfFuncs}
Let $\phi(x,s)$ be a curve in $\mathcal{H}$, viewed as an $S^{1}$ invariant function on $\mathcal{X}$ by $s= -\log(|t|)$.  Then we have the following formula for the second derivatives of $\mathcal{J}$ and $\mathcal{C}$;
\begin{equation}\label{eq: secDerJ}
\sqrt{-1}\del_{t}\del_{\bar{t}}\mathcal{J}(\phi(t)) = (\pi_{\mathcal{A}})_{*}{\rm Re}\left[e^{-\sqrt{-1}\hat{\theta}}\left(\pi_{X}^{*}\omega + \sqrt{-1}\left(\pi_{X}^{*}\alpha+ \DDb\phi\right)\right)^{n+1}\right].
\end{equation}
\begin{equation}\label{eq: secDerC}
\sqrt{-1}\del_{t}\del_{\bar{t}}\mathcal{C}(\phi(t)) = (\pi_{\mathcal{A}})_{*}{\rm Im}\left[e^{-\sqrt{-1}\hat{\theta}}\left(\pi_{X}^{*}\omega + \sqrt{-1}\left(\pi_{X}^{*}\alpha+ \DDb\phi\right)\right)^{n+1}\right].
\end{equation}
Furthermore, the functional $\mathcal{J}$ is strictly convex along non-trivial, smooth $\epsilon$-geodesics for all $\epsilon \geq 0$, and $\mathcal{C}$ is affine along $\epsilon$-geodesics for all $\epsilon \geq 0$.
\end{prop}
\begin{proof}
Suppose that $\phi(s)$ is a smooth $\epsilon$-regularized geodesic in $\mathcal{H}$.  We compute
\[
\begin{aligned}
\frac{d^{2}}{ds^2}CY_{\mathbb{C}}(\phi(s)) &= \int_{X} \left(\frac{d^2}{ds^2}\phi\right)\left(\omega+\sqrt{-1}\alpha_{\phi(s)}\right)^n \\
&\quad +  \int_{X} \left(\frac{d}{ds}\phi\right)n\left((\omega+\sqrt{-1}\alpha_{\phi(s)})^{n-1}\wedge \sqrt{-1}\cdot \ddb\left(\frac{d}{ds}\phi\right)\right)\\
&=\int_{X} \left(\frac{d^2}{ds^2}\phi\right)\left(\omega+\sqrt{-1}\alpha_{\phi(s)}\right)^n \\
& -\sqrt{-1}\int_{X}n\sqrt{-1}\del\left(\frac{d}{ds}\phi\right)\wedge\dbar\left(\frac{d}{ds}\phi\right)\wedge(\omega+\sqrt{-1}\alpha_{\phi(s)})^{n-1}
\end{aligned}
\]
From this expression one can easily check directly using local coordinates that $\mathcal{J}$ is convex along geodesics, and $\mathcal{C}$ is affine.  However, it is useful to rewrite this expression on $\mathcal{X}$.  Let $\pi_X: \mathcal{X} \rightarrow X$ be the projection to $X$, and $\pi_{\mathcal{A}}$ be the projection to annulus.   Recall that $s= -\log(|t|)$ where $t$ is a coordinate on $\mathbb{C}$.  We can rewrite the above expression using the computation in Lemma~\ref{lem: geoEqAnn} as
\begin{equation}\label{eq: secDerAY}
\sqrt{-1}\del_{t}\del_{\bar{t}}CY_{\mathbb{C}}= -\sqrt{-1}(\pi_{A})_{*}\left(\pi_{X}^{*}\omega + \sqrt{-1}\left(\pi_{X}^{*}\alpha+ \DDb\phi\right)\right)^{n+1}
\end{equation}
where $ (\pi_{\mathcal{A}})_{*}$ denotes the push forward along the fibration $\mathcal{X} \rightarrow \mathcal{A}$ (ie. integration along fibers).  Taking the real and imaginary parts of this expression proves~\eqref{eq: secDerJ} and~\eqref{eq: secDerC}.  In order to establish the convexity of $\mathcal{J}$ along $\epsilon$-regularized geodesics we need to evaluate the sign of
\[
\begin{aligned}
&{\rm Re}\left[e^{-\sqrt{-1}\hat{\theta}}\left(\pi_{X}^{*}\omega \sqrt{-1}\left(\pi_{X}^{*}\alpha+ \DDb\phi\right)\right)^{n+1}\right]\\
&= {\rm Re}\left[e^{-\sqrt{-1}\hat{\theta}}\left(\pi_{X}^{*}\omega + \epsilon^{2}\sqrt{-1}dt\wedge d\bar{t} + \sqrt{-1}\left(\pi_{X}^{*}\alpha+ \DDb\phi\right)\right)^{n+1}\right]\\
&\quad- (n+1)\epsilon^{2}\sqrt{-1}dt\wedge d\bar{t} \wedge \pi_{X}^{*}{\rm Re}\left[e^{-\sqrt{-1}\hat{\theta}}\left(\omega +\sqrt{-1}\left(\alpha+ \ddb\phi\right)\right)^{n}\right].
\end{aligned}
\]
In particular, it suffices to show that the fibre integral of the right hand side is positive.  By~\eqref{eq: defnOperators} we have
\[
\begin{aligned}
 &{\rm Re}\left[e^{-\sqrt{-1}\hat{\theta}}\left(\pi_{X}^{*}\omega + \epsilon^{2}\sqrt{-1}dt\wedge d\bar{t} + \sqrt{-1}\left(\pi_{X}^{*}\alpha+ \DDb\phi\right)\right)^{n+1}\right]\\
 & = (n+1)R(\alpha_{\phi}) \epsilon^{2}\sqrt{-1}dt\wedge d\bar{t}\wedge \omega^{n}
 \end{aligned}
 \]
 where we used that $\Theta_{\hat{\omega}_\epsilon}(\alpha_{\phi}) = \hat{\theta}$ since $\phi$ is an $\epsilon$-geodesic.  Here $R(\alpha_{\phi})$ denotes the radius function computed on $(\mathcal{X},\hat{\omega}_{\epsilon})$.  On the other hand, we have
 \[
 {\rm Re}\left[e^{-\sqrt{-1}\hat{\theta}}\left(\omega + \sqrt{-1}\left(\alpha+ \ddb\phi\right)\right)^{n}\right] = r(\phi)\cos\left(\theta_{\omega}(\alpha_{\phi}|_{X}) - \hat{\theta}\right)\omega^{n}
 \]
 where $(\alpha_{\phi})|_{X}$ denotes the restriction of $\alpha_{\phi}$ to $TX \subset T\mathcal{X}$, and $r(\phi)$ is the radius function computed on $(X, \omega)$.  Fix a point $(p_0,t_0) \in \mathcal{X}$, and choose holomorphic normal coordinates $(w_1,\ldots, w_n)$ for $(X,\omega)$ near $p_0$ so that at $(p_0,t_0)$ we have $\omega_{\bar{j}i} = \delta_{\bar{j}i}$, and $(\alpha + \ddb\phi) = \lambda_{i} \delta_{\bar{j}i}$.  Complete this to a set of coordinates on $\mathcal{X}$ by setting $w_0 =t-t_0$.  We can then write
\[
\alpha +\DDb\phi = \left( \begin{array}{c|ccc}
 	a_{00} & \,& \vec{a}_{0}\,^{\dagger} &\, \\\hline
	\, & \lambda_1 &  \cdots&0 \\
	\vec{a}_{0} &  \vdots & \ddots & \vdots\\
	\, & 0 & \cdots & \lambda_n
\end{array}
\right),
\qquad
\hat{g}_{\epsilon}^{-1} = \left( \begin{array}{c|c}
 	\frac{1}{\epsilon^2} &  0 \\\hline
	0 &  1_n
	\end{array}
\right)
\]
where $1_n$ is the $n\times n$ identity matrix.  By definition $r(\phi)^2 = \prod_{i=1}^{n}(1+\lambda_i^2)$, while
\[
1 + \hat{g}_{\epsilon}^{-1}\alpha_{\phi}\hat{g}_{\epsilon}^{-1} \alpha_{\phi} = \left( \begin{array}{c|c}
 	1+\frac{1}{\epsilon^4}a_{00}^2 + \frac{1}{\epsilon^2}|\vec{a}_0|^2 &\frac{a_{00}}{\epsilon^4}  \vec{a}_{0}\,^{\dagger} + \frac{1}{\epsilon^2}\vec{a}_0\,^{\dagger} D \\\hline
	\frac{a_{00}}{\epsilon^2} \vec{a}_{0} + D\vec{a}_0 &1_n+ \frac{1}{\epsilon^2}\vec{a}_0\vec{a}_0\,^{\dagger}+  D^2
	\end{array}
\right)
\]
The function $R(\phi)^2$ is the determinant of this matrix.  Expanding the determinant along the top row we have
\[
R(\phi)^2 =1\cdot \det\left(1_n+ \frac{1}{\epsilon}\vec{a}_0\vec{a}_0\,^{\dagger}+  D^2\right) + \det\left(1_n + \hat{g}_{\epsilon}^{-1}\alpha_{\phi}\hat{g}_{\epsilon}^{-1} \alpha_{\phi}  \right)
\]
where, in the second expression have abusively written
\[
1_n = \left(\begin{array}{c|c} 0 &0\\\hline
	0 &1_n
	\end{array}
	\right).
\]
Clearly $1_n + \hat{g}_{\epsilon}^{-1}\alpha_{\phi}\hat{g}_{\epsilon}^{-1} \alpha_{\phi}$ is a non-negative definite matrix, and
\[
1_n+ \frac{1}{\epsilon}\vec{a}_0\vec{a}_0\,^{\dagger}+  D^2 \geq 1_n+  D^2.
\]
So $R(\phi) \geq \sqrt{\det(1_n+  D^2)} = r(\phi) \geq r(\phi)\cos\left(\Theta_{\omega}(\alpha_{\phi}|_{X}) - \hat{\theta}\right) $.  It follows that
\[
(\pi_{\mathcal{A}})_{*}{\rm Re}\left[e^{-\sqrt{-1}\hat{\theta}}\left(\pi_{X}^{*}\omega+ \sqrt{-1}\left(\pi_{X}^{*}\alpha+ \DDb\phi\right)\right)^{n+1}\right] \geq 0
\]
in the sense of currents, and so $\mathcal{J}$ is convex along $\epsilon$-geodesics.  Furthermore, if we have equality in the above computation at some point $(p_0,t_0)$ then it is easy to see that we must have $a_{00} =0 =  \vec{a}_{0}$, and $ \theta_{\omega}(\alpha_{\phi}|_{X}) - \hat{\theta}=0$.  It follows that if $\epsilon >0$, and $\mathcal{J}$ is not strictly convex along an $\epsilon$-geodesics, then $\alpha_{\phi}(x,t) = \alpha_{\phi}(x,0)$ is the constant $\epsilon$-geodesic emanating from a solution of dHYM.  Finally, when $\epsilon=0$ one can either compute directly, or take a limit as $\epsilon \rightarrow 0$ in the above argument (though this does not give the strict convexity statement).

Next we show that $\mathcal{C}$ is affine along $\epsilon$-geodesics.  We need to show that
\[
\begin{aligned}
0&=(\pi_{\mathcal{A}})_{*}{\rm Im}\left[e^{-\sqrt{-1}\hat{\theta}}\left(\pi_{X}^{*}\omega + \sqrt{-1}\left(\pi_{X}^{*}\alpha+ \DDb\phi\right)\right)^{n+1}\right]\\
&= (\pi_{\mathcal{A}})_{*}{\rm Im}\left[e^{-\sqrt{-1}\hat{\theta}}\left(\pi_{X}^{*}\omega + \epsilon^{2}\sqrt{-1}dt\wedge d\bar{t} + \sqrt{-1}\left(\pi_{X}^{*}\alpha+ \DDb\phi\right)\right)^{n+1}\right]\\
&\quad- (n+1)\epsilon^{2}\sqrt{-1}dt\wedge d\bar{t} \int_{X}{\rm Im}\left[e^{-\sqrt{-1}\hat{\theta}}\left(\omega +\sqrt{-1}\left(\alpha+ \ddb\phi\right)\right)^{n}\right]
\end{aligned}
\]
By the definition of an $\epsilon$-geodesic the term on the second line is zero, and the term on the third line vanishes by the definition of $\hat{\theta}$.
\end{proof}

As a consequence of Proposition~\ref{prop: convexityOfFuncs} we get a whole $S^{1}$ worth of interesting functionals on the space $\mathcal{H}$; namely $e^{\sqrt{-1}\xi}CY_{\mathbb{C}}$.  These functionals are either convex or concave along $\epsilon$-geodesics depending on the choice of $\xi$.  We point out one further functional which will be useful later on.
\begin{defn}\label{defn: centChargFunc}
Suppose that $\hat{\theta}\in ((n-1)\frac{\pi}{2}, n\frac{\pi}{2})$.  We define the $Z$-functional for $[\alpha], [\omega]$ by
\begin{equation}\label{eq: centChargFunc}
Z(\phi) = e^{-\sqrt{-1}n\frac{\pi}{2}}CY_{\mathbb{C}}
\end{equation}
\end{defn}

The variation of $Z$ at $\phi \in \mathcal{H}$ is given by
\[
\begin{aligned}
\delta Z(\phi) &= \int_{X}(\delta \phi){\rm Re}\left(e^{-\sqrt{-1}\frac{n\pi}{2}}(\omega+\sqrt{-1}\alpha_{\phi})^n\right)\\
&\quad  + \sqrt{-1}  \int_{X}(\delta \phi){\rm Im}\left(e^{-\sqrt{-1}\frac{n\pi}{2}}(\omega+\sqrt{-1}\alpha_{\phi})^n\right). 
\end{aligned}
\]
If $\phi \in \mathcal{H}$ then
\begin{equation}\label{eq: ImZVar}
{\rm Im}\left(e^{-\sqrt{-1}\frac{n\pi}{2}}(\omega+\sqrt{-1}\alpha_{\phi})^n\right) = r(\phi)\sin\left(\theta_{\omega}(\alpha_{\phi}) -\frac{n\pi}{2}\right) \omega^{n}
\end{equation}
is a negative measure.  Furthermore, if $\phi$ is a solution of the deformed Hermitian-Yang-Mills equation, then 
\[
{\rm Re}\left(e^{-\sqrt{-1}\frac{n\pi}{2}}(\omega+\sqrt{-1}\alpha_{\phi})^n\right)
\]
is a positive measure. Writing $Z$ in terms of $\mathcal{C}, \mathcal{J}$ and applying Proposition~\ref{prop: convexityOfFuncs} gives

\begin{cor}\label{cor: Zconcave}
The functionals ${\rm Re}(Z), {\rm Im}(Z)$ are concave along smooth $\epsilon$-geodesics.
\end{cor}

From now on we will restrict to the ``hypercritical phase" case, 
\[
\hat{\theta}\in((n-1)\frac{\pi}{2}, n\frac{\pi}{2}).
\]
This is used crucially in the analysis in Sections~\ref{sec: AnalPrelim}-\ref{sec: AlgObstr}.  In Section~\ref{sec: StabCond} we will comment briefly on the new difficulties and phenomena in the case of lower phase.  We remark that all our results work just as well under the assumption that $\hat{\theta} \in (-n\frac{\pi}{2}, -(n-1)\frac{\pi}{2})$.

As in classical GIT, the way to link existence of solutions of the dHYM equation with algebraic geometry is via the function $\mathcal{J}$, which can be regarded as a Kempf-Ness type functional.  If there is a solution $\phi_0$ of the deformed Hermitian-Yang-Mills equation in $\mathcal{H}$, then for every infinite length (smooth) $(\epsilon)$-geodesic $\phi(s)$ emanating from $\phi_0$ we must have
\[
0 < \lim_{s\rightarrow \infty} \frac{d}{ds}\mathcal{J}(\phi(s)).
\]
In special cases we can evaluate the right hand side as an algebraic invariant, and this gives rise to algebraic obstructions for the existence of solutions to dHYM.  The two main difficulties in executing this approach are the lack of smooth geodesics, and the evaluation of the limit slope in terms of algebraic data.  In this paper we will essentially completely resolve the first issue when $\hat{\theta} \in ((n-1)\frac{\pi}{2}, n\frac{\pi}{2})$, and we will evaluate the slope of the functional at infinity in a rather general setting.  Furthermore, we will explain how the resulting invariants can be interpreted as Bridgeland stability type obstructions.  In dimension $3$ this will give a relatively complete picture relating the existence of solutions to dHYM and Bridgeland type stability conditions.

Before explaining the plan of attack, let us make a few formal remarks about the Riemannian structure and functionals considered here.  Suppose that $[\alpha]$ is a K\"ahler class, and rescale $\omega \mapsto t \omega$ for $t>0$.  Then we get a family of infinite dimensional Riemannian manifolds $(\mathcal{H}_{t}, g_{t})$, and functionals $\mathcal{J}_t, \mathcal{C}_t$.  It is not hard to show that, as $t\rightarrow 0$, we have $\hat{\theta}(t) \rightarrow n\frac{\pi}{2}$, and so in the ``small radius limit" we have
\[
\langle \psi_1,\psi_2 \rangle_{t} \approx \int_{X} \psi_1\psi_2 \alpha_{\phi}^n + O(t)
\]
which is precisely the Donaldson-Mabuchi-Semmes Riemmanian structure on the space of K\"ahler metrics in the class $[\alpha]$.  Similarly we have
\[
d(CY_{\mathbb{C}}(t))(\phi) = (\sqrt{-1})^{n}\int_{X}(\delta\phi)\alpha_{\phi}^n + t(\sqrt{-1})^{n-1}\int_{X}(\delta \phi) n\alpha_{\phi}^{n-1}\wedge \omega + O(t^2)
\]
and so $\mathcal{C}_t$ approaches the classical Calabi-Yau functional, while $\mathcal{J}_t$ approaches the $J$ functional of Donaldson \cite{Do} and Chen \cite{Chen04}.

On the other hand, in the ``large radius limit", as $t\rightarrow +\infty$ we have that $\hat{\theta}\rightarrow 0$, and so
\[
\frac{1}{t^{n}}\langle \psi_1,\psi_2 \rangle_{t} \approx \int_{X} \psi_1\psi_2 \omega^n + O(\frac{1}{t})
\]
and so the Riemannian structure converges to the flat metric, while
\[
\frac{1}{t^n}d(CY_{\mathbb{C}}(t))(\phi) = \int_{X}(\delta\phi)\alpha_{\phi}^n + \frac{1}{t}\sqrt{-1}\int_{X}(\delta \phi) n\omega^{n-1}\wedge\alpha_{\phi} + O(\frac{1}{t^2}).
\]
If $[\alpha]= c_{1}(L)$, the large radius limit yields the Riemannian metric on the space of hermitian metrics on $L$, and $\mathcal{J}_{t}$ converges to the Donaldson functional, which is the Kempf-Ness functional for the infinite dimensional GIT framework related to the Hermitian-Yang-Mills equation (albeit on a line bundle).

The next four sections of this paper will be devoted to proving the existence of smooth $\epsilon$-geodesics, and $C^{1,\alpha}$ geodesic segments in the space $\mathcal{H}$.  Our plan of attack is the following; rescale $\mathbb{C}$ by  $t \mapsto \epsilon t$.  The $\epsilon$-geodesic equation becomes
\begin{equation}\label{eq: thinGeoEq}
{\rm Im}\left[e^{-\sqrt{-1}\hat{\theta}}\left(\pi_{X}^{*}\omega + \sqrt{-1}dt\wedge d\bar{t} + \sqrt{-1}\left(\alpha+ \DDb\phi\right)\right)^{n}\right]=0
\end{equation}
\begin{gather}
 \quad \text{ on } \mathcal{X}_{\epsilon} = X \times \mathcal{A}_{\epsilon} := X\times \{\epsilon e^{-1}\leq |t| \leq \epsilon \}\\
 \text{ with }   \phi_0=\phi\bigg|_{|t|=\epsilon}, \qquad \phi_1=\phi\bigg|_{|t|=e^{-1}\epsilon} 
\end{gather}
where $\phi_0, \phi_1 \in \mathcal{H}$.  In particular, rather than work on a fixed manifold with a degenerating metric, we choose to work with a non-degenerate metric at the expense of working on a very ``thin" manifold.  In order solve the geodesic equation we need to pass to the limit as $\epsilon \rightarrow 0$.  This will be possible if we can prove that the solution of ~\eqref{eq: thinGeoEq} satisfies estimates of the form
\[
\begin{aligned}
&|\nabla^{X}\phi|_{\hat{\omega}} \leq C,\qquad& |\nabla_{t}\phi|_{\hat{\omega}} \leq \frac{C}{\epsilon}\\
|\nabla^{X}\overline{\nabla^{X}}\phi|_{\hat{\omega}} &\leq C,\qquad |\nabla_{t}\overline{\nabla^{X}}\phi|_{\hat{\omega}} &\leq \frac{C}{\epsilon}, \quad |\nabla_{t}\nabla_{\bar{t}} \phi| &\leq \frac{C}{\epsilon^2}
\end{aligned}
\]
where $\nabla^{X}, \overline{\nabla^{X}}$ denote the covariant derivative along the fibers of $\mathcal{X}_{\epsilon} \rightarrow \mathcal{A}_{\epsilon}$, and $C$ is a uniform constant independent of $\epsilon$.  If this is possible, then the rescaled solutions $\tilde{\phi} = \phi(x, \epsilon t)$ will solve the $\epsilon$-geodesic equation on $(\mathcal{X}, \hat{\omega}_{\epsilon})$ and be uniformly bounded with respect to the non-degenerate metric $(\mathcal{X}, \hat{\omega})$.  We can then pass to the limit to obtain weak solutions of the geodesic equation with $C^{1,\alpha}$ regularity.

Before beginning the proof we need a few easy lemmas regarding the geometry of the manifolds $(\mathcal{X}_{\epsilon},  \hat{\omega})$.  Throughout the paper we will use the following terminology.
\begin{defn}
A set of space-time adapted coordinates for $(\mathcal{X}_{\epsilon}, \hat{\omega})$, $\alpha_{\phi}$ centered at $(p_0,t_0)$ is the following.
\begin{itemize}
\item  A set of holomorphic normal coordinates $(w_1,\ldots, w_n)$ for $(X, \omega)$ centered at $p_0$ making $\omega_{\bar{j}i} = \delta_{\bar{j}i}$, and $(\alpha_{\phi})_{\bar{j}i} = \lambda_i \delta_{\bar{j}j}$.
\item The coordinate $w_0 = t- t_0$
\end{itemize}
\end{defn}
Note that space-time adapted coordinates are, in particular, holomorphic normal coordinates for $(\mathcal{X}_{\epsilon}, \hat{\omega})$.

\begin{lem}\label{lem: productMetric}
The manifold $(\mathcal{X}_{\epsilon}, \hat{\omega})$ satisfies the following properties
\begin{enumerate}
\item The Riemann curvature tensor satisfies $R(\del_t, \cdot, \cdot, \cdot) =0$.  In particular, in space-time adapted coordinates we have
\[
R_{\bar{j}i\bar{k}p} = 0
\]
whenever one of $i,j,k,p = 0$.
\item The vector fields $\del_t, \del_{\bar{t}}$ are parallel.
\end{enumerate}
\end{lem}
\begin{proof}
The proof is trivial.  Pick $(p_0,t_0)$ and choose space-time adapted coordinates $(w_0,\ldots,w_n)$ on an open ball $B$.  Both statements follow from the fact that $g_{\bar{0}0}\equiv 1$ on $B$, $g_{\bar{0}j} \equiv 0$ if $j \ne 0$.
\end{proof}

\section{Analytic Preliminaries}\label{sec: AnalPrelim}

Following the discussion in the previous section, the construction of geodesic segments in the space $\mathcal{H}$ will rely on solving the deformed Hermitian-Yang-Mills equation on thin manifolds with boundary. Let $M$ be a $(n+1)\times (n+1)$ Hermitian matrix. In order to simplify the notation, and to avoid confusion with the Lagrangian phase operator on $X$, we will denote
\[
F(M) = \sum_{i=0}^{n} \arctan(\mu_i)
\]
where $\mu_i$ are the eigenvalues of $M$.  On $\mathcal{X}_{\epsilon}$ we will write $\alpha_{\phi} = \alpha + \DDb \phi$.   Over the next four sections we will prove a priori estimates for functions $\phi: \mathcal{X}_{\epsilon} \rightarrow \mathbb{R}$ such that
\begin{equation}\label{eq: mainEquation}
\begin{aligned}
F\left((\hat{\omega})^{-1}(\alpha_{\phi})\right) &= h(x,|t|)\\
 \text{ with }   \phi_0=\phi\bigg|_{|t|=\epsilon}, &\qquad \phi_1=\phi\bigg|_{|t|=e^{-1}\epsilon}
 \end{aligned}
 \end{equation}
 where $h(x,|t|): \mathcal{X}_{\epsilon} \rightarrow \mathbb{R}$ is some given $S^1$ invariant function, and $\phi_0,\phi_1 \in \mathcal{H}$. We remark that, since the boundary data is clearly $S^1$ invariant it follows from the maximum principle that the solution to~\eqref{eq: mainEquation} is $S^1$ invariant as well.  We will impose three extra mild structural conditions that the data must satisfy:
 
\begin{itemize}
\item[(C1)] There is a constant $\eta_1 >0$ such that
 \begin{equation}\label{eq: strucAss1}
(n-1)\frac{\pi}{2} +\eta_1 \leq \Theta_{\omega}(\alpha_{\phi_{i}}) < n\frac{\pi}{2}, \quad i=0,1.
\end{equation}
\item[(C2)] 
There is a constant $\eta_2>0$ such that $$h: \mathcal{X}_{\epsilon} \rightarrow [(n-1)\frac{\pi}{2} + \eta_1, (n+1)\frac{\pi}{2}- \eta_2]$$ where $h$ is an $S^1$ invariant function satisfying
\begin{equation}\label{eq: strucAss2}
\Theta_{\omega}((\alpha_{\phi_i})) \geq h(x,|t|)-\frac{\pi}{2} + \eta_1
\end{equation}
for $i=0,1$.
\end{itemize}

Condition (C1) is automatic, since the boundar data $\phi_i \in \mathcal{H}$.  Condition (C2) is also automatically satisfied when $h = \hat{\theta}$, but in order to use the method of continuity we need to consider the Lagrangian phase equation with non-constant right hand side.  The estimates will exploit several properties of $F$, the most basic of which are its first and second derivatives.  It is straightforward to compute that, at a diagonal matrix $M$, the linearization $DF$, and the Hessian $D^2F$ are given by
\[
DF(M) (A) = \sum_{i}\frac{1}{1+\mu_i^2}A_{\bar{i}i}, \quad D^{2}F(M)(A,A) = \sum_{i,j=0}^{n}\frac{\mu_i+\mu_i}{(1+\mu_i^2)(1+\mu_j^2)}|A_{\bar{j}i}|^2.
\]
where $A$ is a Hermitian matrix.  The next lemma summarizes the properties of $F$ we will need.

\begin{lem}\label{lem: AngleBasicProps}
Suppose $\mu_0 \geq \mu_1 \geq \cdots \geq \mu_{n}$ are such that $\sum_{i=0}^{n} \arctan(\mu_i) \geq (n-1)\frac{\pi}{2} + \eta_1$ for some $\eta_1>0$.  The following 
properties hold,
\begin{enumerate}
\item $\mu_0 \geq \mu_1 \geq \cdots \geq \mu_{n-1} >0$ and $|\mu_n| \leq \mu_{n-1}$.
\item $\mu_n\leq 0$, then $\mu_n + \mu_{n-1} \geq \tan(\eta)$, and if $\mu_n \geq 0$, then $\mu_{n-1}\geq \tan(\frac{\eta_1}{2})$.
\item $\sum_{i=0}^{n} \mu_i >0$.
\item $\mu_n \geq - C(\eta_1)$.
\item If $\sum_{i=0}^{n+1}\arctan(\mu_i) \leq (n+1)\frac{\pi}{2} -\eta_2$ then $\mu_n \leq C(\eta_2)$.
\item If $\mu_n <0$, then $\sum_{i=0}^{n} \frac{1}{\mu_i} < -\tan(\eta)$.
\item For any $\sigma \in [(n-1)\frac{\pi}{2}, (n+1)\frac{\pi}{2})$, the set
\[
 \Gamma^{\sigma} := \{ M \in {\rm Herm}(n+1) : F(M) \geq \sigma\}
\]
is convex with boundary a smooth, convex hypersurface.
\item There exists a constant $A$ depending on $\eta_1$ such that the function
\[
\mu \longmapsto -e^{-A\sum_i \arctan(\mu_i)}
\]
is concave on the set $\Gamma^{\sigma+\eta_1}$.
\end{enumerate}
\end{lem}
\begin{proof}
We refer the reader to \cite{Yuan, WY} for properties (1), (3), (4), (6), (7), while property (8) was observed by the first author with Picard and Wu in \cite{CPW}.  Property (2) is implicit in \cite{Yuan}.  First note that
\[
(n-1)\frac{\pi}{2} + \arctan(\mu_{n-1}) + \arctan(\mu_n) \geq \sum_{i=0}^{n} \arctan(\mu_i) \geq (n-1)\frac{\pi}{2} + \eta_1
\]
and hence $\arctan(\mu_{n-1}) + \arctan(\mu_n) \geq \eta_1$.  Now if $\mu_n<0$, the sum on the left hand side of the inequality lies in $[\eta_1, \frac{\pi}{2})$, and so by the arctan addition formula we get
\[
\arctan\left(\frac{\mu_{n-1}+\mu_n}{1-\mu_{n-1}\mu_{n}}\right) \geq \eta_1
\]
Now $1-\mu_{n-1}\mu_{n} >1$, so we obtain $\mu_{n-1}+\mu_n \geq \tan(\eta_1)$.  On the other hand, if $\mu_n\geq 0$, then we have $2\arctan(\mu_{n-1}) \geq \eta_1$, which is the desired estimate.  Property (5) is trivial.
\end{proof}

Fix the following notation. Let $\Gamma_{n+1} = \mathbb{R}^{n+1}_{>0}$ be the positive orthant, and let  $\Gamma \subset \mathbb{R}^{n+1}$ be the cone over the set $\{ \mu \in \mathbb{R}^{n+1} : F(\mu) \geq (n-1)\frac{\pi}{2}\}$ and through the origin.  By Lemma~\ref{lem: AngleBasicProps} (3) and (6), $\Gamma$ is an open convex cone contained in 
\[
\{\mu \in \mathbb{R}^{n+1}: \sum_{i=0}^{n} \mu_i >0 \}
\]
The following definition is due to Sz\'ekelyhidi \cite{Sz}, building on work of Guan \cite{Guan1}.

\begin{defn}
A smooth function $\underline{\phi} :\mathcal{X}_{\epsilon} \rightarrow \mathbb{R}$ is a $\mathcal{C}$-subsolution of~\eqref{eq: mainEquation} if, at each point $p$ the following holds:  Define $\underline{\mathcal{E}}^{i}_{j} = (\hat{\omega})^{i\bar{k}}(\alpha_{\phi})_{\bar{k}j}$.  Then the set
\[
\left\{ \mu' \in \Gamma : \sum_{i=0}^{n}\arctan(\mu_i) = h(p), \quad \text{ and } \mu'-\mu(\underline{\mathcal{E}}(p)) \in \Gamma_{n+1} \right\}
\]
is bounded.  Here $\mu(\underline{\mathcal{E}}(p))$ denotes the eigenvalues of $\underline{\mathcal{E}}(p)$.
\end{defn}

The following lemma from \cite{CJY} gives a simple criterion for a function $\underline{\phi}$ to be a $\mathcal{C}$-subsolution.
\begin{lem}
A function $\underline{\phi}$ is a $\mathcal{C}$-subsolution of the equation~\eqref{eq: mainEquation} if the following holds;  at each point $p \in \mathcal{X}_{\epsilon}$, if $\mu_0,\ldots,\mu_n$ denote the eigenvalues of $\hat{\omega}^{-1}\alpha_{\underline{\phi}}$, then, for any $j=0,\ldots,n$ we have
\[
\sum_{\ell \ne j} \arctan(\mu_{\ell}) > h(p)-\frac{\pi}{2}.
\]
\end{lem}
\begin{rk}
Trivially, if $\underline{\phi}$ is a subsolution satisfying $F((\hat{\omega})^{-1}\alpha_{\underline{\phi}}) \geq h(x)$, then $\underline{\phi}$ is a $\mathcal{C}$-subsolution.
\end{rk}

The main property of $\mathcal{C}$-subsolutions that we will need is the following

\begin{lem}[Sz\'ekelyhidi  \cite{Sz}, Proposition 6]\label{lem: SzCsub}
Let $[a,b] \subset ((n-1)\frac{\pi}{2}, (n+1)\frac{\pi}{2})$, and $\delta, R >0$.  There exists a constant $\kappa_0 >0$, depending only on $a,b, R, \delta$ with the following property.  Suppose that $\sigma \in [a,b]$ and $A$ is a Hermitian matrix such that
\[
(\mu(A) -2\delta I + \Gamma_{n+1}) \cap \del \Gamma^{\sigma} \subset B_{R}(0).
\]
Then for any Hermitian matrix $M$ with $\mu(M) \in \del \Gamma^{\sigma}$, and $|\mu(M)|>R$ we either have
\[
\sum_{p,q}F^{p\bar{q}}(M)\left(A_{\bar{q}p}-M_{\bar{q}p}\right) > \kappa_0 \sum_{p}F^{p\bar{p}}(M)
\]
or 
$F^{i\bar{i}}(M) > \kappa_0 \sum_{p}F^{p\bar{p}}(M)$ for all $0\leq i \leq n$.  Here $F^{p\bar{q}}$ are the coefficients of the linearized operator of $F$.
\end{lem}

Note that if $\underline{\phi}$ is a given subsolution, then it is a $\mathcal{C}$-subsolution, and for each point $p \in \mathcal{X}_{\epsilon}$ we can choose $\delta =\delta(p)>0$, and $R=R(p)>0$ depending only on $|\DDb\underline{\phi}(p)|_{\hat{\omega}}$, $a, b$ so that we have
\[
\left[\mu((\hat{\omega})^{-1}\alpha_{\underline{\phi}}) -2\delta I + \Gamma_{n+1}\right] \cap \del \Gamma^{\sigma} \subset B_{R}(0)
\]
for all $\sigma \in [a,b]$.  In particular, the constants $\kappa_0 = \kappa_0(p), R(p)$ depend only on $|\DDb\underline{\phi}(p)|_{\hat{\omega}}$, $a, b$.

In order to solve~\eqref{eq: mainEquation} on manifolds with boundary, it is essential to construct a subsolution of the equation with given boundary data.  In the case of domains in $\mathbb{C}^n$ this was explained by the first author, Picard and Wu \cite{CPW}, building on work of Guan \cite{Guan2}.  We prove the following

\begin{lem}\label{lem: subsolConstr}
Fix two functions $\phi_0, \phi_1 \in \mathcal{H}$, and a function $h: \mathcal{X}_{\epsilon} \rightarrow \mathbb{R}$. Suppose that this data satisfies structural conditions (C1) and (C2) (see~\eqref{eq: strucAss1} and~\eqref{eq: strucAss2}).  Then there exists a smooth, $S^1$-invariant function $\underline{\phi}$ such that $\underline{\alpha} := \pi_{X}^{*}\alpha+\ddb\underline{\phi}$ satisfies
\begin{itemize}
\item[(i)] $F(\underline{\alpha}) \geq h(p)+ \frac{\eta_1}{2}$ for all $p \in \mathcal{X}_{\epsilon}$.
\item[(ii)] $\underline{\phi}|_{\{|t|=\epsilon\}} = \phi_0$
\item[(iii)] $\underline{\phi}|_{\{|t|=e^{-1}\epsilon\}} = \phi_1$
\item[(iv)] On $\del \mathcal{X}_{\epsilon}$ we have
\[
\bigg|\frac{\del}{\del |t|}\underline{\phi}\bigg| \leq \frac{C}{\epsilon}, \qquad \bigg|\frac{\del}{\del |t|}\underline{\phi}\bigg| \leq \frac{C}{\epsilon^2}
\]
for a constant $C$ depending only on $\eta_1$ and $\|\phi_0-\phi_1\|_{L^{\infty}(X)}$.
\end{itemize}
\end{lem}

Before proving the lemma, we prove a fairly general result which we hope will be of use elsewhere.  Let us first recall Demailly's regularized maximum construction \cite{DemBook}.  Fix $\theta: \mathbb{R} \rightarrow \mathbb{R}$ a smooth, even, positive function with support in $[-1,1]$ and such that
\[
\int_{[-1,1]}\theta(x)dx =1.
\]
For $\delta>0$, and $(t_0,t_1)\in \mathbb{R}^2$ define
\[
M_{\delta}(t_0,t_1) = \int_{\mathbb{R}^2}\max\{t_0+h_0, t_1+h_1\} \frac{1}{\delta}\theta\left(\frac{h_0}{\delta}\right)\frac{1}{\delta}\theta\left(\frac{h_1}{\delta}\right) dh_0dh_1.
\]
The function $M_{\delta}$ has the following properties \cite{DemBook}
\begin{itemize}
\item $M_{\delta}$ is non-decreasing in $t_i$ for $i=0,1$, smooth and convex.
\item $\max\{t_0,t_1\} \leq M_{\delta}(t_0, t_1) \leq \max\{t_0+\delta, t_1+\delta\}$
\item if $t_1+\delta \leq t_0-\delta$ then $M_{\delta}(t_0, t_1) = t_0$, and vice versa.
\item $M_{\delta}(t_0+a, t_1+a) = M_{\delta}(t_0,t_1)+a$.  In particular
\[
\nabla_{(1,1)}M_{\delta} = \frac{\del M_{\delta}}{\del t_0} + \frac{\del M_{\delta}}{\del t_1} =1.
\]
\end{itemize}

Let $\psi_0, \psi_1$ be two $C^{2}$ functions.  We are going to compute two derivatives of $M(\psi_0,\psi_1)$.  We compute
\[
\dbar M_{\delta}(\psi_0, \psi_1) = \frac{\del M_{\delta}}{\del t_0}(\psi_0, \psi_1) \dbar \psi_0 +  \frac{\del M_{\delta}}{\del t_1}(\psi_0, \psi_1)\dbar\psi_1
\]
and so
\begin{equation}\label{eq: 2derDemMax}
\del\dbar M_{\delta}(\psi_0, \psi_1) = \frac{\del M_{\delta}}{\del t_0} \del\dbar \psi_0 +  \frac{\del M_{\delta}}{\del t_1}\del\dbar\psi_1 + (\del \psi_0, \del \psi_1) D^{2}M (\del \psi_0, \del \psi_1)^{\dagger}.
\end{equation}
We now make two observations.  First, by the convexity of $M$, the second term on the right hand side is non-negative. Secondly, by  combining the first and last properties of $M_{\delta}$ the first term on the right hand of~\eqref{eq: 2derDemMax} is a convex combination of $\del\dbar \psi_0$, and $\del\dbar\psi_1$.  We conclude the following very general lemma.
\begin{lem}
Suppose $F: {\rm Herm}(n)\rightarrow \mathbb{R}$ is an elliptic operator such that $\{F \geq 0\}$ is convex.  Suppose $\psi_0, \psi_1$ are two functions satisfying $F(\del\dbar \psi_i) \geq 0$ for $i=0,1$.  Then
\[
F\left(\del\dbar M_{\delta}(\psi_0, \psi_1)\right) \geq 0
\]
\end{lem}
\begin{proof}
By the above computation, using the ellipticity of $F$ we have
\[
F\left(\del\dbar M_{\delta}(\psi_0, \psi_1)\right) \geq F(t\del\dbar \psi_0 +(1-t)\del\dbar \psi_1)
\]
for some $0\leq t \leq1$. Now $F(\del\dbar \psi_i) \geq 0$ and $\{F \geq 0\}$ convex implies the result.
\end{proof}

\begin{rk}
The result also holds fo real functions, with $\del\dbar \psi_i$ replaced by $D^2\psi_i$, provided $F$ is elliptic on the symmetric matrices and $\{F \geq 0\}$ is convex.
\end{rk}

We can now give the proof of Lemma~\ref{lem: subsolConstr}.  
\begin{proof}[Proof of Lemma~\ref{lem: subsolConstr}]
Consider the function
\[
\psi_{0}(x,t) = \phi_0(x) + A_0(|t|^2-\epsilon^2) + C_0\log\left(\frac{|t|^2}{\epsilon^2}\right)
\]
Observe that $\psi_0(x,t) = \phi_0(x)$ whenever $|t|=\epsilon$.  Furthermore we have
\[
\DDb\psi_0 = \ddb\phi_0 + A_0 \sqrt{-1}dt\wedge d\bar{t}
\]
and so
\[
F(\alpha_{\psi_0}) = \Theta_{\omega}(\alpha_{\phi_0}) + \arctan(A_0)
\]
and so if we choose $A_0$ sufficiently large so that $\arctan(A_0) >\frac{\pi}{2}-\frac{\eta_1}{2}$, then
\[
F(\alpha_{\psi_0}) \geq h(x,|t|) - \frac{\pi}{2}+\eta_1 + \frac{\pi}{2}-\frac{\eta_1}{2}  =h(x,|t|) + \frac{\eta_1}{2}.
\]
Similarly define 
\[
\psi_1(x,t) =  \phi_1(x) + A_1(|t|^2-e^{-2}\epsilon^2) - C_1\log\left(\frac{e^{2}|t|^2}{\epsilon^2}\right).
\]
Note that $\psi_1(x,t) =  \phi_1(x)$ whenever $|t|=e^{-1}\epsilon$.  A similar computation shows that we can choose $A_1$ large depending only on $\eta_1$ so that $F(\alpha_{\psi_1}) \geq h(x, |t|) + \frac{\eta_{1}}{2}$. Now we have
\[
\psi_{0}(x,t)|_{\{|t|=e^{-1}\epsilon\}} = \phi_0(x) -A_0\epsilon^{2}(1-e^{-2})-2C_0
\]
Choose $C_0 \geq \|\phi_0 -\phi_1\|_{L^{\infty}(X)}+1$, so that
\[
\psi_{0}(x,t)|_{\{|t|=e^{-1}\epsilon\}} \leq \phi_1(x)-1 = \psi_1(x)|_{\{|t|=e^{-1}\epsilon\}} -1.
\]
Similarly
\[
\psi_{1}(x,t)|_{\{|t|=\epsilon\}} = \phi_1(x) +A_1\epsilon^{2}(1-e^{-2})-2C_1
\]
Choosing $C_1$ large depending on $\|\phi_0-\phi_1\|_{L^{\infty}(X)}$, and $A_1$ we can ensure that
\[
\psi_{1}(x,t)|_{\{|t|=\epsilon\}} \leq \phi_0 -1 =  \psi_0(x)|_{\{|t|=\epsilon\}} -1
\]
Now set $\underline{\phi} = M_{\delta}(\psi_0,\psi_1)$ for $0<\delta \ll 1$.  Since the sublevel sets of the angle operator $\Gamma^{\sigma}$ are convex when $\sigma \geq (n-1)\frac{\pi}{2}$, we can apply the above construction using Demailly's regularized maximum to conclude that $\underline{\phi}$ satisfies $(i)$.  Furthermore, $\underline{\phi}$ satisfies properties $(ii), (iii)$ and the first part of $(iv)$ of the lemma by construction.  It remains only to prove the estimates in $(iv)$.  For this, note that for $\delta\ll 1$, we have that
\[
\begin{aligned}
\underline{\phi} &= \psi_0 \text{ in an open neighbourhood of } |t|=\epsilon\\
\underline{\phi} &= \psi_1 \text{ in an open neighbourhood of } |t|=e^{-1}\epsilon
\end{aligned}
\]
and so we only need to check the claimed estimates for $\phi_0, \phi_1$.  The estimates are automatic from the formulas for $\psi_i$, and the choice of the constants $A_i, C_i$ above.
\end{proof}

\begin{rk}
A word of caution is in order here.  By~\eqref{eq: 2derDemMax} and the construction of $\underline{\phi}$ it is easy to see that $\DDb\underline{\phi}$ is {\em not} bounded above independent of $\epsilon$ on the whole of $\mathcal{X}_{\epsilon}$. 
\end{rk}

Because of this remark it is often more convenient to work with the functions $\psi_0, \psi_1$ constructed above.  In particular, we note the following corollary. 

\begin{cor}\label{cor: intSubSol}
Fix functions $\phi_0, \phi_1 \in \mathcal{H}$, and a function $h: \mathcal{X}_{\epsilon} \rightarrow \mathbb{R}$. Suppose that this data satisfies structural conditions (C1) and (C2) (see~\eqref{eq: strucAss1} and~\eqref{eq: strucAss2}). Then there exist smooth, $S^1$-invariant functions $\hat{\phi}_{i}$ for $i=0,1$ such that $\hat{\phi}_i$, and $\hat{\alpha}_i := \pi_{X}^{*}\alpha+\DDb\hat{\phi}_i$ have the following properties;
\begin{itemize}
\item[$(0)$] $\|\hat{\phi}_i\|_{L^{\infty}(\mathcal{X}_{\epsilon})}$ is bounded by a constant depending only on $\eta_1$, $\phi_0, \phi_1$.
\item[$(i)$] $\hat{\phi}_{0} = \phi_0$ on the set $|t|=\epsilon$, and $\hat{\phi}_{0} < \phi_1$ on $|t|=e^{-1}\epsilon$.
\item[$(ii)$] $\hat{\phi}_{1} = \phi_1$ on the set $|t|=e^{-1}\epsilon$, and $\hat{\phi}_{1} < \phi_0$ on $|t|=\epsilon$.
\item[$(iii)$] $F(\hat{\alpha}_{i}) \geq h(x,t)+ \frac{\eta_1}{2}$ for $i=0,1$
\item[$(iv)$] For each $i=0,1$, $\nabla_{t}\hat{\alpha}_i = \nabla_{\bar{t}}\hat{\alpha}_i =0$, and $|\nabla \hat{\alpha}_i|_{\hat{\omega}} \leq C$ for a constant $C$ depending only on $\|\phi_i\|_{C^{3}(X,\omega)}$.
\item[$(v)$] For each $i=0,1$, $|\DDb \hat{\phi}_i|_{\hat{\omega}}$ is controlled uniformly in terms of $\eta_1, |\ddb \phi_i|_{L^{\infty}(X,\omega)}$
\item[$(vi)$] For each $i=0,1$, $\sup_{\mathcal{X}_{\epsilon}}|\nabla^{X}\hat{\phi}_{i}|_{\hat{g}} = \sup_{X}|\nabla^{X}\phi_{i}|_{g}$
\item[$(vii)$] $(\hat{\alpha}_{i})_{\bar{t}X}=(\hat{\alpha}_{i})_{\bar{X}t}=0$.
\item[$(viii)$] Near $\del X_{\epsilon}$ we have
\[
\bigg|\frac{\del}{\del |t|}\underline{\phi}\bigg| \leq \frac{C}{\epsilon}, \qquad \bigg|\frac{\del}{\del |t|}\underline{\phi}\bigg| \leq \frac{C}{\epsilon^2}
\]
for a constant $C$ depending only on $\eta_1$ and $\|\phi_0-\phi_1\|_{L^{\infty}(X)}$.
\end{itemize}
\end{cor}
\begin{proof}
Take $\hat{\phi}_i = \psi_i$ in proof of Lemma~\ref{lem: subsolConstr}.  Then $(0), (i), (ii), (iii),(v), (vi), (viii)$ hold automatically, and we have
\[
\hat{\alpha} = \alpha+\ddb\phi_0 +A_0\sqrt{-1}dt\wedge d\bar{t}
\]
which implies propert $(vii)$.  Property $(iv)$ follows from the fact that $(X_{\epsilon}, \hat{\omega})$ is the product of $(X,\omega)$ with a flat factor.
\end{proof}
\begin{rk}\label{rk: uniConst}
The advantage of the subsolutions $\hat{\phi}_{i}$ is that the constants $R, \kappa_0$ appearing in Lemma~\ref{lem: SzCsub} can be chosen depending only on the data $\eta_1$, and the $C^2$ norm of the boundary data.
\end{rk}

We finish with the following simple estimate.

\begin{prop}\label{prop: c0bound}
Suppose $\phi: \mathcal{X}_{\epsilon}\rightarrow \mathbb{R}$ solves~\eqref{eq: mainEquation}, and suppose $\underline{\phi}$ is a subsolution of~\eqref{eq: mainEquation}.  Then we have
\[
\underline{\phi} \leq \phi \leq \max\{\|\phi_0\|_{L^{\infty}(X)}, \|\phi_{1}\|_{L^{\infty}(X)}\} + C\epsilon^{2}
\]
for a constant $C$ depending only on $\omega, \alpha$.
\end{prop}
\begin{proof}
The bound $\phi \geq \underline{\phi}$ follows from the comparison principle.  On the other hand, by Lemma~\ref{lem: AngleBasicProps} (3), we have
\[
\Delta_{\hat{\omega}}\phi \geq - \Tr_{\omega}\alpha \geq -C.
\]
It follows that the function $\phi + C|t|^2$ is subharmonic and hence achieves its maximum on the boundary.  The estimate follows.
\end{proof}

In order to prove the desired estimates, we need to proceed with extreme care, ensuring at every step that constants appearing in the estimate are independent of $\epsilon$.  To ease the presentation we introduce the following terminology.

\begin{defn}
A constant $C$ is uniform if it is independent of $\epsilon$, and is invariant after rescaling $t \mapsto \epsilon^{-1}t$.
\end{defn}

For example, constants depending only on the boundary data $\phi_i$, $(X,\omega)$, the constants $\eta_1, \eta_2$ in~\eqref{eq: strucAss1}~\eqref{eq: strucAss2}, and the norms of spatial derivatives $|(\nabla^{X})^{k}(\overline{\nabla^{X}})^r h|_{\hat{\omega}}$ are {\em uniform}.  On the other hand, constants depending on $|\nabla^{X} \nabla_{\bar{t}}h|_{\hat{\omega}}$, and $|\nabla_{t}\nabla_{\bar{t}} h|_{\hat{\omega}}$ are not uniform, since the norms (measured with respect to $\hat{\omega}$) rescale (unless $h$ is a constant).

\section{The $C^1$ estimates}\label{sec: C1est}

The goal of this section is to prove an a priori gradient estimate for solutions of~\eqref{eq: mainEquation}.  We begin with a uniform spatial gradient estimate.  In fact, this estimate can be deduced from the interior spatial $C^2$ estimate proved in Section~\ref{sec: C2est}.  However, we include it since it may be independent interest, and we expect it to be applicable to the study of geodesic rays in $\mathcal{H}$.
\begin{prop}\label{prop: spaceGradEst}
Suppose $\phi$ solves~\eqref{eq: mainEquation} for boundary data $\phi_i \in \mathcal{H}$ for $i=0,1$, and with $\phi_0, \phi_1, h$ satisfying~\eqref{eq: strucAss1},~\eqref{eq: strucAss2}.  Then there exists a uniform constant $C$ so that
\[
|\nabla^{X}\phi|_{\hat{g}} \leq C
\]
\end{prop}

\begin{proof}
In order to estimate the spatial gradient we adapt ideas of B\l ocki \cite{Bl, Bl1} and Phong-Sturm \cite{PS, PS1} used to obtain gradient estimates for solutions of the complex Monge-Amp\`ere equation.  Before beginning the proof, we choose an appropriate background form $\alpha$ on $X_{\epsilon}$.  For this section we choose as our background reference form $\alpha = \pi_{X}^{*}\alpha + \ddb \hat{\phi}_i$ for $i=0$ or $1$, where $\hat{\phi}_i$ is one of the subsolutions constructed in Corollary~\ref{cor: intSubSol}.  To ease notation, we will drop the subscript $i$ in the notation.  If we write the solution to~\eqref{eq: mainEquation} as $\pi_{X}^{*}\alpha +\DDb\phi$, then $ \pi^{*}\alpha +\DDb\phi = \alpha +\DDb(\phi -\hat{\phi})$.  Our goal is to prove
\[
|\nabla^{X}(\phi-\hat{\phi})|_{\hat{g}} \leq C
\]
for a uniform constant $C$.  This implies that $|\nabla^{X}\phi|_{\hat{g}} \leq C'$ with $C'$ depending in addition on the spatial $C^1$ norm of $\hat{\phi}$, which by Corollary~\ref{cor: intSubSol} is uniformly controlled in terms of the boundary data.

With this understood, in order to lighten notation we still write the solution to the equation as $\alpha+\DDb\phi$, and estimate $|\nabla^{X}\phi|_{\hat{g}}$.  Following \cite{PS, PS1, Bl, Bl1} (see also \cite{PSS}), we apply the maximum principle to the quantity
\[
Q:= \log\left(|\nabla^{X}\phi|^{2}\right) + \gamma(\phi)
\]
where $\gamma: \mathbb{R} \rightarrow \mathbb{R}$ is a function to be determined.  Fix a point $(p_0,t_0) \in X_{\epsilon}$, and a space-time adapted coordinate system $(w_0,\ldots,w_n)$ centered a $(p_0,t_0)$.  We compute
\[
\begin{aligned}
\nabla_{\bar{j}}|\nabla^{X}\phi|^2 = &\sum_{1\leq k, \ell \leq n} \hat{g}^{k\bar{\ell}}\left(\nabla_{\bar{j}}\nabla_k\phi \nabla_{\bar{\ell}}\phi + \nabla_k\phi \nabla_{\bar{j}}\nabla_{\bar{\ell}}\phi\right)\\
\nabla_{i}\nabla_{\bar{j}}|\nabla^{X}\phi|^2 =&\sum_{1\leq k, \ell \leq n} \hat{g}^{k\bar{\ell}}\left(\nabla_i\nabla_{\bar{j}}\nabla_k\phi \nabla_{\bar{\ell}}\phi + \nabla_{\bar{j}}\nabla_k\phi \nabla_{i}\nabla_{\bar{\ell}}\phi\right) \\
&\sum_{1\leq k, \ell \leq n}\hat{g}^{k\bar{\ell}}\left( \nabla_{i}\nabla_k\phi \nabla_{\bar{j}}\nabla_{\bar{\ell}}\phi +\nabla_k\phi \nabla_{i}\nabla_{\bar{j}}\nabla_{\bar{\ell}}\phi\right)
\end{aligned}
\]
To deal with the first and last terms, we differentiate the equation to get
\[
\begin{aligned}
F^{i\bar{j}}\nabla_{k}\nabla_{i}\nabla_{\bar{j}}\phi &= \nabla_kh -\nabla_{k}\alpha_{\bar{j}i}\\
F^{i\bar{j}}\nabla_{\bar{\ell}}\nabla_{i}\nabla_{\bar{j}}\phi &= \nabla_{\bar{\ell}}h -\nabla_{\bar{\ell}}\alpha_{\bar{j}i}.
\end{aligned}
\]
Commuting derivatives gives
\[
\begin{aligned}
\nabla_{k}\nabla_{i}\nabla_{\bar{j}}\phi &= \nabla_{i}\nabla_{k}\nabla_{\bar{j}}\phi=\nabla_{i}\nabla_{\bar{j}}\nabla_{k}\phi\\
\nabla_{\bar{\ell}}\nabla_{i}\nabla_{\bar{j}}\phi &= -R_{\bar{\ell}i\bar{j}}\,^{\bar{p}}\nabla_{\bar{p}}\phi +\nabla_{i}\nabla_{\bar{j}}\nabla_{\bar{\ell}}\phi
\end{aligned}
\]
Combining these formulae we get
\[
\begin{aligned}
F^{i\bar{j}}\nabla_{i}\nabla_{\bar{j}}|\nabla^{X}\phi|^2 &= \sum_{1\leq k, \ell \leq n} F^{i\bar{j}}\left(\hat{g}^{k\bar{\ell}}\nabla_{\bar{j}}\nabla_k\phi \nabla_{i}\nabla_{\bar{\ell}}\phi +\nabla_{i}\nabla_k\phi \nabla_{\bar{j}}\nabla_{\bar{\ell}}\phi\right)\\
&\quad +2{\rm Re}\left(\sum_{1\leq k,\ell \leq n} \hat{g}^{k\bar{\ell}}(\nabla_kh -F^{i\bar{j}}\nabla_{k}\alpha_{\bar{j}i})\nabla_{\bar{\ell}}\phi \right)\\
&\quad + F^{i\bar{j}}\sum_{\substack{1\leq k, \ell\leq n\\ 0\leq p \leq n}} \hat{g}^{k\bar{\ell}}R_{\bar{\ell}i\bar{j}}\,^{\bar{p}}\nabla_{\bar{p}}\phi \nabla_k\phi
\end{aligned}
\]
By Lemma~\ref{lem: productMetric} we have that $R_{\bar{\ell}i\bar{j}}\,^{\bar{p}} =0$ if any of $i,j,\ell$, or $p$ are $0$, so we can choose a constant $C_1$, depending only on $(X,\omega)$ such  that
\[
F^{i\bar{j}}\sum_{1\leq k, \ell, p \leq n} \hat{g}^{k\bar{\ell}}R_{\bar{\ell}i\bar{j}}\,^{\bar{p}}\nabla_{\bar{p}}\phi \nabla_k\phi \geq -C_1|\nabla^{X}\phi|_{\hat{g}}^2,
\]
where we also used that $F^{i\bar{j}} \leq \hat{g}^{i\bar{j}}$.  Similarly, by Corollary~\ref{cor: intSubSol} we can choose a constant $C_2$ depending only on the boundary data, and the spatial derivative of $h$, so that
\[
2{\rm Re}\left(\sum_{1\leq k,\ell \leq n} \hat{g}^{k\bar{\ell}}(\nabla_kh -F^{i\bar{j}}\nabla_{k}\alpha_{\bar{j}i})\nabla_{\bar{\ell}}\phi \right) \geq -C_2|\nabla^{X}\phi|_{\hat{g}}.
\]
Summarizing we have
\[
\begin{aligned}
F^{i\bar{j}}\nabla_{i}\nabla_{\bar{j}}\log\left(|\nabla^{X}\phi|_{\hat{g}}^2\right) &\geq\frac{1}{|\nabla^{X}\phi|^2} \sum_{1\leq k, \ell \leq n} F^{i\bar{j}}\hat{g}^{k\bar{\ell}}\left(\nabla_{\bar{j}}\nabla_k\phi \nabla_{i}\nabla_{\bar{\ell}}\phi +\nabla_{i}\nabla_k\phi \nabla_{\bar{j}}\nabla_{\bar{\ell}}\phi\right)\\
&\quad- \frac{1}{(|\nabla^{X}\phi|^2_{\hat{g}})^2}F^{i\bar{j}}\nabla_{i}|\nabla^{X}\phi|^2_{\hat{g}} \nabla_{\bar{j}}|\nabla^{X}\phi|^2_{\hat{g}}- C_1- \frac{C_2}{|\nabla^{X}\phi|_{\hat{g}}}
\end{aligned}
\]
for uniform constants $C_1, C_2$.  Abusing notation, define a norm on $1$-forms by $\langle\sigma, \sigma'\rangle_{F} = F^{i\bar{j}}\sigma_i\overline{\sigma'_j}$, and write
\[
 \sum_{1\leq k, \ell \leq n} F^{i\bar{j}}\hat{g}^{k\bar{\ell}}\left(\nabla_{\bar{j}}\nabla_k\phi \nabla_{i}\nabla_{\bar{\ell}}\phi +\nabla_{i}\nabla_k\phi \nabla_{\bar{j}}\nabla_{\bar{\ell}}\phi\right) = |\overline{\nabla}\nabla^{X}\phi|^{2}_{F\otimes \hat{g}} +  |\nabla\nabla^{X}\phi|^{2}_{F\otimes \hat{g}}.
\]
Define $1$-forms
\[
S_i := \langle \nabla_i\nabla^{X}\phi, \nabla^{X}\phi \rangle_{\hat{g}}, \quad T_{j}:= \langle \nabla^{X}\phi, \nabla_{\bar{j}}\nabla^{X}\phi \rangle_{\hat{g}}
\]
so that $\nabla|\nabla^{X}\phi|^2_{\hat{g}} = S+T$.  By the Cauchy-Schwarz inequality we have
\[
|S|^2_{F}+|T|^2_{F} \leq |\nabla^{X}\phi|^2_{\hat{g}}\left(|\overline{\nabla}\nabla^{X}\phi|^{2}_{F\otimes \hat{g}} +  |\nabla\nabla^{X}\phi|^{2}_{F\otimes \hat{g}}\right).
\]
It follows that
\[
|\nabla|\nabla^{X}\phi|^2_{\hat{g}}|^{2}_{F} -2{\rm Re}\left( \left\langle \nabla|\nabla^{X}\phi|^2_{\hat{g}}, T \right\rangle_{F}\right) \leq |\nabla^{X}\phi|^2_{\hat{g}}\left(|\overline{\nabla}\nabla^{X}\phi|^{2}_{F\otimes \hat{g}} +  |\nabla\nabla^{X}\phi|^{2}_{F\otimes \hat{g}}\right).
\]
Plugging this into the estimate yields
\[
F^{i\bar{j}}\nabla_{i}\nabla_{\bar{j}}\log\left(|\nabla^{X}\phi|_{\hat{g}}^2\right) \geq -2{\rm Re}\left( \left\langle \frac{\nabla|\nabla^{X}\phi|^2_{\hat{g}}}{|\nabla^{X}\phi|^2_{\hat{g}}},\frac{T}{|\nabla^{X}\phi|^2_{\hat{g}}} \right\rangle_{F}\right) - C_1- \frac{C_2}{|\nabla^{X}\phi|_{\hat{g}}}.
\]
Let us examine the term $T$.  In our coordinate system we have
\[
T_j=\sum_{1\leq k,\ell \leq n } \nabla_{k}\phi \hat{g}^{k\bar{\ell}}\left((\alpha_{\phi})_{\bar{\ell}j} - \alpha_{\bar{\ell}j}\right).
\]
If we let $\mathcal{E}, \mathcal{E}_{0}$ denote that endomorphisms $\hat{g}^{k\bar{\ell}}(\alpha_{\phi})_{\bar{\ell}j}$, and $\hat{g}^{k\bar{\ell}}\alpha_{\bar{\ell}j}$ respectively,  then we can write $T$ invariantly as 
\[
T= (\mathcal{E}-\mathcal{E}_0)\nabla^{X}\phi.
\]
Suppose $Q$ achieves an interior maximum at the point $(p_0,t_0)$.  At this point we have $\nabla Q=0$, and $F^{i\bar{j}}\nabla_{i}\nabla_{\bar{j}}Q \leq 0$.  Writing out the first of these gives
\[
\frac{\nabla|\nabla^{X}\phi|^2_{\hat{g}}}{|\nabla^{X}\phi|^2_{\hat{g}}} = \gamma'(\phi) \nabla\phi
\]
and so at $(p_0,t_0)$ we get
\[
 {\rm Re}\left( \left\langle \frac{\nabla|\nabla^{X}\phi|^2_{\hat{g}}}{|\nabla^{X}\phi|^2_{\hat{g}}}, \frac{T}{|\nabla^{X}\phi|^2_{\hat{g}}} \right\rangle_{F}\right) =\frac{\gamma'(\phi)}{|\nabla^{X}\phi|^2_{\hat{g}}} {\rm Re}\left(\langle\nabla \phi, (\mathcal{E}-\mathcal{E}_0)\nabla^{X}\phi\rangle_{F}\right)
 \]
 For one moment, let us choose another system of holomorphic normal coordinates for $\hat{g}$, $(z_0, \ldots, z_n)$ centered at $(p_0,t_0)$ so that $\hat{g}$ is the identity, and $\mathcal{E}$ is diagonal, with eigenvalues $\mu_0, \ldots, \mu_n$.  Write $\nabla^{X}\phi = b_i\del_{z_i}$ for some complex numbers $b_i$.  Then
 \[
 \begin{aligned}
\bigg| \langle\nabla \phi, \mathcal{E}\nabla^{X}\phi\rangle_{F}\bigg| = \bigg|\sum_{i}\frac{\mu_i\phi_i \overline{b_i}}{(1+\mu_i^2)}\bigg| &\leq \left(\sum_{i}\frac{\mu_i^2|b_i|^2}{1+\mu_i^2}\right)^{1/2}\left(\sum_j \frac{|\phi_j|^2}{1+\mu_j^2}\right)^{1/2}\\
&\leq |\nabla^{X}\phi|_{\hat{g}}|\nabla \phi|_{F}
\end{aligned}
 \]
 Similarly, if we write $\mathcal{E}_{0}\nabla^{X}\phi = c_i\del_{z_i}$, then we have
 \[
 \begin{aligned}
 \bigg|\langle\nabla \phi, \mathcal{E}_0\nabla^{X}\phi\rangle_{F}\bigg| &= \bigg|\sum_i \frac{\phi_i \overline{c_i}}{1+\mu_i^2}\bigg| \leq  \left(\sum_{i}\frac{|c_i|^2}{1+\mu_i^2}\right)^{1/2}\left(\sum_j \frac{|\phi_j|^2}{1+\mu_j^2}\right)^{1/2}\\
 & \leq |\mathcal{E}_0 \nabla^{X} \phi|_{\hat{g}}|\nabla \phi|_{F}
 \end{aligned}
 \]
By Corollary~\ref{cor: intSubSol} we have $\mathcal{E}_{0}TX \subset TX$.  Since $\nabla^{X}\phi \in TX\subset T\mathcal{X}_{\epsilon}$ we get
 \[
  |\mathcal{E}_0 \nabla^{X} \phi|_{\hat{g}} \leq C_3|\nabla^{X}\phi|_{\hat{g}}
 \]
 for a uniform constant $C_3$ depending on the boundary data.  Putting everything together gives
 \begin{equation}\label{eq: GradMaxPrin}
 0 \geq -C_1 -\frac{C_2}{|\nabla^{X}\phi|_{\hat{g}}}-C_3\gamma'(\phi)\frac{|\nabla \phi|_{F}}{|\nabla^{X}\phi|_{\hat{g}}} - \gamma'(\phi)F^{i\bar{j}}\phi_{\bar{j}i} - \gamma''(\phi)|\nabla \phi|^2_{F}
 \end{equation}
Following Phong-Sturm \cite{PS, PS1, PSS} we choose
\[
\gamma(\phi) = B\phi - \frac{1}{\phi+C_4}
\]
where $C_4 = -\inf_{\mathcal{X}_{\epsilon}}\phi +1$, and $B$ is a large positive constant to be determined.  Note that
\[
B\phi -1 \leq \gamma(\phi) \leq B\phi, \quad B \leq \gamma'(\phi) \leq B+1, \quad \gamma''(\phi) = -2\frac{1}{(\phi+C_4)^3} <0.
\] 
We may assume that $|\nabla^{X} \phi|_{\hat{g}} >1$ at the maximum point of $Q$, for otherwise we're done.  We need to consider several cases.

First, suppose that 
\begin{equation}\label{eq: gradCase1}
\epsilon_0^2|\nabla^{X}\phi|^2 \geq  |\nabla \phi|^2_{F}
\end{equation}
for some $0<\epsilon_0 <1$ to be determined.  Let $R, \kappa_0$ be the constants appearing in Lemma~\ref{lem: SzCsub} for the subsolution $\hat{\phi}$ and recall that they are uniformly controlled (see Remark~\ref{rk: uniConst}). If every eigenvalue of $\mathcal{E}$ is smaller than $R$, then we have
\[
|\nabla \phi|^{2}_{F} \geq \frac{1}{1+R^2}|\nabla \phi|^2_{\hat{g}} \geq \frac{1}{1+R^2}|\nabla^{X}\phi|^2_{\hat{g}}.
\]
If we take $\epsilon_0$ small so that $100 \epsilon_0^2 \leq \frac{1}{1+R^2}$, then under the assumption~\eqref{eq: gradCase1} some eigenvalue of $\mathcal{E}$ at $(p_0,t_0)$ must be larger than $R$.  In particular, by Lemma~\ref{lem: SzCsub} we get
\[
F^{i\bar{j}}\phi_{\bar{j}i}  = F^{i\bar{j}}(\alpha_\phi)_{\bar{j}i} - \alpha_{\bar{j}i} \leq -\kappa_0\sum_{p=0}^{n}\frac{1}{1+\mu_{p}^2} < -\kappa_0\frac{1}{1+C(\eta)^2}
\]
where $C(\eta):= C(\eta_1,\eta_2)$ is the bound for $|\mu_n|$ in Lemma~\ref{lem: AngleBasicProps}.  Thus, assuming~\eqref{eq: gradCase1} implies
\[
0\geq -C_1-C_2-C_3(B+1)\epsilon_0+B\kappa_0\frac{1}{1+C(\eta)^2}+ \frac{2}{\phi+C_4}|\nabla\phi|^2_{F}
\]
We now choose $B$ large depending only on $C_1,C_2, C_3$ so that
\[
B\kappa_0\frac{1}{1+C(\eta)^2} \geq C_1+C_2+C_3+1
\]
and choose $\epsilon_0$ small depending on $B$ so that $\epsilon_0(B+1)\leq 1$.  Clearly $\epsilon_0, B$ can be chosen to be uniform constants.  With these choices we conclude that $Q$ cannot attain an interior maximum at which~\eqref{eq: gradCase1} holds.

We may therefore assume that if $Q$ achieves an interior maximum at $(p_0, t_0)$ then
\begin{equation}\label{eq: gradCase2}
\epsilon_0^2|\nabla^{X}\phi|^2 \leq  |\nabla \phi|^2_{F}
\end{equation}
at $(p_0,t_0)$.  We may also assume that $|\nabla \phi|_{F} \geq 1$ at $(p_0,t_0)$, for otherwise $|\nabla^{X}\phi|^2 \leq \epsilon_0^{-2}$ and we are done.  Rearranging~\eqref{eq: GradMaxPrin} we get
\[
\frac{2}{(1+\osc_{X_{\epsilon}}\phi)^3}|\nabla \phi|^{2}_{F} \leq C_1+C_2+(B+1)(C_3\frac{|\nabla \phi|_{F}}{|\nabla^{X}\phi|_{\hat{g}}} +C_5)
\]
where we used that
\[
F^{i\bar{j}}\phi_{\bar{j}i} = \sum_i \frac{\mu_i -\alpha_{\bar{i}i}}{1+\mu_i^2} \leq C_5
\]
for a uniform constant $C_5$ by Corollary~\ref{cor: intSubSol}.  Let us simplify the notation by writing
\[
\tilde{\delta} = \frac{2}{(1+\osc_{\mathcal{X}_{\epsilon}}\phi)^3}, \quad \tilde{B} = (B+1)C_3, \quad \tilde{A}= C_1+C_2+(B+1)C_5.
\]
so that $\tilde{A}, \tilde{B}, \tilde{\delta}$ are uniform constants and
\begin{equation}\label{eq: gradCase2est}
\tilde{\delta}|\nabla \phi|^{2}_{F} \leq \tilde{A} + \tilde{B}\frac{|\nabla \phi|_{F}}{|\nabla^{X}\phi|_{\hat{g}}}.
\end{equation}
There are now two cases.   First, observe that if $\tilde{\delta} \leq \frac{2\tilde{B}}{|\nabla_{X}\phi|_{\hat{g}}}$ then we are done.  So we may assume that $\tilde{\delta} > \frac{2\tilde{B}}{|\nabla_{X}\phi|_{\hat{g}}}$. Upon rearranging~\eqref{eq: gradCase2est} we obtain
\[
\tilde{A} \geq \left( \tilde{\delta}|\nabla \phi|_{F} - \frac{\tilde{B}}{|\nabla^{X}\phi|_{\hat{g}}}\right)|\nabla \phi|_{F} \geq \frac{\tilde{\delta}}{2}|\nabla \phi|_{F} 
\]
where we used our assumption that $|\nabla \phi|_{F} \geq 1$.  We now use~\eqref{eq: gradCase2} to obtain
\[
|\nabla^{X}\phi|_{\hat{g}} \leq \epsilon_0^{-1}|\nabla \phi|_{F} \leq \frac{2\tilde{A}}{\epsilon_0\tilde{\delta}},
\]
which is the desired estimate.  Since the constants $\epsilon_0, \tilde{\delta}, \tilde{A}, \tilde{B}$ are uniform, we have shown that, if $Q$ attains an interior maximum it is bounded uniformly.  Since $\phi$ is bounded uniformly it follows that $|\nabla^{X}\phi|_{\hat{g}}^2$ is uniformly bounded.  Thus we are finished unless $Q$ attains its maximum on the boundary.  But on the boundary $Q$ is clearly bounded from above by a constant depending only on the boundary data.
\end{proof}

It only remains to estimate the temporal derivative $|\nabla_{t}\phi|_{\hat{g}}$.  The first step is to reduce the estimate for $|\nabla_t \phi|_{\hat{g}}$ to a boundary estimate.  Write the solution to ~\eqref{eq: mainEquation} as $\pi^{*}\alpha + \DDb \phi$, write $t=u+\sqrt{-1}v$ and compute
\[
F^{i\bar{j}}\nabla_i\nabla_{\bar{j}}\nabla_{u}\phi = F^{i\bar{j}}\nabla_{u}\nabla_i\nabla_{\bar{j}}\phi \geq - \hat{C}
\]
where $\hat{C}=-\sup_{\mathcal{X}_{\epsilon}}|\nabla_t h|_{\hat{g}}$.  A similar estimate holds for $v$.   In this computation we've used that the curvature vanishes along any $\del_t$ direction, and that $\nabla_{t}\pi_{X}^{*}\alpha =0$ by Lemma~\ref{lem: productMetric}.  Note that $\hat{C}$ is not a uniform constant.  Consider the quantity 
\[
Q= \del_u\phi + A|t|^2 - B(\phi-\hat{\phi}),
\]
where $\hat{\phi}$ is either one of the subsolutions constructed in Corollary~\ref{cor: intSubSol}.  Suppose $Q$ has an interior maximum at $(p_0, t_0)$.  Then
\[
0 \geq -\hat{C} +AF^{\bar{t}t} +BF^{i\bar{j}}((\alpha_{\hat{\phi}})_{\bar{j}j} -(\alpha_{\phi})_{\bar{j}i})
\]
Suppose that $|\mu| >R$ at the point $(p_0,t_0)$.  Then by Lemma~\ref{lem: SzCsub} we have
\[
0 \geq -\hat{C} + B\kappa_{0}\frac{1}{1+C(\eta)^2} >0
\]
where $C(\eta) = C(\eta_1, \eta_2)$ is the bound for $|\mu_n|$ (see Lemma~\ref{lem: AngleBasicProps}), and we have chosen $B = B'\hat{C}$ for $B'$ large depending only on $\kappa_0, \eta_1,\eta_2$.  If $|\mu| <R$ at the point $(p_0,t_0)$ then $F^{t\bar{t}}\geq \frac{1}{1+R^2}$, and so
\[
\begin{aligned}
0&\geq -\hat{C}+ A\frac{1}{1+R^2} + B \sum_{i}\frac{(\alpha_{\hat{\phi}})_{\bar{i}i} -\mu_i}{1+\mu_i^2}\\
& \geq -\hat{C} + A\frac{1}{1+R^2} - B|\alpha_{\hat{\phi}}|_{\hat{g}} - B(n+1) >0
\end{aligned}
\]
provided we choose $A= A'\hat{C}$ for $A'$ large depending only on $R, n, B'$ and $\alpha_{\hat{\phi}}$.  In particular, by Corollary~\ref{cor: intSubSol} the constant $A'$ can be chosen uniformly.  Similar arguments using $-\del_u \phi, \pm \del_{v}\phi$ prove
\begin{lem}
There exists a uniform constant $C$ so that
\[
|\nabla_t \phi|_{\hat{g}} \leq C(1+\sup_{X_{\epsilon}}|\nabla_{t} h|_{\hat{g}}) + \sup_{\del X_{\epsilon}}|\nabla_{t} \phi|_{\hat{g}}
\]
\end{lem}
 
It remains only to estimate $|\nabla_{t}\phi|$ on the boundary. Since $\phi$ is $S^1$ invariant, it suffices to consider the boundary derivative with respect to the coordinate $r=|t|$.  Consider the boundary $|t|=\epsilon$.  Then $\phi \geq \hat{\phi}_0$ in $\mathcal{X}_{\epsilon}$, and $\phi = \hat{\phi}_0$ on $\del \mathcal{X}_{\epsilon}$ we have
\[
\frac{\del}{\del r}\bigg|_{r=\epsilon}\phi \leq \frac{\del}{\del r}\bigg|_{r=\epsilon} \hat{\phi}_0 \leq \frac{C}{\epsilon}
\]
by Corollary~\ref{cor: intSubSol}.  Similarly,
\[
 \frac{\del}{\del r}\bigg|_{r=e^{-1}\epsilon}\phi \geq  \frac{\del}{\del r}\bigg|_{r=e^{-1}\epsilon} \hat{\phi}_{1} \geq -\frac{C}{\epsilon}.
\]
For the remaining estimates we construct barriers from above.  To estimate near $\{r= \epsilon\}$ consider
\[
\psi_0 = \phi_0- A_0(|t|^2-\epsilon^2) - C_0\log\left(\frac{|t|^2}{\epsilon^2}\right).
\]
Clearly $\psi_0 = \phi_0 = \phi$ on $\{r=\epsilon\}$.  Furthermore we have
\[
\Delta_{\hat{\omega}}\psi_0 = \Delta_{\omega}\phi_0 - A_0
\]
and so we can choose $A_0$ large depending only on $\| \phi_0 \|_{C^{2}(X,\omega)}$ so that $\Delta_{\hat{\omega}}\psi_0 \leq 0$.  We next choose $C_0$ large depending only on $A_0, \|\phi_0-\phi_1\|_{L^{\infty}(X)}$ so that on $\{r=e^{-1}\epsilon\}$ we have
\[
\psi_0\bigg|_{|t|=e^{-1}\epsilon} = \phi_0+A_0(1-e^{-2})\epsilon^2+2C_0 \geq \phi_1.
\]
By Lemma~\ref{lem: AngleBasicProps} we have
\[ 
\Delta_{\hat{\omega}}\psi_0 \leq 0 < \Delta_{\hat{\omega}}\phi
\]
and $\psi_0 \geq \phi$ on the boundary, with equality when $|t|=\epsilon$.  By the maximum principle we obtain
\[
\frac{\del}{\del r}\bigg|_{r=\epsilon}\phi \geq \frac{\del}{\del r}\bigg|_{r=\epsilon}\psi_0 \geq -\frac{C}{\epsilon}
\]
for a uniform constant $C$.  Similar estimates work near $\{r=e^{-1}\epsilon\}$ to prove estimates near $\{r=e^{-1}\epsilon\}$.  We have therefore proved
\begin{thm}\label{thm: gradEst}
Suppose $\phi$ solves~\eqref{eq: mainEquation} for boundary data $\phi_i \in \mathcal{H}$ for $i=0,1$, and with $\phi_0, \phi_1, h$ satisfying~\eqref{eq: strucAss1},~\eqref{eq: strucAss2}.  Then there exists a uniform constant $C$ so that
\[
|\nabla^{X}\phi|_{\hat{g}} \leq C, \qquad |\nabla_{t} \phi|_{\hat{g}} \leq C\left(1+ \sup_{X_{\epsilon}}|\nabla_{t} h|_{\hat{g}})  + \frac{1}{\epsilon}\right)
\]
\end{thm}

\section{Interior $C^2$ estimates}\label{sec: C2est}

This section comprises the heart of the analysis towards proving the existence of geodesics in the space $\mathcal{H}$.  The goal is to prove the following theorem
\begin{thm}\label{thm: C2est}
Suppose $\phi$ solves~\eqref{eq: mainEquation} for boundary data $\phi_i \in \mathcal{H}$ for $i=0,1$, and with $\phi_0, \phi_1, h$ satisfying~\eqref{eq: strucAss1},~\eqref{eq: strucAss2}.  Then there exists a uniform constant $C$ so that
\begin{equation}\label{eq: thmSpaceC2}
|\nabla^{X}\overline{\nabla^{X}}\phi|_{\hat{\omega}} \leq C
\end{equation}
\begin{equation}
|\nabla^{X}\nabla_{\bar{t}}\phi| \leq C\left(1+ \|h_{\bar{t}t}\|_{L^{\infty}(\mathcal{X}_{\epsilon})} + \|h_{t}\|^2_{L^{\infty}(\mathcal{X}_{\epsilon})} + \sup_{\del \mathcal{X}_{\epsilon}}|\phi_{\bar{t}t}|_{\hat{g}}\right)^{\frac{1}{2}}
\end{equation}
\begin{equation}
|\nabla_{t}\nabla_{\bar{t}}\phi|_{\hat{g}} \leq C\left(1+ \|h_{\bar{t}t}\|_{L^{\infty}(\mathcal{X}_{\epsilon})} + \|h_{t}\|^2_{L^{\infty}(\mathcal{X}_{\epsilon})}+ \sup_{\del \mathcal{X}_{\epsilon}}|\phi_{\bar{t}t}|_{\hat{g}}\right)
\end{equation}
\end{thm}

The uniform estimate for the spatial $C^2$ norm,~\eqref{eq: thmSpaceC2}, is the most difficult of the three estimates.  Let us briefly recall what is known in this direction.  In joint work with Jacob \cite{CJY} the authors proved a $C^2$ estimate for solutions of $F=h$, when $h: X \rightarrow ((n-1)\frac{\pi}{2}, (n+1)\frac{\pi}{2})$ on compact manifolds without boundary, provided a subsolution exists.  In the current setting the same estimate works to prove an interior $C^2$ estimate of the form
\[
|\nabla \overline{\nabla} \phi|_{\hat{g}} \leq C(1+\sup_{\mathcal{X}_{\epsilon}} |\nabla \phi|^{2}_{\hat{g}}) + \sup_{\del \mathcal{X}_{\epsilon}}|\nabla\bar{\nabla}\phi|_{\hat{g}}
\]
Ignoring the troublesome boundary term, Theorem~\ref{thm: gradEst} gives the bound $|\nabla \phi|^{2}_{\hat{g}} \leq \frac{1}{\epsilon^2}$, and this bound is saturated if $\phi_0\ne \phi_1$.  Nevertheless, the above estimate would be good enough to prove the existence of solutions to~\eqref{eq: mainEquation} on $\mathcal{X}_{\epsilon}$, but not good enough to deduce regularity of the rescaled solutions as $\epsilon \rightarrow 0$.

In order to prove a uniform estimate for the spatial $C^2$ norm, we will apply the maximum principle to a quantity involving the largest {\em spatial} eigenvalue.  That is, for each point $(p,t) \in \mathcal{X}_{\epsilon}$ define $\lambda_1(p,t)$ to be the largest eigenvalue of $(\alpha + \DDb \phi)|_{X}$ with respect to $\omega$.  We will bound this quantity from above.  As a first step, we need to compute two derivatives of $\lambda_1$.
\begin{prop}
Suppose $\mathcal{E}$ is a smooth section of ${\rm End}(T\mathcal{X}_{\epsilon})$ which is hermitian with respect to $\hat{g}$.  Define $\chi = \mathcal{E}|_{TX} = \pi \mathcal{E} \pi$, where $\pi$ is the orthogonal projection to $TX \subset T\mathcal{X}_{\epsilon}$ defined by $\hat{g}$.  Suppose $\lambda_1 > \lambda_2 > \cdots > \lambda_n$ are the eigenvalues of $\chi$ at $(p,t) \in X_{\epsilon}$, and let $e_1, \ldots, e_n$ be the corresponding unit length, orthonormal eigenvectors. Let $(z_{i})$ be holomorphic coordinates on $\mathcal{X}_{\epsilon}$ centered at $(p_0,t_0)$.  Then at $(p_0,t_0)$ we have
\[
\nabla_i\lambda_{\alpha} = \langle (\nabla_{i}\E)e_{\alpha}, e_{\alpha}\rangle, \qquad \nabla_{\bar{i}}\lambda_{\alpha}  \langle (\nabla_{\bar{i}}\E)e_\alpha, e_\alpha\rangle
\]
and,
\[
\begin{aligned}
\nabla_{i}\nabla_{\bar{j}} \lambda_{\alpha} &=  \langle (\nabla_i\nabla_{\bar{j}}\E)e_\alpha, e_\alpha\rangle\\
&\quad + \sum_{\beta \ne \alpha} \frac{\langle (\nabla_i \E)e_{\alpha}, e_{\beta}\rangle\overline{\langle(\nabla_j\E)e_{\alpha},e_{\beta}\rangle}}{\lambda_{\alpha}-\lambda_{\beta}}+\frac{\langle (\nabla_{\bar{j}} \E)e_{\alpha}, e_{\beta}\rangle\overline{\langle(\nabla_{\bar{i}}\E)e_{\alpha},e_{\beta}\rangle}}{\lambda_{\alpha}-\lambda_{\beta}}
\end{aligned}
\]
where we view $e_\alpha$ as vectors in $T\mathcal{X}_{\epsilon}$ by the inclusion $TX \hookrightarrow T\mathcal{X}_{\epsilon}$.
\end{prop}
\begin{proof}
First note that, by assumption the eigenvalues $\lambda_\alpha$ are smooth functions near $(p_0,t_0)$, and we can find smooth spatial vector fields denoted $e_\alpha$ which are sections of $TX\subset T\mathcal{X}_{\epsilon}$, so that
\[
\chi e_\alpha(p,t) = \lambda_\alpha(p,t) e_\alpha, \qquad \|e_\alpha(p,t)\|_{\hat{g}} =1
\]
and clearly we have $e_\alpha(p_0,t_0) = e_\alpha$.  Now, since $e_\alpha =\pi e_\alpha$, we can write
\begin{equation}\label{eq: EndSpaceTime}
\E e_\alpha = \chi e_\alpha + (1-\pi)\E e_\alpha
\end{equation}
and so $\lambda_\alpha = \langle \E e_\alpha, e_\alpha\rangle$.  We now differentiate this equation to get
\[
\begin{aligned}
\nabla_{i}\lambda_{\alpha} &= \langle (\nabla_{i}\E)e_{\alpha}, e_{\alpha}\rangle + \langle \E (\nabla_{i} e_{\alpha}), e_{\alpha}\rangle + \langle \E  e_{\alpha}, \nabla_{\bar{i}}e_{\alpha}\rangle\\
&=\langle (\nabla_{i}\E)e_{\alpha}, e_{\alpha}\rangle + \langle  (\nabla_{i} e_{\alpha}), \E e_{\alpha}\rangle +\lambda_{\alpha} \langle  e_{\alpha}, \nabla_{\bar{i}}e_{\alpha}\rangle\\
&= \langle (\nabla_{i}\E)e_{\alpha}, e_{\alpha}\rangle +2\lambda_{\alpha} \langle  (\nabla_{i} e_{\alpha}),  e_{\alpha}\rangle \\
&= \langle (\nabla_{i}\E)e_{\alpha}, e_{\alpha}\rangle
\end{aligned}
\]
where we used that $\E$ is hermitian, and $\|e_{\alpha}\| =1$, so that $\langle \nabla_{i}e_{\alpha}, e_{\alpha}\rangle =0$. Similarly we have
\[
\nabla_{\bar{i}}\lambda_\alpha =  \langle (\nabla_{\bar{i}}\E)e_\alpha, e_\alpha\rangle.
\]
Before proceeding, let us remark that since $\E$ is hermitian, for any vectors $V,W \in T\mathcal{X}_{\epsilon}$ there holds
\begin{equation}\label{eq: HermDurr}
\langle (\nabla_{i}\E)V, W\rangle = \langle V, (\nabla_{\bar{i}}\E)W\rangle 
\end{equation}
and similarly for barred indices.  Next we compute
\begin{equation}\label{eq: evSecondDeriv1}
\nabla_{i}\nabla_{\bar{j}} \lambda_{\alpha} =  \langle (\nabla_i\nabla_{\bar{j}}\E)e_\alpha, e_\alpha\rangle +  \langle (\nabla_{\bar{j}}\E)(\nabla_ie_\alpha), e_\alpha\rangle +\langle (\nabla_{\bar{j}}\E)e_\alpha, \nabla_{\bar{i}}e_\alpha\rangle.
\end{equation}
At the point $(p_0,t_0)$ we claim that 
\[
\nabla_{i}e_\alpha = \sum_{\beta\ne \alpha} a^{\alpha}_{i\beta}e_\beta,\qquad \nabla_{\bar{i}}e_\alpha = \sum_{\beta\ne \alpha} a^{\alpha}_{\bar{i}\beta}e_\beta.
\]
Let us explain how to see this for the first expression, with the second expression being treated in the same way.  Let $\del_t$ denote the vector field generated by the time direction.  Since $e_1,\ldots, e_n$ span the orthogonal complement of ${\rm span}\{\del_t\} \subset T\mathcal{X}_{\epsilon}$, it suffices to show that $\langle \nabla_{i}e_{\alpha}, \del_t\rangle =0$.  To do this we differentiate the equation $\langle e_\alpha, \del_t\rangle =0$ to get
\[
0 = \langle \nabla_i e_{\alpha} ,\del_t \rangle + \langle  e_{\alpha} ,\nabla_{\bar{i}}\del_t \rangle
\]
and use that by Lemma~\ref{lem: productMetric}, $\del_t$ is parallel. In fact, in this case one can use that $\del_t$ is a holomorphic vector field, but this does not work to prove the analogous claim for $\nabla_{\bar{i}}e_{\alpha}$.  A similar computation shows that $\langle \nabla_{\bar{i}}e_{\alpha}, \del_t\rangle =0$.  Finally, the fact that $e_i$ does not appear in the sum follows from $\|e_i\|=1$.  Now we can solve for the $a^{\alpha}_{i\beta}$.  We differentiate~\eqref{eq: EndSpaceTime} to get
\[
(\nabla_i \E)e_{\alpha} + \E(\nabla_{i}e_{\alpha}) = \nabla_i\lambda_{\alpha} e_{\alpha} + \lambda_{\alpha}\nabla_i e_{\alpha} + \nabla_i((1-\pi)\E e_{\alpha}).
\]
Taking the inner product with $e_{\beta}$ for $\beta \ne \alpha$, and using that $\E$ is hermitian gives
\[
\begin{aligned}
\langle (\nabla_i \E)e_{\alpha} , e_{\beta}\rangle + \lambda_{\beta}a^{\alpha}_{i\beta} &= \lambda_{\alpha}\langle\nabla_i e_{\alpha}, e_{\beta}\rangle + \langle \nabla_i((1-\pi)\E e_{\alpha}), e_{\beta}\rangle\\
&=\lambda_{\alpha}a^{\alpha}_{i\beta} + \langle ((1-\pi)\E e_{\alpha}), \nabla_{\bar i}e_{\beta}\rangle\\
&=\lambda_{\alpha}a^{\alpha}_{i\beta}
\end{aligned}
\]
where in the last two lines we used that $e_{\beta}, \nabla_{\bar{i}}e_{\beta}$ are orthogonal to ${\rm span}\{\del_t\} \ni (1-\pi)\E e_{\alpha}$.  So
\[
(\lambda_{\alpha}-\lambda_{\beta})a^{\alpha}_{i\beta} = \langle (\nabla_i \E)e_{\alpha} , e_{\beta}\rangle.
\]
 Similarly, we have
\[
(\lambda_{\alpha}-\lambda_{\beta})a^{\alpha}_{\bar{i}\beta} = \langle (\nabla_{\bar{i}} \E)e_{\alpha} , e_{\beta}\rangle.
\]
Plugging this into~\eqref{eq: evSecondDeriv1}, using~\eqref{eq: HermDurr} and doing some algebra we obtain
\[
\begin{aligned}
\nabla_{i}\nabla_{\bar{j}} \lambda_{\alpha} &=  \langle (\nabla_i\nabla_{\bar{j}}\E)e_\alpha, e_\alpha\rangle\\
&\quad + \sum_{\beta \ne \alpha} \frac{\langle (\nabla_i \E)e_{\alpha}, e_{\beta}\rangle\overline{\langle(\nabla_j\E)e_{\alpha},e_{\beta}\rangle}}{\lambda_{\alpha}-\lambda_{\beta}}+\frac{\langle (\nabla_{\bar{j}} \E)e_{\alpha}, e_{\beta}\rangle\overline{\langle(\nabla_{\bar{i}}\E)e_{\alpha},e_{\beta}\rangle}}{\lambda_{\alpha}-\lambda_{\beta}}
\end{aligned}
\]
By noting that $\lambda_1$ is a smooth function provided $\lambda_1 >\lambda_2$ we have
\[
\nabla_{i}\nabla_{\bar{i}} \lambda_{1} =  \langle (\nabla_i\nabla_{\bar{i}}\E)e_1, e_1\rangle + \sum_{\beta \ne 1} \frac{|\langle (\nabla_i \E)e_{1}, e_{\beta}\rangle|^{2}+|\langle (\nabla_{\bar{i}} \E)e_{1}, e_{\beta}\rangle|^{2}}{\lambda_{1}-\lambda_{\beta}}.
\]
\end{proof}

Next we compute the linearized operator applied to the largest spatial eigenvalue $\lambda_1$ of $\E := \hat{\omega}^{-1}(\pi^{*}\alpha+ \DDb\phi)$.  First, we have to perturb this endomorphism to ensure that the largest spatial eigenvalue is smooth.  Fix a point $(p_0,t_0) \in \mathcal{X}_{\epsilon}$, and choose holomorphic normal coordinates $(z_0,\ldots,z_n)$ for $\hat{\omega}$ centered at $(p_0,t_0)$ so that $\E(p_0,t_0)$ is diagonal with eigenvalues $\mu_0\geq \mu_1 \geq \cdots \geq \mu_n$, and we may assume that $\mu_0 > \mu_n$.  Consider a matrix $B=(B^{i}_{j}) = B_{ii}\delta^{i}_{j}$ defined near $(p_0,t_0)$ with the property that $B_{00}=0= B_{nn} <B_{n-1n-1} <\cdots < B_{11}$.  Let $e_\beta$ be an $\hat{\omega}$ orthonormal frame of eigenvectors for $\E|_{TX}$ at $(p_0,t_0)$, with corresponding eigenvalues $\lambda_1 \geq \lambda_2 \geq \cdots \geq \lambda_n$, and assume that $\lambda_1 \gg \lambda_n$. As usual, we regard the $e_i$ as vectors in $(T\mathcal{X}_{\epsilon})_{p_0,t_0}$.  Furthermore, we make the following stipulation; if $\lambda_1 = \mu_0$, then $e_1$ is in the span of the vectors with eigenvalue $\mu_0$.  After possibly rotating our original coordinate system, we choose $e_1=\del_{z_0}$.  We choose the $B_{ii}$ in the following way.  We require that
\[
Be_1=0.
\]
If $e_1 =\del_{z_0}$ this is redundant, otherwise we have gained a single linear equation, and hence we can choose $B\in \mathbb{R}^{n+1}_{\geq 0}$ in the complement of at most $3$ hyperplanes.  On the orthogonal complement of ${\rm span}\{\del_{z_0},\del_{z_n}, e_1\}$ we clearly have that $B$ is positive definite. Extending $B$ to be constant in our local coordinate patch, we can view it as a local holomorphic section of ${\rm End}(T\mathcal{X}_{\epsilon})$.  Consider the endomorphism $\tilde{\E}:= \E-B$.  Clearly the eigenvalues of $\tilde{\E}$ are less or equal the eigenvalues of $\E$. Let $\tilde{\mu_{0}}\geq \tilde{\mu_{1}} \geq\cdots\geq\tilde{\mu_{n}} $ denote the eigenvalues of $\tilde{\E}$ at $(p_0,t_0)$, and $\tilde{\lambda}_{1}\geq \tilde{\lambda}_{2} \geq \cdots \geq \tilde{\lambda}_{n} $ denote the eigenvalues of $\tilde{\E}|_{TX}$.  Then we have
\[
\tilde{\mu_n} = \mu_n,\quad  \tilde{\mu_0} = \mu_0,\quad \tilde{\lambda}_1 = \lambda_1 \qquad \text{ at } (p_0, t_0).
\]
Furthermore, we have $\tilde{\beta}_{\alpha} < \lambda_{\beta}$ for all $\beta \ne 1, n$.  Since $\lambda_n \ll \lambda_1$, by assumption, this implies that $\tilde{\lambda}_{1} > \tilde{\lambda}_{2} \geq \cdots \geq \tilde{\lambda}_n$.  Thus $\tilde{\lambda}_{2}<\tilde{\lambda}_{1} \leq \lambda_1$ near $(p_0,t_0)$, with $\tilde{\lambda}_{1} = \lambda_1$ at $(p_0,t_0)$.  This is the desired perturbation.  We now compute
\begin{equation}\label{eq: evLinOp1}
\begin{aligned}
&F^{i\bar{j}}\nabla_{i}\nabla_{\bar{j}}\tilde{\lambda}_1 = \sum_{i=0}^{n}\frac{1}{1+\mu_i^2}\nabla_i\nabla_{\bar{i}}\tilde{\lambda}_1\\
&= \sum_{i}\frac{1}{1+\mu_i^2}\langle (\nabla_i\nabla_{\bar{i}}\tilde{\E})e_1, e_1\rangle + \sum_{i=0}^{n}\sum_{\beta \ne 1} \frac{|\langle (\nabla_i \tilde{\E})e_{1}, e_{\beta}\rangle|^{2}+|\langle (\nabla_{\bar{i}} \tilde{\E})e_{1}, e_{\beta}\rangle|^{2}}{(\tilde{\lambda}_{1}-\tilde{\lambda}_{\beta})(1+\mu_i^2)}.
\end{aligned}
\end{equation}
Now at $(p_0,t_0)$
\[
\nabla_{i}\tilde{\E}^{k}_{p} = \nabla_i(\alpha_{\phi})_{\bar{k}p},
\]
\[
\nabla_i\nabla_{\bar{j}}\tilde{\E}^{k}_{p} = \nabla_{i}\nabla_{\bar{j}}\left((\hat{g})^{k\bar{\ell}}(\alpha_{\phi})_{\bar{\ell}p} - B^{k}_{p}\right)=(\hat{g})^{k\bar{\ell}}\nabla_i\nabla_{\bar{j}}(\alpha_{\phi})_{\bar{\ell}p},
\]
since $B$ is a local holomorphic section.  At $(p_0,t_0)$ write $e_\beta = V_{\beta}^{j} \del_{z_j}$, for complex numbers $V_{\beta}^{j}$, and $1\leq \beta \leq n$. For simplicity we denote
\[
V_\beta = V_{\beta}^{j} \del_{z_j}
\]
regarded as vector fields defined in an open neighbourhood of $(p_0,t_0)$.  Then at $(p_0,t_0)$ we have
\begin{equation}\label{eq: diffEalpha}
\sum_{i}\frac{1}{1+\mu_i^2}\langle (\nabla_i\nabla_{\bar{i}}\tilde{\E})e_1, e_1\rangle = V_{1}^{p}\overline{V_{1}^{k}}\sum_{i}\frac{1}{1+\mu_i^2}\nabla_i\nabla_{\bar{i}}(\alpha_\phi)_{\bar{k}p}.
\end{equation}
Using that $V_1$ has constant coefficients, and $\alpha_{\phi}$ is closed we compute
\begin{equation}\label{eq: 2diffCommute}
\begin{aligned}
\nabla_{\bar{V_1}}\nabla_{V_1}(\alpha_\phi)_{\bar{j}i} &= \overline{V_1^{k}}V_{1}^{p}\nabla_{\bar{k}}\nabla_p(\alpha_{\phi})_{\bar{j}i}\\
&=\overline{V_1^{k}}V_{1}^{p}\nabla_{\bar{k}}\nabla_i(\alpha_{\phi})_{\bar{j}p}\\
&=\overline{V_1^{k}}V_{1}^{p}\left(\nabla_i\nabla_{\bar{k}}(\alpha_{\phi})_{\bar{j}p} + R_{\bar{k}i}\,^{s}\,_{p}(\alpha_{\phi})_{\bar{j}s} - R_{\bar{k}i\bar{j}}\,^{\bar{s}}(\alpha_{\phi})_{\bar{s}p}\right)\\
&=\overline{V_1^{k}}V_{1}^{p}\left(\nabla_i\nabla_{\bar{j}}(\alpha_{\phi})_{\bar{k}p} + R_{\bar{k}i}\,^{s}\,_{p}(\alpha_{\phi})_{\bar{j}s} - R_{\bar{k}i\bar{j}}\,^{\bar{s}}(\alpha_{\phi})_{\bar{s}p}\right)\\
&=\sum_{k,p=0}^{n}\overline{V_1^{k}}V_{1}^{p}\left(\nabla_i\nabla_{\bar{j}}(\alpha_{\phi})_{\bar{k}p} + R_{\bar{k}i}\,^{j}\,_{p}\mu_j - R_{\bar{k}i\bar{j}}\,^{\bar{p}}\mu_p\right).
\end{aligned}
\end{equation}
We combine equations~\eqref{eq: diffEalpha}~\eqref{eq: 2diffCommute} to obtain at $(p_0,t_0)$
\begin{equation}\label{eq: evLinOp2}
\begin{aligned}
\sum_{i}\frac{1}{1+\mu_i^2}&\langle (\nabla_i\nabla_{\bar{i}}\tilde{\E})e_1, e_1\rangle = \sum_{i=0}^{n}\frac{1}{1+\mu_i^2}\nabla_{\bar{V_1}}\nabla_{V_1}(\alpha_{\phi})_{\bar{i}i}\\
&\quad-\sum_{i,k,p=0}^{n}\frac{1}{1+\mu_i^2}R_{\bar{k}i}\,^{i}\,_p\mu_iV^{p}\overline{V^{k}} + \sum_{i,k,p}\frac{1}{1+\mu_i^2}R_{\bar{k}i\bar{i}}\,^{p}\mu_pV^{p}\overline{V^{k}}.
\end{aligned}
\end{equation}
Differentiate the equation $F(\E)=h$ in the $V_1$ direction to get
\[
\nabla_{V_1}h= F^{i\bar{j}}\nabla_{V_1}(\alpha_{\phi})_{\bar{j}i} = \sum_{i=0}^{n} \frac{1}{1+\mu_i^2}\nabla_{V_1}(\alpha_{\phi})_{\bar{i}i},
\]
\begin{equation}\label{eq:2diffEquation}
\begin{aligned}
\nabla_{\bar{V_1}}\nabla_{V_1}h&= F^{i\bar{j}}\nabla_{\bar{V_1}}\nabla_{V_1}(\alpha_{\phi})_{\bar{j}i} + F^{i\bar{j}, m\bar{\ell}}\nabla_{V_1}(\alpha_{\phi})_{\bar{j}i}\nabla_{\bar{V_1}}(\alpha_{\phi})_{\bar{\ell}m} \\
&=\sum_{i=0}^{n}\frac{1}{1+\mu_i^2}\nabla_{\bar{V_1}}\nabla_{V_1}(\alpha_{\phi})_{\bar{i}i} - \sum_{i,j=0}^{n}\frac{\mu_i+\mu_j}{(1+\mu_i^2)(1+\mu_j^2)}|\nabla_{V_1}(\alpha_{\phi})_{\bar{j}i}|^2.
\end{aligned}
\end{equation}
Before substituting~\eqref{eq: evLinOp2} into~\eqref{eq:2diffEquation}, we note the following easy, but useful lemma. 
\begin{lem}\label{lem: convComb}
In the above notation, at $(p_0,t_0)$ we have
\[
\lambda_1 = \sum_{i=0}^{n}|V_1^{i}|^2\mu_i, \qquad \sum_{i=0}^{n}|V_1^{i}|^2=1
\]
In particular, $\lambda_1$ is a convex combination of the $\mu_i$, $\lambda_1 \leq \mu_0$, with equality if and only if $V_1$ is in the span of the eigenvectors with eigenvalue $\mu_0$.
\end{lem}
\begin{proof}
The proof is just a consequence of writing the equations $\lambda_1 = \langle \E e_1, e_1\rangle$, and $\|e_1\|^2=1$ in our local coordinate system, and using the definition of $V_1$.
\end{proof}

We note the following simple corollary
\begin{cor}\label{cor: convComb}
For all $j \ne n$ we have
\[
|V_1^{j}|^2\mu_j \leq \lambda_1 +\max\{0,-\mu_n\} \leq \lambda_1 + C(\eta_1)
\]
\end{cor}
\begin{proof}
Just observe that, by Lemma~\ref{lem: AngleBasicProps} the $i=n$ term is the only term is the only possible negative contribution to the sum
\[
\lambda_1=  \sum_{i=0}^{n}|V_1^{i}|^2\mu_i.
\]
Since $\sum_{i=0}^{n}|V_1^{i}|^2=1$ we can rearrange, and apply Lemma~\ref{lem: AngleBasicProps} to conclude.
\end{proof}
Let $C_1>0$ be a two-sided bound for the sectional curvature of $\hat{\omega}$, which by Lemma~\ref{lem: productMetric} depends only on a bound for the sectional curvature of $(X,\omega)$, and is therefore a uniform constant. Then we have
\begin{equation}\label{eq: curvTermBd1}
\begin{aligned}
\sum_{k,p}R_{\bar{k}i\bar{i}}\,^{p}\mu_pV^{p}\overline{V^{k}} \geq -C_1\sum_{p=0}^{n}|\mu_p||V^{p}|^2 &= -C_1\sum_{p=0}^{n}\mu_{p}|V_1^{p}|^2 -C_1(|\mu_n|-\mu_n)|V_1^{n}|^2\\
&=-C_1\lambda_1-C_2
\end{aligned}
\end{equation}
where we used Lemma~\ref{lem: convComb} and that $\mu_{n} > -C(\eta_1)$ by Lemma~\ref{lem: AngleBasicProps}.  For the remaining curvature term we use that $|\frac{\mu_i}{1+\mu_i^2}|\leq 1$ together with $\sum_i |V_{1}^{i}|^2=1$ to get
\begin{equation}\label{eq: curvTermBd2}
\sum_{i,k,p=0}^{n}\frac{1}{1+\mu_i^2}R_{\bar{k}i}\,^{i}\,_p\mu_iV^{p}\overline{V^{k}} \leq C_1(n+1).
\end{equation}
Returning to equation~\eqref{eq: evLinOp2}, substituting the bounds~\eqref{eq: curvTermBd1},\eqref{eq: curvTermBd2}, and equation~\eqref{eq:2diffEquation} we obtain
\[
\sum_{i}\frac{1}{1+\mu_i^2}\langle (\nabla_i\nabla_{\bar{i}}\tilde{\E})e_1, e_1\rangle \geq \sum_{i,j=0}^{n}\frac{\mu_i+\mu_j}{(1+\mu_i^2)(1+\mu_j^2)}|\nabla_{V_1}(\alpha_{\phi})_{\bar{j}i}|^2-C\lambda_1 - C
\]
for a constant $C$ depending only on $n, \eta$, and bound for $\nabla_{\bar{V_1}}\nabla_{V_1}h$ and $(X,\omega)$.  Note that, since $V_1$ is spatial, the constant $C$ is uniform.
Finally, by noting that $\lambda_1=\tilde{\lambda}_1 > \tilde{\lambda}_{\beta}$ for $\beta \ne 1$ we arrive at

\begin{prop}\label{prop: firstBd}
In the above notation, at the point $(p_0,t_0)$ we have
\begin{equation}\label{eq: firstBdProp}
\begin{aligned}
F^{i\bar{j}}\nabla_{i}\nabla_{\bar{j}}\log(\tilde{\lambda}_1) &\geq \frac{1}{\lambda_1}\sum_{i,j=0}^{n}\frac{\mu_i+\mu_j}{(1+\mu_i^2)(1+\mu_j^2)}|\nabla_{V_1}(\alpha_{\phi})_{\bar{j}i}|^2\\
&\quad -\frac{1}{\lambda_1^2} \sum_{i=0}^{n}\frac{1}{1+\mu_i^2}|\nabla_i(\alpha_{\phi})_{\bar{V_1}V_1}|^2 -C
\end{aligned}
\end{equation}
for a uniform constant $C$.
\end{prop}

In order to apply the maximum principle, we need a lower bound for the quantity appearing on the right hand side of~\eqref{eq: firstBdProp}.  Note that $\nabla_i(\alpha_{\phi})_{V_1\bar{V_1}} = \nabla_{V_1}(\alpha_{\phi})_{\bar{V_1}i}$ since $\alpha_{\phi}$ is closed and $V_1$ has constant coefficients.  Thus, all the negative terms appearing in
\begin{equation}\label{eq: GradTerm}
 \sum_{i=0}^{n}\frac{1}{1+\mu_i^2}|\nabla_i\alpha_{\bar{V_1}V_1}|^2
 \end{equation}
 have partners appearing in the sum
\begin{equation}\label{eq: HessTerm}
\frac{1}{\lambda_1}\sum_{i,j=0}^{n}\frac{\mu_i+\mu_j}{(1+\mu_i^2)(1+\mu_j^2)}|\nabla_{V_1}(\alpha_{\phi})_{\bar{j}i}|^2.
\end{equation}
There are three significant difficulties in attaining the estimate we want. The first is that the $i=j=n$ term appear in~\eqref{eq: HessTerm} will be non-positive when $\mu_n\leq0$, and even when $\mu_n>0$, the contribution can be arbitrarily small.  Of course, by Lemma~\ref{lem: AngleBasicProps}, we know that $-e^{-AF}$ is concave.  However, this only implies that the bad term $i=j=n$ term appearing in~\eqref{eq: HessTerm} can be controlled at the expense of {\em all} the terms with $i=j \ne n$ appearing in~\eqref{eq: HessTerm}.  Unfortunately, terms like $|\nabla_{V_1}\alpha_{i\bar{i}}|^2$ evidently appear in $|\nabla_i\alpha_{V_1\bar{V_1}}|^2$.  Thus, invoking concavity of $-e^{-AF}$ leaves us with no way to control~\eqref{eq: GradTerm}.  This is in stark contrast to the case of concave elliptic operators \cite{Sz, Guan1}, where the terms in~\eqref{eq: HessTerm} for $i\ne j$ are not used.

The second significant difficulty occurs when trying to invoke that the gradient vanishes at the maximum point.  More precisely, a natural approach to controlling~\eqref{eq: GradTerm}, is to apply the fact that the gradient of our test function vanishes at the maximum.  Our test function will be of the form
\[
\tilde{Q} := \log(\tilde{\lambda}_1) + H(\phi)
\]
for a specially chosen function $H$.  Thus, at a maximum we will have
\[
\frac{\nabla_i (\alpha_{\phi})_{\bar{V_1}V_1}}{\lambda_1} =- H'(\phi) \nabla_i\phi.
\]
Such an argument is essentially doomed, since we do not control $|\nabla_i\phi|_{\hat{g}}^2$ uniformly.  In particular, the vector $\del_{z_i}$ could have a component pointing in the time direction, and by Theorem~\ref{thm: gradEst}, such a component would contribute a term of order $\epsilon^{-2}$ to the estimate.  Again, this is in contrast to the case of concave elliptic operators on compact manifolds \cite{Sz, Guan1}, where this argument is used repeatedly to obtain a $C^2$ estimate \cite{Sz, Guan1}.

To see the final significant problem, note that~\eqref{eq: GradTerm} contains a term like $\frac{1}{(1+\mu_0^2)\lambda_1^2}|\nabla_{V_1}\alpha_{\bar{0}0}|^2$, and the corresponding term in~\eqref{eq: HessTerm} is  $\frac{2\mu_0}{(1+\mu_0^2)^2\lambda_1}|\nabla_{V_1}\alpha_{\bar{0}0}|^2$.  Since $\lambda_1$ is the largest {\em spatial} eigenvalue, it can (and indeed must) happen that $\mu_0 \gg \lambda_1$, so that
\[
\frac{2\mu_0}{(1+\mu_0^2)^2\lambda_1}\ll \frac{1}{(1+\mu_0^2)\lambda_1^2}.
\]
Again, when studying concave elliptic operators on compact manifolds \cite{Sz, Guan1} we have the $\mu_0 = \lambda_1$, and the two terms above are easily comparable.

There are two key points in the arguments that follow.  The first is to split the eigenvalues $\mu_i$ according to whether they are large, or small, compared to $\lambda_1$.  Roughly speaking, if $\mu_i$ is small relative to $\lambda_1$ we expect that $\del_{z_i}$ can contain at most a very small component pointing in the temporal direction.  On the other hand, if $\mu_i$ is very large relative to $\lambda_1$ the we must make efficient use of the coefficients $V_1^{i}$ that appear in~\eqref{eq: GradTerm}, together with the following trivial observation from Lemma~\ref{lem: convComb}: if $\mu_i \gg \lambda_1$, then $|V_{1}^{i}|^2 \ll 1$.

Before beginning the estimate, let us fix a suitable background form, and some notation.  Let $\hat{\phi}$ be one of the subsolutions constructed in Corollary~\ref{cor: intSubSol}, and write
\[
\hat{\alpha} = \pi_{X}^{*}\alpha +\DDb\hat{\phi}
\]
Then we can write $\alpha_{\phi} = \hat{\alpha}+\DDb(\phi - \hat{\phi})$.  To simplify the notation, let us denote $\alpha = \alpha_{\phi}$.

As a first step, we are going to address the first issue outlined above by estimating the term $\nabla_{V_1}\alpha_{\bar{n}n}$, assuming $\mu_n<0$.  We first observe that $\alpha_{\bar{n}n} =\mu_n$ at $(p_0,t_0)$, and so by differentiating the equation we have
\[
\frac{1}{1+\mu_n^2}\nabla_{V_1}\alpha_{\bar{n}n} = \nabla_{V_1}h-\sum_{k<n} \frac{1}{1+\mu_k^2}\nabla_{V_1}\alpha_{\bar{k}k}. 
\]
By Cauchy-Schwarz we have
\begin{equation}\label{eq: CSmun}
\frac{|\nabla_{V_1}\alpha_{\bar{n}n}|^2}{(1+\mu_{n}^2)^2} \leq (1+\delta_1^{-1})|\nabla_{V_1}h|^2 + (1+\delta_1)\bigg|\sum_{k<n}\frac{1}{1+\mu_k^2}\nabla_{V_1}\alpha_{\bar{k}k}\bigg|^2
\end{equation}
for a constant $\delta_1$ to be determined.  Fix a constant $\frac{1}{2}>\delta_0 >0$.  This constants will be determined in the course of the proof.  We say that $\mu_{\beta}$ is big if $\mu_{\beta}\geq \delta_0 \lambda_1$, and small if $\mu_{\beta}< \delta_0 \lambda_1$.  Define
\[
B= \{ k \in \{0,1,\ldots, n\} : \mu_{k}\geq \delta_0 \lambda_1\}, \qquad S= \{ k \in \{0,1,\ldots, n\} : \mu_{k}< \delta_0 \lambda_1\}.
\]
Clearly $\mu_{n} \in S$ if $\mu_n <0$ and also if $\lambda_1$ is sufficiently large, depending on $\delta_0$.  We write
\[
\begin{aligned}
\bigg|\sum_{k<n}\frac{1}{1+\mu_k^2}\nabla_{V_1}\alpha_{\bar{k}k}\bigg|^2 &=\bigg|\sum_{k \in B}\frac{\nabla_{V_1}\alpha_{\bar{k}k}}{(1+\mu_k^2)}\bigg|^2 +\bigg|\sum_{k \in S,\, k<n}\frac{\nabla_{V_1}\alpha_{\bar{k}k}}{(1+\mu_k^2)}\bigg|^2\\
&\quad + 2{\rm Re}\left(\sum_{k\in B,\,\ell\in S,\,\ell <n}\frac{\nabla_{V_1}\alpha_{\bar{k}k}\overline{\nabla_{V_1}\alpha_{\bar{\ell}\ell}}}{(1+\mu_k^2)(1+\mu_{\ell}^2)}\right).
\end{aligned}
\]
Each sum needs to be estimated differently.  Starting with the last sum over $S$, using Cauchy-Schwarz we estimate
\[
\bigg|\sum_{k \in S,\,k<n}\frac{\nabla_{V_1}\alpha_{\bar{k}k}}{(1+\mu_k^2)}\bigg|^2 \leq\left(\sum_{k\in S, k<n}\frac{\mu_k |\nabla_{V_1}\alpha_{\bar{k}k}|^{2}}{(1+\mu_{k}^2)^2}\right)\cdot\left(\sum_{j\in S, j<n} \frac{1}{\mu_j}\right)
\]
where we used that $\mu_j>0$ for $j\ne n$ by Lemma~\ref{lem: AngleBasicProps}.  Since $\sum_{i=0}^{n}\arctan(\mu_i) > (n-1)\frac{\pi}{2} +\eta_1$, and $\mu_n <0$ we have
\begin{equation}\label{eq: hessEstSmall1}
\sum_{j\in S, j <n}\frac{1}{\mu_j} + \frac{1}{\mu_n} < -\tan(\eta_1)
\end{equation}
by Lemma~\ref{lem: AngleBasicProps}, together with $\mu_{\ell}>0$ if $\ell \in B$.  This estimate is clearly not accurate if $|\mu_{n}|$ is small.  In that case we instead have the estimate
\begin{equation}\label{eq: hessEstSmall2}
\sum_{j\in S, j <n}\frac{1}{\mu_j} + \frac{1}{\mu_n} < \frac{n+1}{\tan(\eta_1)} + \frac{1}{\mu_n}
\end{equation}
where we used that $\mu_i > \tan(\eta_1)$ for all $i<n$ by Lemma~\ref{lem: AngleBasicProps}.  Estimate~\eqref{eq: hessEstSmall2} is more precise when $|\mu_n|<\frac{\tan(\eta_1)}{(n+1)+(\tan(\eta_1))^2}$, and less precise otherwise.  By considering each case separately it follows that
\[
\mu_n\left(\sum_{j\in S, j<n} \frac{1}{\mu_j}\right) \geq -\frac{(n+1)}{(n+1) + (\tan(\eta_1))^2}.
\]
The final estimate we obtain is
\begin{equation}\label{eq: hessEstSTerms}
\mu_n\bigg|\sum_{k \in S,\, k<n }\frac{\nabla_{V_1}\alpha_{\bar{k}k}}{(1+\mu_k^2)}\bigg|^2 \geq \left( -\frac{(n+1)}{(n+1) + (\tan(\eta_1))^2}\right)\cdot \left(\sum_{k\in S,\, k<n}\frac{\mu_k |\nabla_{V_1}\alpha_{\bar{k}k}|^{2}}{(1+\mu_{k}^2)^2}\right).
\end{equation}
Next we estimate the cross terms.  
\[
\begin{aligned}
&\bigg|{\rm Re}\left(\sum_{k\in B,\,\ell\in S,\, \ell<n}\frac{\nabla_{V_1}\alpha_{\bar{k}k}\overline{\nabla_{V_1}\alpha_{\bar{\ell}\ell}}}{(1+\mu_k^2)(1+\mu_{\ell}^2)}\right)\bigg| \leq\sum_{k\in B,\,\ell\in S,\,\ell <n}\frac{|\nabla_{V_1}\alpha_{\bar{k}k}||\nabla_{V_1}\alpha_{\bar{\ell}\ell}|}{(1+\mu_k^2)(1+\mu_{\ell}^2)}\\
&\leq \sum_{k\in B,\,\ell\in S,\,\ell <n}\frac{(n+1)|\nabla_{V_1}\alpha_{\bar{k}k}|^2}{\epsilon_1 \mu_{\ell}(1+\mu_k^2)^2}  + \sum_{k\in B,\,\ell\in S,\,\ell <n}\frac{\epsilon_1\mu_{\ell}|\nabla_{V_1}\alpha_{\bar{\ell}\ell}|^2}{(n+1)(1+\mu_\ell^2)^2}\\
&\leq \frac{(n+1)^2}{\epsilon_1 \tan(\eta_1)}\sum_{k\in B}\frac{|\nabla_{V_1}\alpha_{\bar{k}k}|^2}{(1+\mu_k^2)^2} + \epsilon_1\sum_{\ell\in S ,\ell <n}\frac{\mu_{\ell}|\nabla_{V_1}\alpha_{\bar{\ell}\ell}|^2}{(1+\mu_\ell^2)^2},
\end{aligned}
\]
where $\epsilon_1 >0$ is a constant to be determined, and we have again used the lower bound $\mu_{\ell}> \tan(\eta_1)>0$ for $\ell<n$.  Finally, we estimate the big terms using the Cauchy-Schwarz inequality  
\[
\bigg|\sum_{k \in B}\frac{\nabla_{V_1}\alpha_{\bar{k}k}}{(1+\mu_k^2)}\bigg|^2 \leq (n+1)\sum_{k\in B}\frac{|\nabla_{V_1}\alpha_{\bar{k}k}|^2}{(1+\mu_k^2)^2}.
\]
We now combine these estimates to estimate~\eqref{eq: HessTerm} from below when $\mu_n <0$.  We have
\[
\begin{aligned}
&\sum_{i,j=0}^{n}\frac{\mu_i+\mu_j}{(1+\mu_i^2)(1+\mu_j^2)}|\nabla_{V_1}\alpha_{\bar{j}i}|^2 \geq \sum_{i,j=0, i\ne j}^{n-1}\frac{\mu_i+\mu_j}{(1+\mu_i^2)(1+\mu_j^2)}|\nabla_{V_1}\alpha_{\bar{j}i}|^2 \\
&+\left(1+(1+\delta_1)\left(\frac{-(n+1)}{(n+1)+(\tan(\eta_1))^2} + 2\mu_n\epsilon_1\right)\right)\sum_{\ell \in S, \ell <n} \frac{2 \mu_\ell}{(1+\mu_{\ell}^2)^2}|\nabla_{V_1}\alpha_{\bar{\ell}\ell}|^2\\
&+\sum_{k \in B}\frac{2|\nabla_{V_1}\alpha_{\bar{k}k}|^2}{(1+\mu_k^2)^2}\left(\mu_{k} +(1+\delta_1)\mu_n\left(\frac{2(n+1)^2}{\epsilon_1 \tan(\eta_1)} + (n+1)\right) \right)\\
&+2(1+\delta_1^{-1})\mu_{n}|\nabla_{V_1}h|^2
\end{aligned}
\]
For ease of notation set
\[
\kappa_1 = \frac{(\tan(\eta_1))^2}{2[(n+1)+(\tan(\eta_1))^2]}.
\]
Examining the coefficient in front of the small terms yields 
\[
1+(1+\delta_1)\left(\frac{-(n+1)}{(n+1)+(\tan(\eta_1))^2} + 2\mu_n\epsilon_1\right) = 2\kappa_1-\delta_1(1-2\kappa_1) + 2(1+\delta_1)\mu_n\epsilon_1.
\]
 Take $\delta_1 = \kappa_1$, then
\[
\begin{aligned}
1+(1+\delta_1)\left(\frac{-(n+1)}{(n+1)+(\tan(\eta_1))^2} + 2\mu_n\epsilon_1\right) &= \kappa_1+2\kappa_1^2 + 2(1+\kappa_1)\mu_n\epsilon_1\\
&\geq \kappa_1
\end{aligned}
\]
provided we take $\epsilon_1 = \frac{\kappa_1^2}{(1+\kappa_1)C(\eta_1)}$, where $\mu_n >-C(\eta_1)$ is a lower bound for $\mu_n$.   Since $\kappa_1, \delta_1, \epsilon_1$ are universal we have
\[
\mu_{k} +(1+\delta_1)\mu_n\left(\frac{2(n+1)^2}{\epsilon_1 \tan(\eta_1)} + (n+1)\right) \geq \mu_k- 2C_1,
\]
\[
2(1+\delta_1^{-1})\mu_{n}|\nabla_{V_1}h|^2 \geq -C_2
\]
for universal constants $C_1, C_2$. We have proved

\begin{prop}\label{prop: HessEst}
We have the following lower bound for~\eqref{eq: HessTerm}.
\[
\begin{aligned}
\sum_{i,j=0}^{n}\frac{\mu_i+\mu_j}{(1+\mu_i^2)(1+\mu_j^2)}|\nabla_{V_1}\alpha_{\bar{j}i}|^2 &\geq-C_2+ \sum_{\substack{0\leq i,j \leq n\\ i\ne j}}\frac{\mu_i+\mu_j}{(1+\mu_i^2)(1+\mu_j^2)}|\nabla_{V_1}\alpha_{\bar{j}i}|^2\\
&+\kappa_1\sum_{\ell \in S, \ell <n} \frac{2 \mu_\ell}{(1+\mu_{\ell}^2)^2}|\nabla_{V_1}\alpha_{\bar{\ell}\ell}|^2\\
&+ \sum_{k \in B}\frac{\left(2\mu_{k} -C_1\right)|\nabla_{V_1}\alpha_{\bar{k}k}|^2}{(1+\mu_k^2)^2}
\end{aligned}
\]
where $C_1=C_2=0$ and $\kappa_1=1$ if $\mu_{n} \geq 0$, and if $\mu_n <0$, then
\[
\kappa_1 = \frac{(\tan(\eta_1))^2}{2[(n+1)+(\tan(\eta_1))^2]}
\]
and  $C_1, C_2 \geq 0$ are uniform constants.
\end{prop}

We now turn our focus to estimating the gradient term~\eqref{eq: GradTerm}.  Fix $0 \leq i \leq n$ and write
\[
\begin{aligned}
|\nabla_{i} \alpha_{\bar{V_1}V_1}|^2 &= \bigg|\sum_{j=0}^{n}V_{1}^{j}\nabla_{V_1}\alpha_{\bar{j}i}\bigg|^2\\
&\leq \sum_{j,\ell \in B}  |V_{1}^{j}||V_1^{\ell}||\nabla_{V_1}\alpha_{\bar{j}i}||\nabla_{V_1}\alpha_{\bar{\ell}i}|\\
&\quad+2\sum_{j\in B, \ell \in S}  |V_{1}^{j}||V_1^{\ell}||\nabla_{V_1}\alpha_{\bar{j}i}||\nabla_{V_1}\alpha_{\bar{\ell}i}|\\
&\quad+\sum_{j,\ell \in S}  |V_{1}^{j}||V_1^{\ell}||\nabla_{V_1}\alpha_{\bar{j}i}||\nabla_{V_1}\alpha_{\bar{\ell}i}|
\end{aligned}
\]
For the big terms, we use Cauchy-Schwarz in the following way
\begin{equation}\label{eq: GradTermEst1}
\begin{aligned}
\sum_{j,\ell \in B}  |V_{1}^{j}||V_1^{\ell}||\nabla_{V_1}\alpha_{\bar{j}i}||\nabla_{V_1}\alpha_{\bar{\ell}i}| &\leq \left(\sum_{j,\ell \in B}  |V_1^{\ell}|^2 \frac{\mu_{\ell}}{\mu_{j}}|\nabla_{V_1}\alpha_{\bar{j}i}|^2\right)^{1/2}\left(\sum_{j,\ell \in B}  |V_1^{j}|^2 \frac{\mu_{j}}{\mu_{\ell}}|\nabla_{V_1}\alpha_{\bar{\ell}i}|^2\right)^{1/2}\\
&=\sum_{j,\ell \in B}  |V_1^{\ell}|^2 \frac{\mu_{\ell}}{\mu_{j}}|\nabla_{V_1}\alpha_{\bar{j}i}|^2\\
&=\left(\lambda_1-\sum_{\ell \in S}|V_1^{\ell}|^2\mu_{\ell}\right)\sum_{j \in B} \frac{1}{\mu_{j}}|\nabla_{V_1}\alpha_{\bar{j}i}|^2
\end{aligned}
\end{equation}
where in the last line we used Lemma~\ref{lem: convComb}.  For the small terms we apply Cauchy-Schwarz to get
\begin{equation}\label{eq: GradTermEst2}
\begin{aligned}
\sum_{j,\ell \in S}  |V_{1}^{j}||V_1^{\ell}||\nabla_{V_1}\alpha_{\bar{j}i}||\nabla_{V_1}\alpha_{\bar{\ell}i}|&\leq \left(\sum_{j,\ell \in S}  |V_1^{\ell}|^2|\nabla_{V_1}\alpha_{\bar{j}i}|^2\right)^{1/2}\left(\sum_{j,\ell \in S}  |V_1^{j}|^2|\nabla_{V_1}\alpha_{\bar{\ell}i}|^2\right)^{1/2}\\
&= \sum_{j,\ell \in S}  |V_1^{\ell}|^2|\nabla_{V_1}\alpha_{\bar{j}i}|^2 \leq \sum_{j\in S} |\nabla_{V_1}\alpha_{\bar{j}i}|^2
\end{aligned}
\end{equation}
where in the last line we used Lemma~\ref{lem: convComb} again.  Finally we estimate the cross terms. 
\begin{equation}\label{eq: GradTermEst3}
\begin{aligned}
2\sum_{j\in B, \ell \in S}  |V_{1}^{j}||V_1^{\ell}||\nabla_{V_1}\alpha_{\bar{j}i}||\nabla_{V_1}\alpha_{\bar{\ell}i}| &\leq \epsilon_0\sum_{j\in B, \ell \in S} |V_1^{j}|^2|V_1^{\ell}|^2|\nabla_{V_1}\alpha_{\bar{j}i}|^2 + \frac{1}{\epsilon_0}\sum_{j\in B, \ell \in S} |\nabla_{V_1}\alpha_{\bar{\ell}i}|^2\\
&\leq \epsilon_0\sum_{j\in B} |V_1^{j}|^2|\nabla_{V_1}\alpha_{\bar{j}i}|^2 +\frac{n+1}{\epsilon_0}\sum_{\ell \in S} |\nabla_{V_1}\alpha_{\bar{\ell}i}|^2\\
&\leq \epsilon_0\left(\lambda_1-\sum_{\ell \in S}|V_{1}^{\ell}|^2\mu_{\ell}\right)\sum_{j\in B}\frac{1}{\mu_{j}}|\nabla_{V_1}\alpha_{\bar{j}i}|^2\\
&\quad +\frac{n+1}{\epsilon_0}\sum_{\ell \in S} |\nabla_{V_1}\alpha_{\bar{\ell}i}|^2
\end{aligned}
\end{equation}
for a constant $\epsilon_0>0$ to be determined.  In the final inequality we used the following estimate.  If $j \in B$, we have $\mu_{j}>0$, so Lemma~\ref{lem: convComb} shows
\[
|V_{1}^{j}|^2\mu_{j} \leq \sum_{k \in B} |V_1^{k}|^2 \mu_k = \lambda_1 - \sum_{\ell \in S}|V_1^{\ell}|^2\mu_{\ell}.
\]
Combining estimates~\eqref{eq: GradTermEst1}~\eqref{eq: GradTermEst2} and ~\eqref{eq: GradTermEst3} we arrive at
\[
\begin{aligned}
|\nabla_{i} \alpha_{\bar{V_1}V_1}|^2 &\leq(1+\epsilon_0)\left(\lambda_1 - \sum_{\ell \in S}|V_1^{\ell}|^2\mu_{\ell}\right)\sum_{j \in B} \frac{1}{\mu_{j}}|\nabla_{V_1}\alpha_{\bar{j}i}|^2\\
&\quad+(1 +\frac{n+1}{\epsilon_0})\sum_{j \in S} |\nabla_{V_1}\alpha_{\bar{j}i}|^2.
\end{aligned}
\]
Every term appearing on the right hand side of this estimate has a partner appearing on the right hand side of the estimate in Proposition~\ref{prop: HessEst} {\em except} the term $i=j=n$.  For this term we return to~\eqref{eq: CSmun} and apply Cauchy-Schwarz to get
\[
|\nabla_{V_1}\alpha_{\bar{n}n}|^{2} \leq 2|\nabla_{V_1}h|^2 + 2n \sum_{\ell <n}\frac{|\nabla_{V_1}\alpha_{\bar{\ell}\ell}|^2}{(1+\mu_{\ell}^2)^2}.
\]
Summarizing we have proved
\begin{prop}\label{prop: GradEst}
For any constant $\epsilon_0 \in(0,1)$ we have the following estimates.
\begin{enumerate}
\item  For $0\leq i <n$
\[
\begin{aligned}
|\nabla_{i} \alpha_{\bar{V_1}V_1}|^2 &\leq(1+\epsilon_0)\left(\lambda_1 - \sum_{\ell \in S}|V_1^{\ell}|^2\mu_{\ell}\right)\sum_{j \in B} \frac{1}{\mu_{j}}|\nabla_{V_1}\alpha_{\bar{j}i}|^2\\
&\quad+(1 +\frac{n+1}{\epsilon_0})\sum_{j \in S} |\nabla_{V_1}\alpha_{\bar{j}i}|^2.
\end{aligned}
\]
\item For $i=n$ we have
\[
\begin{aligned}
|\nabla_{n} \alpha_{\bar{V_1}V_1}|^2 &\leq(1+\epsilon_0)\left(\lambda_1 - \sum_{\ell \in S}|V_1^{\ell}|^2\mu_{\ell}\right)\sum_{j \in B} \frac{1}{\mu_{j}}|\nabla_{V_1}\alpha_{\bar{j}n}|^2\\
&\quad+(1 +\frac{n+1}{\epsilon_0})\sum_{j \in S, j<n} |\nabla_{V_1}\alpha_{\bar{j}n}|^2\\
&+ 2(1 +\frac{n+1}{\epsilon_0})C_3 + 2n(1 +\frac{n+1}{\epsilon_1})\sum_{\ell <n}\frac{|\nabla_{V_1}\alpha_{\bar{\ell}\ell}|^2}{(1+\mu_{\ell}^2)^2}.
\end{aligned}
\]
for a uniform constant $C_3$.
\end{enumerate}
\end{prop}
We are now ready to combine Propositions~\ref{prop: HessEst} and~\ref{prop: GradEst} to obtain the key estimate towards the interior $C^2$ estimate.  Recall that we need to estimate the following quantity from below;
\begin{equation}\label{eq: needToEst}
\frac{1}{\lambda_1}\sum_{i,j=0}^{n}\frac{\mu_i+\mu_j}{(1+\mu_i^2)(1+\mu_j^2)}|\nabla_{V_1}\alpha_{\bar{j}i}|^2 -\frac{1}{\lambda_1^2} \sum_{i=0}^{n}\frac{1}{1+\mu_i^2}|\nabla_i\alpha_{\bar{V_1}V_1}|^2.
\end{equation}
In order to do this, we will estimate a related quantity.  Let us defined
\[
\Gamma := \sum_{\ell \in S}|V_1^{\ell}|^2\mu_{\ell}
\]
and note that, by definition of $S$ we have $\Gamma \leq \delta_0\lambda_1 \leq \frac{1}{2}\lambda_1$.  Consider
\[
\Upsilon:=\frac{(1-c_0\lambda_1^{-1})}{\lambda_1(\lambda_1-\Gamma)(1+\epsilon_0)}\sum_{i=0}^{n}\frac{1}{1+\mu_i^2}|\nabla_i\alpha_{\bar{V_1}V_1}|^2
\]
Where $c_0$ is a positive constant to be determined, and we assume that 
\begin{equation}\label{eq: const1}
\lambda_1 > c_0.
\end{equation}
From now on, in order to ensure that our constants can be chosen consistently, we will note each constraint as a separate equation.  By Proposition~\ref{prop: GradEst} we have
\[
\begin{aligned}
\Upsilon &\leq \frac{(1-c_0\lambda_1^{-1})}{\lambda_1}\sum_{\substack{0\leq i \leq n\\ j\in B}} \frac{|\nabla_{V_1}\alpha_{\bar{j}i}|^2}{\mu_{j}(1+\mu_i^2)}\\
&\quad + \frac{2(n+1)(1-c_0\lambda_1^{-1})}{\lambda_1(\lambda_1-\Gamma)(1+\epsilon_0)\epsilon_0}\sum_{\substack{0\leq i \leq n, \, j \in S\\ (i,j)\ne (n,n)}}\frac{ |\nabla_{V_1}\alpha_{\bar{j}i}|^2}{1+\mu_i^2}\\
&+ \frac{4n(n+1)(1-c_0\lambda_1^{-1})}{\lambda_1(1+\mu_n^2)(\lambda_1-\Gamma)(1+\epsilon_0)\epsilon_0}\sum_{\ell <n}\frac{|\nabla_{V_1}\alpha_{\bar{\ell}\ell}|^2}{(1+\mu_{\ell}^2)^2}\\
&\quad + \frac{4(n+1)(1-c_0\lambda_1^{-1})}{\lambda_1(1+\mu_n^2)(\lambda_1-\Gamma)(1+\epsilon_0)\epsilon_0}C_3 
\end{aligned}
\]
We are going to estimate
\[
\frac{1}{\lambda_1}\sum_{i,j=0}^{n}\frac{\mu_i+\mu_j}{(1+\mu_i^2)(1+\mu_j^2)}|\nabla_{V_1}\alpha_{\bar{j}i}|^2-\Upsilon
\]
from below, by making use of Proposition~\ref{prop: HessEst}, and comparing the coefficients of $|\nabla_{V_1}\alpha_{\bar{j}i}|^2$ term by term. There are four cases depending on $(i,j)$.
\begin{itemize}
\item[Case 1:] $0 \leq i \leq n, j\in B$, and $i \ne j$.
\item[Case 2:] $ i = j \in B$.
\item[Case 3:] $ 0 \leq i \leq n, j \in S$ and $i\ne j$.
\item[Case 4:] $ 0 \leq i <n$, and $i=j \in S$.
\end{itemize}

\noindent {\bf Case 1}.   We we have to compare $\frac{\mu_i+\mu_j}{\lambda_1(1+\mu_i^2)(1+\mu_{j}^2)}$ and $ \frac{(1-c_0\lambda_1^{-1})}{\lambda_1(1+\mu_i^2)\mu_j}$.  Ignoring the common factors it suffices to estimate
\[
\frac{\mu_i+\mu_j}{(1+\mu_{j}^2)} - \frac{(1-c_0\lambda_1^{-1})}{\mu_j} = \frac{c_0\lambda_1^{-1}\mu_j^2 + \mu_i\mu_j+c_0\lambda_1^{-1}-1}{\mu_j(1+\mu_j^2)}.
\]
We now use that $j\in B$, so $\mu_j \lambda_1^{-1} >\delta_0$.  Since $\mu_i \geq \mu_n \geq -C(\eta_1)$ we get the estimate
\[
\frac{c_0\lambda_1^{-1}\mu_j^2 + \mu_i\mu_j+c_0\lambda_1^{-1}-1}{\mu_j(1+\mu_j^2)} \geq  \frac{(c_0\delta_0 - C(\eta_1))\mu_j -1}{\mu_j(1+\mu_j^2)} >0
\]
as long as we have 
\begin{equation}\label{eq: const2}
c_0 >\delta_0^{-1}( C(\eta_1) +1), \quad \text{ and }\quad \lambda_1> \delta_0^{-1}
\end{equation}
where the latter condition guarantees $\mu_j>1$ since $j \in B$.
\\
\\
{\bf Case 2} In this case we need to estimate
\begin{equation}\label{eq: Case2Est}
\frac{(2\mu_{j}- C_1)}{1+\mu_j^2} -\frac{(1-c_0\lambda_1^{-1})}{\mu_j} -   \frac{4n(n+1)(1-c_0\lambda_1^{-1})}{(1+\mu_n^2)(1+\mu_j^2)(\lambda_1-\Gamma)(1+\epsilon_0)\epsilon_0}
\end{equation}
Observe that, from the definition of $\Gamma$ we have
\[
\lambda_1-\Gamma  = \lambda_1 -\sum_{\ell\in S}|V_{1}^{\ell}|^2\mu_{\ell} \geq (1-\delta_0)\lambda_1 \geq \frac{1}{2}\lambda_1
\]
Therefore we can estimate~\eqref{eq: Case2Est} from below by
\[
 \frac{1}{\mu_j(1+\mu_j^2)}\left(\mu_j^2 - C_1\mu_j -1 - 8n(n+1)\frac{\mu_j}{\lambda_1\epsilon_0}\right) >0
 \]
provided $\mu_j > C_1+1+\frac{8n(n+1)}{\lambda_1\epsilon_0}$.  Since $j\in B$ this is guaranteed as soon as
\begin{equation}\label{eq: const3}
\lambda_1 \geq \delta_0^{-1}\left(C_1+1+\frac{8n(n+1)}{\lambda_1\epsilon_0}\right)
\end{equation}
\\
\\
{\bf Case 3} We have to estimate
\[
\frac{\mu_i+\mu_j}{(1+\mu_j^2)} -\frac{2(n+1)(1-c_0\lambda_1^{-1})}{(\lambda_1-\Gamma)(1+\epsilon_0)\epsilon_0} \geq \frac{\mu_i+\mu_j}{(1+\mu_j^2)} -\frac{4(n+1)}{\lambda_1\epsilon_0}
\]
We now consider three subcases separately, according to whether $\mu_j$ is relatively large or relatively small compared to $\mu_n$, or $j=n$.

{\bf Case 3a}  First suppose $j<n$, and that $\mu_i + \mu_j \geq \frac{3}{4}\mu_j$ (which is always the case if $\mu_n\geq 0$).  Then we have
\[
\begin{aligned}
\mu_i+\mu_j -\frac{4(n+1)(1+\mu_j^2)}{\lambda_1\epsilon_0} &\geq \frac{3\mu_j}{4} - \frac{4(n+1)}{\epsilon_0}\frac{(1+\mu_j^2)}{\lambda_1}\\
&\geq \frac{3\mu_j}{4} - \frac{4(n+1)}{\epsilon_0 \lambda_1} -\frac{ 4(n+1)\delta_0 \mu_j}{\epsilon_1}
\end{aligned}
\]
where we used that  $\frac{1+\mu_j^2}{\lambda_1} < \frac{1}{\lambda_1} + \delta_0 \mu_j$  by the definition of $S$.  Using that $j\ne n$, if we choose 
\begin{equation}\label{eq: const4}
\delta_0 < \frac{\epsilon_0}{16(n+1)}
\end{equation}
 we get
\[
\begin{aligned}
\mu_i+\mu_j -\frac{4(n+1)(1+\mu_j^2)}{\lambda_1\epsilon_1} &\geq \frac{1}{2}\mu_j - \frac{4(n+1)}{\epsilon_0 \lambda_1}\\
&\geq \frac{\tan(\eta_1)}{4} - \frac{4(n+1)}{\epsilon_0 \lambda_1} >0
\end{aligned}
\]
provided 
\begin{equation}\label{eq: const5}
\lambda_1 > \frac{16(n+1)}{\epsilon_0 \tan(\eta_1)}.
\end{equation}

{\bf Case 3b} In this case we still assume $j<n$, but that $\mu_i + \mu_j \leq \frac{3}{4}\mu_j$.  Then necessarily $i=n$,  $\mu_n <0$, and we get that $\mu_j \leq 4C(\eta_1)$ for a constant depending only on $\eta_1$.  Since $\mu_j + \mu_n >\tan(\eta_1)$ by Lemma~\ref{lem: AngleBasicProps} we conclude
\begin{equation}\label{eq: ineqDegenCase}
 \frac{\mu_i+\mu_j}{(1+\mu_j^2)} -\frac{4(n+1)}{\lambda_1\epsilon_0} \geq \frac{\tan(\eta_1)}{1+16C(\eta_1)^2}  -\frac{4(n+1)}{\lambda_1\epsilon_0} >0
 \end{equation}
provided
\begin{equation}\label{eq: const6}
 \lambda_1> \frac{4(n+1)(1+16C(\eta_1)^2)}{\epsilon_0 \tan(\eta_1)}
 \end{equation}
 
{\bf Case 3c} If $j=n$, and $0\leq i <n$, then~\eqref{eq: ineqDegenCase} also holds, and the same estimate works.
\\
\\
{\bf Case 4} Finally, we consider the fourth case.  We need to estimate from below the quantity 
\[
\begin{aligned}
&\kappa_1\frac{2\mu_{j}}{(1+\mu_{j}^2)}  -\frac{4(n+1)(1-c_0\lambda_1^{-1})}{\lambda_1(1+\epsilon_0)\epsilon_0}-\frac{8n(n+1)(1-c_0\lambda_1^{-1})}{(1+\mu_n^2)\lambda_1(1+\epsilon_0)(1+\mu_j^2)\epsilon_0}\\
&\geq \kappa_1\frac{2\mu_{j}}{(1+\mu_{j}^2)} - \frac{c_1}{\lambda_1 \epsilon_0} - \frac{c_1}{\lambda_1(1+\mu_j^2)\epsilon_0}
\end{aligned}
\]
where $c_1>0$ depends only on $n, \eta_1, \eta_2$, and where $\kappa_1$ is the constant appearing in Proposition~\ref{prop: HessEst}.  We therefore need to determine the sign of
\[
2\kappa_1\epsilon_0 \lambda_1 \mu_j - c_1(1+\mu_j^2) - c_1 =  \mu_j(2\kappa_1\epsilon_0\lambda_1-c_1\mu_j) - 2c_1.
\]
Since $j\in S$ we have $\mu_j \leq \delta_0 \lambda_1$ and so
\[
\mu_j(2\kappa_1\epsilon_0\lambda_1-c_1\mu_j) - 2c_1 \geq \mu_j\lambda_1(2\kappa_1\epsilon_0-c_1\delta_0) - 2c_1 >0
\]
provided
\begin{equation}\label{eq: const 7}
\delta_0 <\frac{\kappa_1 \epsilon_0}{c_1} \qquad \lambda_1 > \frac{4c_1}{\tan(\eta_1)\kappa_1\epsilon_0}
\end{equation}

We can now choose the constants $\delta_0, \epsilon_0, c_0$ consistently.  From~\eqref{eq: const4} and~\eqref{eq: const 7} we see that it suffices to choose $\delta_0 < \max\{c\kappa_1\epsilon_0, \frac{1}{2}\}$ for $0<c<1$ a positive, uniform constant depending only on $n, \eta_1, \eta_2$.  Recall that $\kappa_1$ is the uniform constant appearing in Proposition~\ref{prop: HessEst}.  From~\eqref{eq: const2} we choose $c_0 =K_0\epsilon_0^{-1}$ for a positive uniform constant $K_0$ depending only on $n, \eta_1$.  Finally, from~\eqref{eq: const3} we conclude

\begin{prop}\label{prop: C2KeyEst}
For every constant $\epsilon_0>0$, there exists a uniform constant $C_4$ depending only on $n, \eta_1, \eta_2$ so that if $\lambda_1 \geq C_4\epsilon_0^{-2}$ we have
\[
\begin{aligned}
&\frac{1}{\lambda_1}\sum_{i,j=0}^{n}\frac{\mu_i+\mu_j}{(1+\mu_i^2)(1+\mu_j^2)}|\nabla_{V_1}\alpha_{\bar{j}i}|^2 -\frac{1}{\lambda_1^2} \sum_{i=0}^{n}\frac{1}{1+\mu_i^2}|\nabla_i\alpha_{\bar{V_1}V_1}|^2\\
& \geq -\frac{2\epsilon_0}{\lambda_1^2} \sum_{i=0}^{n}\frac{1}{1+\mu_i^2}|\nabla_i\alpha_{\bar{V_1}V_1}|^2 -\frac{C_2}{\lambda_1} - \frac{C_5}{\epsilon_0 \lambda_1^2}
\end{aligned}
\]
For uniform constants $C_2, C_5$.
\end{prop}
\begin{proof}
By combining Propositions~\ref{prop: HessEst} and~\ref{prop: GradEst} we have show that
\[
\frac{1}{\lambda_1}\sum_{i,j=0}^{n}\frac{\mu_i+\mu_j}{(1+\mu_i^2)(1+\mu_j^2)}|\nabla_{V_1}\alpha_{\bar{j}i}|^2 - \Upsilon \geq - \frac{C_2}{\lambda_1} -C(n,\eta_1, \eta_2)\frac{(1-c_0\lambda_1^{-1})}{\epsilon_1 \lambda_1^2}C_3.
\]
We only need to compare $\Upsilon$ with the negative term containing $|\nabla_i\alpha_{\bar{V_1}V_1}|^2$.  Choose $C_4 = 100 K_0$, where $K_0$ is the constant defining $c_0$ above. If $\lambda_1>C_4\epsilon_{0}^{-2}$, then $c_0\lambda_1^{-1} < \epsilon_0$ and we have
\[
\begin{aligned}
 \frac{(1-c_0\lambda_1^{-1})}{(\lambda_1-\Gamma)(1+\epsilon_0)} -\frac{1}{\lambda_1} &\geq \frac{1-\epsilon_0}{\lambda_1(1+\epsilon_0)} -\frac{1}{\lambda_1}\\
 &\geq\frac{-2\epsilon_0}{\lambda_1}.
 \end{aligned}
\]
The result follows.
\end{proof}

We can now prove the interior, spatial $C^2$ estimate.  Recall that we have written $\alpha = \alpha_{\phi} = \hat{\alpha}+\DDb\phi$.  Normalize $\phi$ so that $\inf_{X_{\epsilon}}\phi =0$.  Define $H : [0,\sup_{X_{\epsilon}}\phi]$ by
\[
H(t) = -2At + \frac{A \tau}{2} t^2
\]
for constants $A\gg0, \tau>0$ to be determined.  By Proposition~\ref{prop: c0bound}, and Corollary~\ref{cor: intSubSol}  we have a uniform bound for $\|\phi\|_{L^{\infty}}$, and so we can choose $\tau < \min\{\frac{1}{2\|\phi\|_{L^{\infty}}},1\}$ so that
\[
-2A \leq H' \leq -A \qquad H'' = A\tau >0.
\]
We apply the maximum principle to the test function
\[
Q := \log(\lambda_1) + H(\phi).
\]
If this quantity achieves its maximum on $\del X_{\epsilon}$, then it is uniformly bounded in terms of the boundary data, and we obtain
\[
\lambda_1 \leq Ce^{-A\osc_{X_{\epsilon}}\phi},
\]
which is the desired estimate.  Otherwise, suppose $Q$ attains an interior maximum at $(p_0,t_0)$.   Fix holomorphic normal coordinates $(z_0,\ldots,z_n)$ for $\hat{\omega}$ in a neighborhood of $(p_0,t_0)$ so that $\alpha$ is diagonal at $(p_0,t_0)$ with eigenvalues $\mu_0 \geq \mu_1 \geq \cdots \geq \mu_n$.  We perform the perturbation described at the beginning of the section, and consider the quantity
\[
\tilde{Q} =  \log(\tilde{\lambda}_1)+ H(\phi).
\]
By construction we have $\tilde{Q} \leq Q$ near $(p_0,t_0)$, and $\tilde{Q}(p_0,t_0) = Q(p_0,t_0)$.  Thus $\tilde{Q}$ achieves an interior maximum at $(p_0,t_0)$.  Applying the linearized operator at $(p_0,t_0)$ we have
\[
0 \geq F^{i\bar{j}}\nabla_i\nabla_{\bar{j}}\tilde{Q}.
\]
The linearized operator applied to the $H(\phi)$ term gives
\[
\begin{aligned}
F^{i\bar{j}}\nabla_{i}\nabla_{\bar{j}}H(\phi) &= H'(\phi)\sum_{i}\frac{1}{1+\mu_{i}^2} \nabla_i\nabla_{\bar{i}}\phi + H''(\phi) \sum_{i}\frac{1}{1+\mu_i^2}|\nabla_i \phi|^2\\
&=- H'(\phi)\sum_{i}\frac{ \hat{\alpha}_{\bar{i}i}-\mu_i }{1+\mu_i^2} + A\tau\sum_{i}\frac{1}{1+\mu_i^2}|\nabla_i \phi|^2
\end{aligned}
\]
Thus, at $(p_0, t_0)$ we have
\[
\begin{aligned}
0 \geq F^{i\bar{j}}\nabla_i\nabla_{\bar{j}}\tilde{Q} &\geq \frac{1}{\lambda_1}\sum_{i,j=0}^{n}\frac{\mu_i+\mu_j}{(1+\mu_i^2)(1+\mu_j^2)}|\nabla_{V_1}\alpha_{\bar{j}i}|^2\\
&\quad -\frac{1}{\lambda_1^2} \sum_{i=0}^{n}\frac{1}{1+\mu_i^2}|\nabla_i\alpha_{\bar{V_1}V_1}|^2 -C\\
&- H'(\phi)\sum_{i}\frac{ \hat{\alpha}_{\bar{i}i}-\mu_i }{1+\mu_i^2} + A\tau\sum_{i}\frac{1}{1+\mu_i^2}|\nabla_i \phi|^2
\end{aligned}
\]

Combining Proposition~\ref{prop: firstBd} and Proposition~\ref{prop: C2KeyEst} there is a uniform constant $C_4$ so that if $\lambda_1 \geq C_4\epsilon_0^{-2}$, then we have 
\[
\begin{aligned}
0&\geq  -\frac{2\epsilon_0}{\lambda_1^2}\sum_{i=0}^{n} \frac{1}{1+\mu_i^2}|\nabla_i\alpha_{\bar{V_1}V_1}|^2- C-\frac{C_2}{\lambda_1} - \frac{C_5}{\epsilon_0 \lambda_1^2} \\
&- H'(\phi)\sum_{i}\frac{ \hat{\alpha}_{\bar{i}i}-\mu_i }{1+\mu_i^2} + A\tau\sum_{i}\frac{1}{1+\mu_i^2}|\nabla_i \phi|^2
\end{aligned}
\]
for uniform constants $C,C_2, C_5$.  At the maximum of $\tilde{Q}$ we have $\nabla_i\tilde{Q}=0$, which gives
\[
\frac{\nabla_i\alpha_{\bar{V_1}V_1}}{\lambda_1} = -H' \nabla_i\phi 
\]
So
\[
\frac{1}{\lambda_1^2}\sum_{i=0}^{n} \frac{1}{1+\mu_i^2}|\nabla_i\alpha_{\bar{V_1}V_1}|^2 =(H')^2\sum_{i}\frac{1}{1+\mu_i^2}|\nabla_i\phi|^2
\]
Choose $\epsilon_0=\max\{\frac{\tau}{10A}, 1\}$, then $2\epsilon_0(H')^2 \leq 8\epsilon_0A^2 < \tau A$ and so we get
\[
\begin{aligned}
0&\geq -C -\frac{C_2}{\lambda_1} - \frac{10AC_5}{\tau \lambda_1^2}\\
&+A\sum_{i}\frac{ \hat{\alpha}_{\bar{i}i}-\mu_i }{1+\mu_i^2}
\end{aligned}
\]
We may assume that $\lambda_1 > R$, where $R$ is the uniform constant appearing in Lemma~\ref{lem: SzCsub} for the subsolution $\hat{\alpha}$ (see Remark~\ref{rk: uniConst}).  Since $\mu_0 \geq \lambda_1$ we can apply Lemma~\ref{lem: SzCsub}.  Since $|\mu_n|$ is bounded by Lemma~\ref{lem: AngleBasicProps} we can always assume we are in the first case of Lemma~\ref{lem: SzCsub} and so we get
\[
\sum_{i}\frac{ \hat{\alpha}_{\bar{i}i}-\mu_i }{1+\mu_i^2} > \kappa_0\sum_{i=1}^{n}\frac{1}{1+\mu_i^2} > \tilde{\kappa}_0
\]
for $\tilde{\kappa}_0$ a uniform constant.  Finally, choose $A$ so that $A\tilde{\kappa}_0 =C+C_2+1$, which is a uniform constant. Then we obtain
\[
0 \geq 1- \frac{10(C+C_2+1)}{\tau\tilde{\kappa_0} \lambda_1^2}.
\]
In particular, $\lambda_1 < C_6$ for a uniform constant $C_6$.  This implies
\[
\lambda_1 \leq C_6e^{A(\phi-\inf_{\mathcal{X}_{\epsilon}}\phi)}
\] 
finishing the proof of~\eqref{eq: thmSpaceC2}.

We next move on to estimating the size of $\alpha_{\bar{t}t}$.  This is easily done by the maximum principle.  Recall that by Lemma~\ref{lem: AngleBasicProps} there is a constant $A$ depending only on $\eta_1$ so that $G=-e^{-AF}$ is concave.   Consider the quantity
\[
Q= \alpha_{\bar{t}t} - C \phi,
\]
and recall that $\alpha = \hat{\alpha} + \DDb\phi$.  Differentiating the equation $G= -e^{-Ah}$ twice, using that $G$ is concave, $\alpha$ is closed, and the curvature vanishes in any temporal direction yields
\[
G^{i\bar{j}}\nabla_{i}\nabla_{\bar{j}}(\alpha_{\phi})_{\bar{t}t} \geq A(\nabla_{t}\nabla_{\bar{t}}h + A|\nabla_{t}h|^2)e^{-Ah}
\]
Choose $C = \frac{A\|h_{\bar{t}t}\|_{L^{\infty}(X_{\epsilon})} +A^2\|h_{t}\|^2_{L^{\infty}(X_{\epsilon})}+1 }{\tilde{\kappa}_0}$ where $\tilde{\kappa}_0= \frac{\kappa_0}{1+C(\eta)^2}$, and $\kappa_0$ is the uniform constant appearing in Lemma~\ref{lem: SzCsub} for $\hat{\alpha}$, while $C(\eta)$ is a bound for $|\mu_n|$ (see Lemma~\ref{lem: AngleBasicProps}).  Suppose that $Q$ achieves an interior maximum on $\mathcal{X}_{\epsilon}$ at the point $(p_0, t_0)$.  Choose holomorphic normal coordinates near $(p_0,t_0)$ such that $\alpha$ is diagonal with eigenvalues $\mu_0, \ldots, \mu_n$.  We get
\[
0 \geq  A(\nabla_{t}\nabla_{\bar{t}}h + A|\nabla_{t}h|^2)e^{-Ah} +C \sum_{i} \frac{\hat{\alpha}_{\bar{i}i}-\mu_i}{1+\mu_i^2}.
\]
If $\alpha_{\bar{t}t} > R$, then by Lemma~\ref{lem: SzCsub} and our choice of $C$ we obtain
\[
0 \geq   A(\nabla_{t}\nabla_{\bar{t}}h + A|\nabla_{t}h|^2)e^{-Ah} +A(\|h_{\bar{t}t}\|_{L^{\infty}(X_{\epsilon})} +A\|h_{t}\|^2_{L^{\infty}(X_{\epsilon})}+1)
\]
a contradiction. Thus, $(\alpha_{\phi})_{\bar{t}t} \leq R$, and so
\[
Q \leq Q(p_0,t_0) \leq R +\frac{A\|h_{\bar{t}t}\|_{L^{\infty}(X_{\epsilon})} +A\|h_{t}\|^2_{L^{\infty}(X_{\epsilon})}+1 }{\tilde{\kappa}_0}\|\phi\|_{L^{\infty}(X_{\epsilon})}
\]
unless $Q$ attains its maximum on the boundary.  Summarizing we have
\begin{prop}
There is a uniform constant $C$ so that
\[
|(\alpha_{\phi})_{\bar{t}t}|_{\hat{g}} \leq C\left(1 + \|h_{\bar{t}t}\|_{L^{\infty}(\mathcal{X}_{\epsilon})} + \|h_{t}\|^2_{L^{\infty}(\mathcal{X}_{\epsilon})}\right) + \sup_{\del \mathcal{X}_{\epsilon}}|(\alpha_\phi)_{\bar{t}t}|_{\hat{g}}
\]
\end{prop}

To estimate the off diagonal terms we use that $\alpha_{\phi} \geq - C(\eta_1) \hat{\omega}$.  Pick any point $(p_0,t_0) \in \mathcal{X}_{\epsilon}$, and choose space-time adapted coordinates so that $(\alpha_{\phi})|_{TX}$ is diagonal with entries $\lambda_1, \ldots, \lambda_n$.  At $p_0$ we have
\[
\alpha_\phi + (C(\eta_1)+1)\hat{\omega} = 
\begin{pmatrix} (\alpha_{\phi})_{\bar{t}t} + C(\eta_1)+1 & (\alpha_{\phi})_{\bar{1}t} & \cdots &(\alpha_{\phi})_{\bar{n}t}\\
(\alpha_{\phi})_{\bar{t}1} & \lambda_1 + C(\eta_1) +1& \cdots &0\\
\vdots & \vdots &\ddots &\vdots\\
(\alpha_{\phi})_{\bar{n}t} & 0 &\cdots &\lambda_n + C(\eta_1) +1
\end{pmatrix}
\]
Let $D$ be the $n\times n$ matrix with $D_{\bar{j}i} = (\lambda_1 +C(\eta)+1)\delta_{\bar{j}i}$. Since $\lambda_i + C(\eta_1) +1 >1$ for all $i$ we can compute the determinant as
\[
0 \leq \det(\alpha_\phi + (C(\eta_1)+1)\hat{\omega}) = \det(D) \left((\alpha_\phi)_{\bar{t}t} + C(\eta_1)+1 - \sum_{i}\frac{|(\alpha_{\phi})_{\bar{t}i}|^2}{\lambda_i+C(\eta_1)+1}\right).
\]
Now since $1 \leq \lambda_i+C(\eta_1)+1 \leq C$ for a uniform constant $C$ by the spatial $C^2$ estimate we obtain
\[
|(\alpha_{\phi})_{\bar{t}i}| \leq C\left(|(\alpha_\phi)_{\bar{t}t}| + C(\eta_1)+1\right)^{1/2}
\]
which finishes the proof of Theorem~\ref{thm: C2est}.

\begin{rk}
Note that the spatial $C^2$ estimate is independent of the estimate for the spatial gradient.  In particular, Proposition~\ref{prop: spaceGradEst} can be obtained directly from the $C^2$ estimate by applying the elliptic theory along the fibers of $\mathcal{X}_{\epsilon}\rightarrow \mathcal{A}_{\epsilon}$.  Nevertheless, we have decided to include the estimate as it may have applications to the existence of geodesic rays in $\mathcal{H}$; see for example \cite{PSS, PS, PS1}.
\end{rk}

\section{Boundary $C^2$ estimates and existence of solutions to the degenerate Lagrangian phase equation}\label{sec: bndryEst}

It remains only to prove the boundary estimates.  The boundary estimates are based on estimates of Guan \cite{Guan1, Guan2} which are in turn inspired by estimates of Trudinger \cite{Tr}.  Similar ideas were used by the first author, Picard and Wu to solve the Dirichlet problem for the Lagrangian phase operator~\cite{CPW}.   In fact, the proof here is much simpler than the boundary estimates for the Lagrangian phase operator established in \cite{CPW} due to the special structure of the boundary.

Without loss of generality, we work near $\{|t|=\epsilon\}$.  Consider the function
\[
v= (\phi-\hat{\phi}_0) + c_0(\epsilon - |t|) - N(\epsilon- |t|)^2
\]
for constants $N, c_0 >0$ to be determined.  For simplicity, let us write $\hat{\phi}_0 = \hat{\phi}$.  Our goal is to choose uniform constants $c_0, N$  so that $F^{i\bar{j}}\nabla_i\nabla_{\bar{j}} < -\epsilon_0$ near $|t|=\epsilon$ for a uniform constant $\epsilon_0$, and so that $v>0$ on a neighboourhood of $\{r=\epsilon\}$.  We compute at a point $(p_0,t_0)$ in coordinates where $\hat{\omega}$ is the identity and $\alpha_{\phi}$ is diagonal.
\[
F^{i\bar{j}}\nabla_i\nabla_{\bar{j}} v = \sum_{i} \frac{\mu_i - \hat{\alpha}_{\bar{i}i}}{1+\mu_i^2} - c_0F^{t\bar{t}}\frac{1}{4|t|} - \frac{N}{2}F^{t\bar{t}} + N\frac{(\epsilon-|t|)}{2|t|}F^{t\bar{t}}.
\]
Suppose $|\mu|\leq R$, where $R$ is the constant in Lemma~\ref{lem: SzCsub} for $\hat{\phi}_0$.  Then we have $1 \geq F^{t\bar{t}} > (1+R^2)^{-1}$.  Also, $\hat{\alpha}_{\bar{i}i} > -C(\eta_1)$.  Combining these estimates gives
\begin{equation}
F^{i\bar{j}}\nabla_i\nabla_{\bar{j}} v \leq(n+1)+C(\eta_1) - c_0 \frac{F^{t\bar{t}}}{4|t|} - \frac{N}{2(1+R^2)} + N\frac{(\epsilon-|t|)}{2|t|}.
\end{equation}
If instead $|\mu|>R$ the we apply Lemma~\ref{lem: SzCsub} and conclude
\begin{equation}
F^{i\bar{j}}\nabla_i\nabla_{\bar{j}} v <-\kappa_0- c_0 \frac{F^{t\bar{t}}}{4|t|} - \frac{N}{2}F^{t\bar{t}} + N\frac{(\epsilon-|t|)}{2|t|},
\end{equation}
for a uniform constant $\kappa_0$.  We consider $v$ on the domain
\[
\Omega_{\delta} := X \times \{ (1-\delta)\epsilon \leq |t|\leq \epsilon\}
\]
where $0<\delta <1-e^{-1}$ is to be determined.  On $\Omega_{\delta}$ we have
\[
0\leq \frac{(\epsilon-|t|)}{|t|} \leq \frac{\delta}{1-\delta} <1.
\]
Choose $N$ so that
\[
\frac{N}{2(1+R^2)} = n+2 + C(\eta_1),
\]
and note that $N$ is uniform.  Then when $|\mu|<R$ we have
\[
F^{i\bar{j}}\nabla_i\nabla_{\bar{j}} v \leq -1 + N\frac{\delta}{(1-\delta)} < -\frac{1}{2}
\]
provided $\delta < \frac{1}{2N+1}$.  When $|\mu|>R$ we have
\[
F^{i\bar{j}}\nabla_i\nabla_{\bar{j}} v \leq - \kappa_0 + N\frac{\delta}{(1-\delta)} < -\frac{\kappa_0}{2}
\]
provided $\delta < \frac{\kappa_0}{2N+\kappa_0}$.  We choose $\delta = \min\{\frac{1}{2N+1}, \frac{\kappa_0}{2N+\kappa_0}\}$, which is a uniform constant.  It remains only to determine $c_0$.  On the $|t|=\epsilon$ component of $\del \Omega_{\delta}$ we have $v=0$, while on the $|t|=(1-\delta)\epsilon$ component we have
\[
v \geq \delta\epsilon - N(\delta\epsilon)^2 \geq \delta \epsilon \left(c_0 - \frac{1}{2}\epsilon\right)>0
\]
provided we take $c_0=1$.  Summarizing we have
\begin{lem}\label{lem: vsubSol}
There exist uniform constants $\delta, N$ so that the function
\[
v:= (\phi-\hat{\phi}_0) + (\epsilon - |t|) - N(\epsilon- |t|)^2
\]
satisfies 
\[
F^{i\bar{j}}\nabla_i\nabla_{\bar{j}}v \leq -\frac{\kappa_0}{2}, \qquad v \geq 0
\]
on $\Omega_{\delta} := X \times \{ (1-\delta)\epsilon \leq |t|\leq \epsilon\}$.
\end{lem}

We are going to estimate the tangent-normal derivatives of $\phi$ near the boundary $\{|t|=\epsilon\}$.  Fix a point $(p_0, t_0)\in \del X_{\epsilon}$, and space-time adapted coordinates $(w_0,\ldots,w_n) = (w_0,w')$.  For $1\leq \ell\leq n$ we compute
\[
F^{i\bar{j}}\nabla_{i}\nabla_{\bar{j}}\nabla_{\ell}(\phi - \hat{\phi}) = \nabla_{\ell}h - \sum_{1\leq i,j\leq n}F^{i\bar{j}}\nabla_{\ell}\hat{\phi}
\]
where in the last line we used that $\nabla_{t}\nabla_{\ell}\hat{\phi} = \nabla_{\bar{t}}\nabla_{\ell}\hat{\phi}=0$ (see Corollary~\ref{cor: intSubSol}).  Next we compute
\[
F^{i\bar{j}}\nabla_{i}\nabla_{\bar{j}}\nabla_{\bar{\ell}}(\phi - \hat{\phi})  = \nabla_{\bar{\ell}}h - F^{i\bar{j}}\sum_{1\leq k \leq n}R_{\bar{\ell}i\bar{j}}\,^{\bar{k}} \nabla_{\bar{k}}\phi - \sum_{1\leq i,j \leq n} F^{i\bar{j}}\nabla_i\nabla_{\bar{j}}\nabla_{\bar{\ell}}\hat{\phi}
\]
where we used that $R_{\bar{\ell}i\bar{j}k} =0$ if $k=0$.  Write $w_\ell = x_{\ell} + \sqrt{-1}y_{\ell}$.  Then combining the spatial gradient bound from Theorem~\ref{thm: gradEst} with Corollary~\ref{cor: intSubSol} we have
\[
-C \leq F^{i\bar{j}}\nabla_{i}\nabla_{\bar{j}}\nabla_{\del_{x_{\ell}}}(\phi - \hat{\phi}) \leq C
\]
for a uniform constant $C$.  Clearly the same estimate holds for $\del_{y_{\ell}}$.  We now consider the quantity
\[
Q_{\pm} := Av \pm \del_{x_{\ell}}(\phi- \hat{\phi}) + B\sum_{i=i}^{n}|w_i|^2 - \frac{D}{2}\log\left(\frac{|t|^2}{\epsilon^2}\right)
\]
for positive constants $A,B,D$ to be determined, which is defined on the set
\[
\hat{\Omega}_{\delta} := \left\{ \sum_{i=1}^{n}|w_i|^2 < (\delta')^2 \right\} \times \{(1-\delta)\epsilon \leq |t| \leq \epsilon\}
\]
where $\delta'$ is a constant depending only on $(X,\omega)$ (namely, the size of the coordinate chart on which holomorphic normal coordinates are defined).  On $|t|=\epsilon$ we have $Q_{\pm} \geq 0$ with equality at $w'=0$.  On $w'= \delta'$ we have
\[
Q_{\pm} \geq -\sup_{\mathcal{X}_{\epsilon}}|\nabla^{X}\phi|_{\hat{g}} + B(\delta')^2.
\]
By Proposition~\ref{prop: spaceGradEst} we can therefore choose a uniform constant $B$ large enough so that $Q_{\pm} \geq 0$ on $|w'|=\delta'$.  Finally, when $|t|=(1-\delta)\epsilon$ we have
\[
 Q_{\pm} \geq -\sup_{X_{\epsilon}}|\nabla^{X}\phi|_{\hat{g}} + \frac{D}{2}\log\left(\frac{1}{(1-\delta)^2}\right)
\]
Since $\delta>0$ is uniform, by Proposition~\ref{prop: spaceGradEst} we can choose $D$ large and uniform so that $Q_{\pm} \geq 0$ on $|t|=(1-\delta)\epsilon$.  Finally, we compute
\[
F^{\bar{i}{j}} \leq -A\frac{\kappa_0}{2} + \sup_{\mathcal{X}_{\epsilon}}|\nabla^{X}\phi|_{\hat{g}} + B\sum_{1\leq i \leq n} F^{i\bar{i}} \leq  -A\frac{\kappa_0}{2} + \sup_{X_{\epsilon}}|\nabla^{X}\phi|_{\hat{g}} + Bn
\]
Therefore, another application of Proposition~\ref{prop: spaceGradEst} shows that we can choose $A$ uniform, sufficiently large  so that $F^{i\bar{j}}Q_{\pm} \leq 0$.  Since $Q_{\pm}\geq 0$ on $\del \hat{\Omega}_{\delta}$ we have $Q_{\pm} \geq 0= Q_{\pm}(p_0,t_0)$.  We conclude that $\del_r Q_{\pm} \leq 0$. Now at $(p_0,t_0)$ we have $\del_r \del_{x_{\ell}}\hat{\phi} =0$ and $\del_r v = \del_r\phi - \del_r\hat{\phi} -1$.  Thus by Theorem~\ref{thm: gradEst} we obtain
\[
|\del_r \del_{x_{\ell}}\phi| \leq \frac{C}{\epsilon}
\]
Repeating the argument with $y_{\ell}$ yields
\begin{prop}\label{prop: TNbdEst}
There is a uniform constant $C$ so that
\[
\sup_{\del X_{\epsilon}}\left(|\nabla_{t} \overline{\nabla^{X}}\phi|_{\hat{g}} + |\nabla_{\bar{t}} \nabla^{X}\phi|_{\hat{g}}\right) \leq \frac{C}{\epsilon}
\]
\end{prop}
Finally, we estimate $\nabla_{t}\nabla_{\bar{t}}\phi$ on the boundary.  In fact the estimate we need follows from a lemma of Caffarelli-Nirenberg-Spruck \cite{CNS3}, which we now recall

\begin{lem}\label{lem: CNS}
Consider the $n\times n$ hermitian matrix
\[
M:=\begin{pmatrix}
a &b_1 & b_2 & \cdots & b_n\\
\overline{b_1} & d_1 &0&\cdots&0\\
\overline{b_2}&0&d_2&\cdots&0\\
\vdots&\vdots&\cdots&\ddots&\vdots\\
\overline{b_n}&0&0&\cdots&d_n
\end{pmatrix}
\]
where we assume $a\geq 1$.  Let $\mu_0,\mu_1,\ldots,\mu_n$ be the eigenvalues of $M$.  For all $0<\epsilon_0 \ll 1$, there exists $\delta(\epsilon_0)>0$ depending only on $d_1,\ldots, d_n, \epsilon$ such that, if $\sum_{i=1}^{n} \frac{|b_i|^2}{a}<\delta(\epsilon_0)$ then
\[
|\mu_i - d_i| < \epsilon_0.
\]
for $i=1, \ldots, n$.  Furthermore, we have
\[
\mu_0 = a\left(1+ O\left(\sum_{i=1}^{n} \frac{|b_i|^2}{a}\right)\right)
\]
with implied constants depending only on $d_1,\ldots, d_n$.
\end{lem}
\begin{proof}

We give the proof, since the statement in \cite{CNS3} is not exactly what we need.  Just as in \cite{CNS3}, the eigenvalues are given by the zeroes of
\[
\det\begin{pmatrix}
1-\frac{\mu}{a} &\frac{b_1}{a} & \frac{b_2}{a} & \cdots & \frac{b_n}{a}\\
\overline{b_1} & d_1-\mu &0&\cdots&0\\
\overline{b_2}&0&d_2-\mu&\cdots&0\\
\vdots&\vdots&\cdots&\ddots&\vdots\\
\overline{b_n}&0&0&\cdots&d_n-\mu
\end{pmatrix} =0.
\]
Expanding the determinant gives
\[
0= (1-\frac{\mu}{a})\left(\prod_{i=1}^{n} (d_i-\mu)\right) +\sum_{j=1}^{n}(-1)^{j}\frac{|b_j|^2}{a}\prod_{i\ne j}(d_i-\mu).
\]
Introduce parameters $t_0 = \frac{1}{a}$ and $t_{i} = \frac{|b_i|^2}{a}$, and write this equation as 
\[
P(t_0, t_1\ldots, t_n, \lambda)=0
\]
for a polynomial $P$. When $(t_1,\ldots, t_n)=0$ we have that $\mu=d_i$ is a zero.  Since the roots of $P(t_0,\ldots,t_n)$ depend continuously on the coefficients we get that, for all $\epsilon_0 >0$, there exists $\delta>0$ depending only on $d_1,\ldots, d_n, \epsilon_0$ such that if $|t|<\delta$, then
\[
|\mu_i - d_i| < \epsilon_0.
\]
For the last eigenvalue we set $\mu = a\gamma$ and consider
\[
\det\begin{pmatrix}
1-\gamma &\frac{b_1}{a} & \frac{b_2}{a} & \cdots & \frac{b_n}{a}\\
\frac{\overline{b_1}}{a} & \frac{d_1}{a}-\gamma &0&\cdots&0\\
\frac{\overline{b_2}}{a}&0&\frac{d_2}{a}-\gamma&\cdots&0\\
\vdots&\vdots&\cdots&\ddots&\vdots\\
\frac{\overline{b_n}}{a}&0&0&\cdots&\frac{d_n}{a}-\gamma
\end{pmatrix} =0
\]
which we write as
\[
0 = (1-\gamma)\left(\prod_{i=1}^{n}(t_0d_1-\gamma)\right)+ \sum_{j=1}^{n}(-1)^{j}t_0t_j\prod_{i\ne j}(t_0d_i-\gamma).
\]
When $t=0, \gamma=1$ is a simple zero, and hence the implicit function theorem gives that for $|t|<\delta$, there is a constant $C$ depending on $d_1,\ldots,d_n$ such that 
\[
|1-\gamma(t)| \leq C|t|
\]
whence
\[
\mu_0(t) = a(1+O(|t|))
\]
with constants depending only on $d_1,\ldots,d_n$.
\end{proof}
This lemma, together with the boundary tangent-normal estimates in Proposition~\ref{prop: TNbdEst} immediately implies the normal-normal estimate.  Suppose there is a point $(p_0,t_0)\in \del X_{\epsilon}$ where $(\alpha_{\phi})_{\bar{t}t} \geq \frac{K}{\epsilon^2}$ for some constant $K$ to be determined (note that the lower bound is automatic from Lemma~\ref{lem: AngleBasicProps}).  Fix space-time adapted coordinates at $(p_0,t_0)$, and let $\mu_i$ $0\leq i \leq n$ be the eigenvalues of $\alpha_{\phi}$. By the tangent-normal estimates we have
\[
\sum_{i} \frac{|(\alpha_{\phi})_{\bar{t}i}|^2}{(\alpha_{\phi})_{\bar{t}t}} \leq \frac{C}{K}
\]
where $C$ is a uniform constant.  Fixing $\epsilon_0>0$, if $C/K$ is sufficiently small depending only on the bound for the spatial $C^2$ norm and $\epsilon_0$, then we have
\[
|\mu_0 - \alpha_{\bar{t}t}| < \epsilon_0, \qquad |\alpha_{\bar{i}i} - \mu_i| \leq \epsilon_0.
\]
It follows that
\[
\bigg|\sum_{i=0}^{n} \arctan(\mu_i) - \left(\sum_{i=1}^{n} \arctan((\alpha_{\phi})_{\bar{i}i}) + \arctan((\alpha_{\phi})_{\bar{t}t})\right)\bigg| \leq (n+1)\epsilon_0,
\]
since the derivative of arctan is bounded by $1$.  On the boundary we have $\phi = \underline{\phi}$ and $\underline{\phi}$ is a subsolution, satisfying
\[
\sum_{i=0}^{n} \arctan(\underline{\mu}_i) \geq h+\frac{\eta_1}{2}
\]
where $\underline{\mu}_i$ are the eigenvalues of $\underline{\alpha}:=\pi_{X}^{*}\alpha+ \DDb \underline{\phi}$.  By the Schur-Horn theorem \cite{H1} and the convexity of the super-level sets $\{F \geq \sigma\}$ for $\sigma \geq (n-1)\frac{\pi}{2}$ we have
\[
\sum_{i=1}^{n} \arctan(\underline{\alpha}_{\bar{i}i}) \geq \sum_{i=0}^{n} \arctan(\underline{\mu}_i)-\arctan(\underline{\alpha}_{\bar{t}t}).
\]
So, we have
\[
\begin{aligned}
h(x)+\frac{\eta_1}{2} &\leq \sum_{i=0}^{n} \arctan(\underline{\mu}_i) \leq  \sum_{i=1}^{n} \arctan(\underline{\alpha}_{\bar{i}i}) +\arctan(\underline{\alpha}_{\bar{t}t})\\
& \leq (n+1)\epsilon_0+ \sum_{i=0}^{n}\arctan(\mu_i) -\arctan((\alpha_{\phi})_{\bar{t}t}) + \arctan(\underline{\alpha}_{\bar{t}t})\\
&= (n+1)\epsilon_0 + h(x) -\arctan((\alpha_{\phi})_{\bar{t}t}) + \arctan(\underline{\alpha}_{\bar{t}t}).
\end{aligned}
\]
This implies
\[
\frac{\eta_1}{2} +\arctan((\alpha_{\phi})_{\bar{t}t})\leq (n+1)\epsilon_0 + \arctan(\underline{\alpha}_{\bar{t}t}).
\]
By Lemma~\ref{lem: subsolConstr} we have
\[
\underline{\alpha}_{\bar{t}t} \leq \frac{C_1}{\epsilon^2}
\]
for a uniform constant $C_1$.  Choose $\epsilon_0$ sufficiently small so that $(n+1)\epsilon_0 \leq \frac{\eta_1}{2}$, we conclude that $K\leq C_1$. Thus we conclude 
\begin{prop}
There exists a uniform constant $C$ so that
\[
\sup_{\del\mathcal{X}_{\epsilon}} |\nabla_t\nabla_{\bar{t}} \phi| \leq \frac{C}{\epsilon^2}
\]
\end{prop}

Combining this with Theorem~\ref{thm: C2est}, and Theorem~\ref{thm: gradEst} we conclude
\begin{thm}\label{thm: estSummary}
Suppose $\phi(x,t)$ is a smooth $S^1$ invariant function on $(\mathcal{X}_{\epsilon}, \hat{\omega})$ with $\alpha_{\phi} := \alpha +\DDb\phi(x,t)$ solving the Lagrangian phase equation
\[
F(\hat{\omega}^{-1}\alpha_{\phi}) = h(x,|t|).
\]
with $\phi(x,\epsilon) = \phi_0$ and $\phi(x,\epsilon e^{-1}) = \phi_1$, and $\phi_i \in \mathcal{H}$.  Suppose in addition that $\phi_0, \phi_1, h$ satisfy the structural conditions $(C1), (C2)$ with constants $\eta_1, \eta_2$.  The following estimates hold
\[
\osc_{\mathcal{X}_{\epsilon}}\phi + |\nabla^{X}\phi|_{\hat{\omega}} + |\nabla^{X} \overline{\nabla^{X}}\phi| \leq C
\]
\[
|\nabla_{t} \phi|_{\hat{\omega}} \leq C\left(1+\sup_{\mathcal{X}_{\epsilon}}|\nabla_t h|_{\hat{\omega}} + \frac{1}{\epsilon}\right)
\]
\[
|\nabla_{\bar{t}}\nabla^{X}\phi| \leq C\left(1+\frac{1}{\epsilon} + \sqrt{\sup_{\mathcal{X}_{\epsilon}}|\nabla_t\nabla_{\bar{t}} h|_{\hat{g}}} + \sqrt{\sup_{\mathcal{X}_{\epsilon}}|\nabla_t h|^2_{\hat{g}}}\right)
\]
\[
|\nabla_{t}\nabla_{\bar{t}}\phi| \leq C\left(1+\frac{1}{\epsilon^2} +\sup_{\mathcal{X}_{\epsilon}}|\nabla_t\nabla_{\bar{t}} h|_{\hat{g}}+ \sup_{\mathcal{X}_{\epsilon}}|\nabla_t h|^2_{\hat{g}}\right)
\]
where $C$ is a uniform constant depending only on $\phi_0, \phi_1, \nabla^{X}\overline{\nabla^{X}} h, (X,\omega)$ and the structural constants $\eta_1, \eta_2$.
\end{thm}
First we apply these estimates to solve the Dirichlet problem for the Lagrangian phase operator on $\mathcal{X}_{\epsilon}$.

\begin{thm}\label{thm: DirProbXe}
Suppose $\phi_0, \phi_1 \in \mathcal{H}$, and $h(x,|t|) : \mathcal{X}_{\epsilon} \rightarrow \mathbb{R}$ is a smooth function satisfying structural constraints $(C1), (C2)$.  Then, on 
$\mathcal{X}_{\epsilon}$, there exists a smooth, $S^1$ invariant solution of the equation
\[
F(\hat{\omega}^{-1}\alpha_{\phi}) = h(x,|t|).
\]
with boundary values $\phi_0, \phi_1$.  In particular, for any $\phi_0, \phi_1 \in \mathcal{H}$ there is a unique, smooth $\epsilon$ geodesic joining $\phi_0, \phi_1$.
\end{thm}
\begin{proof}
The corollary follows easily from Theorem~\ref{thm: estSummary}.  Let $\underline{h}(x) = F(\alpha_{\underline{\phi}})$ where $\underline{\phi}$ is the function constructed in Lemma~\ref{lem: subsolConstr} and consider the equation 
\[
F(\hat{\omega}^{-1}\alpha_{\phi_{u}})= (1-u)\underline{h}+u h
\]
where $u \in [0,1]$, and $\phi_u$ has boundary values $\phi_0, \phi_1$.  Note that structural constraints $(C1), (C2)$ hold uniform for the functions $(1-u)\underline{h}+u h$.   Let $I$ be the set of $u\in [0,1]$ for which this equation admits a solution. By the implicit function theorem $I$ is open.  Suppose $I \ni u_i \rightarrow u_{*}$  Combining the estimates in Theorem~\ref{thm: estSummary} with the Evans-Krylov theorem, and arguing as in~\cite{CPW} we conclude that $I$ is closed.
\end{proof}

Since we have obtained estimates that scale appropriately we can pass to the limit as $\epsilon \rightarrow 0$ to get weak solutions to the space-time lifted degenerate Lagrangian phase equation, see Definition~\ref{def: sldlp}.  Before explaining how this is done, let us explain how to make sense of the limiting equation weakly.  First note that if $\phi: \mathcal{X}\rightarrow \mathbb{R}$ is bounded and satisfies $\DDb\phi \geq - C\hat{\omega}$, then we can define
\[
(\pi_{X}^{*}\alpha + \DDb\phi)^{k}
\]
as a $(k,k)$ current for all $k$ using the Bedford-Taylor theory~\cite{BT}.  In particular, for such functions $\phi$
\[
(\pi_{X}^{*}\omega + \sqrt{-1}(\pi_{X}^{*}\alpha+\DDb\phi))^{n+1}
\]
defines a complex measure on $\mathcal{X}$ and hence the equation
\begin{equation}\label{eq: measureEquation}
{\rm Im}\left(e^{-\sqrt{-1}h(x,|t|)}(\pi_{X}^{*}\omega + \sqrt{-1}(\pi_{X}^{*}\alpha+\DDb\phi))^{n+1}\right)=0
\end{equation}
can be interpreted as an equality of measures.  With this in mind we have

\begin{thm}\label{thm: geoExistence}
Suppose $\phi_0, \phi_1 \in \mathcal{H}$, and $h(x,|t|) : \mathcal{X} \rightarrow \mathbb{R}$ is a smooth function satisfying structural constraints $(C1), (C2)$.  Then there exists an $S^1$ invariant function $\phi$ so that $\phi \in C^{1,\alpha}$, for all $\alpha \in (0,1)$, $\DDb\phi \in L^{\infty}(\mathcal{X}_{\epsilon},\hat{\omega})$, and $\phi$ solves
\begin{equation}\label{eq: genDegPhase}
{\rm Im}\left(e^{-\sqrt{-1}h(x,|t|)}(\omega+\sqrt{-1}(\alpha+\DDb\phi))^{n+1}\right) =0.
\end{equation}
pointwise a.e., and in the sense of pluripotential theory, with boundary values $\phi_0, \phi_1$. Equivalently, $\phi$ solves
\[
\widetilde{\Theta}_{\omega}(\hat{\omega}^{-1}\alpha_{\phi}) = h(x,|t|)
\]
in the sense of Harvey-Lawson's Dirichlet Duality.  In particular, for any $\phi_0, \phi_1 \in \mathcal{H}$ there is a unique weak geodesic joining $\phi_0, \phi_1$.
\end{thm}
\begin{proof}
Given $h(x,|t|) : \mathcal{X} \rightarrow \mathbb{R}$, we consider
\[
h_{\epsilon}(x,t) = h\left(x,\frac{t}{\epsilon}\right) : \mathcal{X}_{\epsilon} \rightarrow \mathbb{R}.
\]
Clearly we have
\[
\nabla^{X}\nabla_{\bar{t}} h_{\epsilon} = \frac{1}{\epsilon} \nabla^{X}\nabla_{\bar{t}} h \qquad \nabla_{t}\nabla_{\bar{t}} h_{\epsilon}= \frac{1}{\epsilon^2} \nabla_{t}\nabla_{\bar{t}} h.
\]
By Theorem~\ref{thm: DirProbXe} we have functions $\psi_{\epsilon}:\mathcal{X}_{\epsilon} \rightarrow \mathbb{R}$ solving
\[
F(\hat{\omega}^{-1}\alpha_{\psi_{\epsilon}})= h_{\epsilon}.
\]
Let $\phi_{\epsilon}  = \psi_{\epsilon}(x,\epsilon t) : \mathcal{X}\rightarrow \mathbb{R}$. Then $\phi_{\epsilon}$ solve the Lagrangian phase equation $F((\hat{\omega}_{\epsilon})^{-1}\alpha_{\phi_{\epsilon}}) = h$, 
or equivalently
\[
{\rm Im}\left(e^{-\sqrt{-1}h}\left(\pi_{X}^{*}\omega +\epsilon^{1}\sqrt{-1}dt\wedge d\bar{t} +\sqrt{-1}\left(\pi_{X}^{*}\alpha + \DDb\phi\right)\right)^{n+1}\right)=0
\]
on $\mathcal{X}$. By Theorem~\ref{thm: estSummary}, $\phi_{\epsilon}$ satisfies 
\[
\|\phi_{\epsilon} \|_{L^{\infty}} +\sup_{\mathcal{X}} |\nabla \phi_{\epsilon}|_{\hat{\omega}} + \| \DDb\phi_{\epsilon}\|_{L^{\infty}(\mathcal{X}, \hat{\omega})} \leq C
\]
for a uniform constant $C$ independent of $\epsilon$.  We can therefore take a limit as $\epsilon \rightarrow 0$ and get $\phi_{\epsilon} \rightarrow \phi$ where the convergence is uniform in $C^{1,\alpha}$.  Clearly
\[
C\hat{\omega} \geq \pi_{X}^{*}\alpha +\DDb\phi \geq - C\hat{\omega}
\]
and so by the continuity of the Monge-Amp\`ere operator along uniformly convergent sequences \cite{BT} we conclude that $\phi$ is a weak solution of~\eqref{eq: genDegPhase}.  Next we argue that $\phi$ is also a solution of the space-time lifted degenerate Lagrangian phase equation in the sense of Harvey-Lawson.  We refer the reader to the work of Rubinstein-Solomon \cite{RuSol} for the construction of the degenerate Lagrangian phase operator.  The two key properties we need are
\[
F((\hat{\omega}_{\epsilon})^{-1}\alpha_{\psi}) \rightarrow \widetilde{\Theta}(\alpha_{\psi}) \quad \text{ as } \epsilon \rightarrow 0
\]
by \cite[Theorem A.3]{RuSol}, and that $\widetilde{\Theta}(\cdot)$ is upper-semi continuous on the space of hermitian matrices \cite{RuSol}.  By the Harvey-Lawson theory, and \cite[Theorem 5.1]{RuSol} we need to show that if $u$ is a $C^{2}$ function defined on a ball $B\subset \mathcal{X}$, with $u \geq \phi$, and $u(p) = \phi(p)$, then
\begin{equation}\label{eq: TouchAbove}
\widetilde{\Theta}(\alpha + \DDb u)(p) \geq h(p)
\end{equation}
and similarly, that if $u$ touches $\phi$ from below, then 
\begin{equation}\label{eq: TouchBelow}
\widetilde{\Theta}\left(-\left(\alpha + \DDb u\right)\right)(p) \geq -h(p)
\end{equation}
Our proof of this is based an elementary result on the viscosity theory \cite[Proposition 2.9]{CafCab}, with the added complication that the background metric used to construct the elliptic operator is not constant. 

Everything is local, so we may assume that we're working in $B_1 \subset \mathbb{C}^{n}$, and that $p=0$.  Let $B_{r}$ denote a ball of radius $r$ centered at $0$. Suppose $u\geq \phi$ on $B_1$.  Fix $\delta>0$.  Since $\phi_\epsilon \rightarrow \phi$ uniformly on $B_1$, for any $\eta>0$ we can choose $\epsilon < \epsilon_0(\eta,\delta)$ sufficiently small so that
\[
u_{\eta}:= u+\eta |z|^2  \geq \phi_\epsilon \quad \text{ on } \del B_{\delta}.
\]
Since $u(p) = \phi(p)$, it follows that for $\epsilon$ sufficiently small depending on $\eta, \delta$, $u_{\eta}-\phi_\epsilon$ has an interior minimum at some point $p_{(\epsilon,\eta)} \in B_{\delta}(p)$.  For now, let us suppress the dependence on $\eta$, and write $p_{(\epsilon,\eta)}=  p_{\epsilon}$.  At this point we have
\begin{equation}\label{eq: uphiComp}
\alpha_{u}(p_\epsilon) + \eta \sum_{i=0}^{n}\sqrt{-1}dz_i \wedge d\bar{z}_i \geq \alpha_{\phi_{\epsilon}}(p_\epsilon) 
\end{equation}
Up to taking a subsequence we can assume that $p_{\epsilon} \rightarrow p_* \in \overline{B_{\delta}(p)} \subset B_{2\delta}(p)$ as $\epsilon \rightarrow 0$.  For $\epsilon$ sufficiently small we have
\[
\alpha_{u}(p_*)+ 2\eta \sum_{i=0}^{n}\sqrt{-1}dz_i \wedge d\bar{z}_i \geq \alpha_{u}(p_\epsilon)  + \eta \sum_{i=0}^{n}\sqrt{-1}dz_i \wedge d\bar{z}_i
\]
as $(1,1)$ forms on $\mathbb{C}^{n}$.  Now, since $\hat{\omega}_{\epsilon} = \pi_{X}^{*}\omega+ \epsilon^{2}\sqrt{-1}dt\wedge d\bar{t}$ is a product, there is a constant $C$ independent of $\epsilon$ so that
\[
-C |x-y| \leq \hat{\omega}_{\epsilon}(x)^{-1} - \hat{\omega}_{\epsilon}(y)^{-1} \leq C|x-y|, \qquad \hat{\omega}^{-1}_{\epsilon}(x) > C^{-1}\sum_{i=0}^{n} \sqrt{-1}dz_i \wedge d\bar{z}_i
\]
for any points $x,y \in B_1 \subset \mathbb{C}^{n+1}$, where $\hat{\omega}_{\epsilon}$ is regarded as a K\"ahler metric on $\mathbb{C}^{n+1}$.  Thus, for $\epsilon\ll \eta$ sufficiently small we have
\[
\begin{aligned}
& \hat{\omega}_{\epsilon}(p_*)^{-1}\left(\alpha_{u}(p_*)+ \left(4\eta \sum_{i=0}^{n}\sqrt{-1}dz_i \wedge d\bar{z}_i\right) \right)\\
&\geq  \hat{\omega}_{\epsilon}(p_{\epsilon})^{-1}\left(\alpha_{u}(p_*)+ 2\eta \sum_{i=0}^{n}\sqrt{-1}dz_i \wedge d\bar{z}_i \right)\\
 &\geq  \hat{\omega}_{\epsilon}(p_{\epsilon})^{-1}\left( \alpha_{u}(p_\epsilon)  + \eta \sum_{i=0}^{n}\sqrt{-1}dz_i \wedge d\bar{z}_i\right)
 \end{aligned}
 \]
Combining this inequality with~\eqref{eq: uphiComp} and applying the elliptic operator $F$ gives
\[
F\left((\hat{\omega}_{\epsilon})^{-1}\left(\alpha_{u}(p_*) + 4\eta \sum_{i=0}^{n}\sqrt{-1}dz_i \wedge d\bar{z}_i\right)\right) \geq F(\hat{\omega}_{\epsilon}^{-1}\alpha_{\phi_{\epsilon}})(p_{\epsilon}) =h(p_{\epsilon}),
\]
for all $\epsilon$ sufficiently small.  We now take a limit as $\epsilon \rightarrow 0$ to get
\begin{equation}\label{eq: ineqAteta}
\widetilde{\Theta}\left(\left(\alpha_{u}(p_*) + 4\eta \sum_{i=0}^{n}\sqrt{-1}dz_i \wedge d\bar{z}_i\right)\right) \geq h(p_{*})
\end{equation}
We now reinstate the dependence on $\eta$, and write $p_{*}= p_{\eta}$.  Since~\eqref{eq: ineqAteta} holds for all $\eta>0$, we can take a limit as $\eta \rightarrow 0$.  Up to taking a subsequence we can assume $p_{\eta} \rightarrow p_{\infty} \in \overline{B_{\delta}(p)}$.  By the upper-semi continuity of $\widetilde{\Theta}$ we get
\[
\widetilde{\Theta}\left(\alpha_{u}(p_\infty)\right) \geq h(p_{\infty}).
\]
This holds for all $\delta >0$, and so we may finally take a limit as $\delta \rightarrow 0$, applying the upper semi-continuity again to conclude
\[
\widetilde{\Theta}\left(\alpha_{u}(p)\right) \geq h(p).
\]
The same argument works to prove~\eqref{eq: TouchBelow}.
\end{proof}

The next corollary is essential for infinite dimensional GIT.

\begin{cor}
Let $\phi_0, \phi_1 \in \mathcal{H}$, and let $\phi(x,s)$ be a weak geodesic with $\phi(x,0) = \phi_0$, $\phi(x,1) = \phi_1$.  Then the functional $CY_{\mathbb{C}}$ is well-defined along the curve $\phi(x,s)$.  Furthermore $\mathcal{C}$ is affine, $\mathcal{J}$ is convex, and ${\rm Re}(Z), {\rm Im}(Z)$ are concave.
\end{cor}
\begin{proof}
That $CY_{\mathbb{C}}$ is well defined follows from the Bedford-Taylor theory \cite{BT}, together with the bounds  $\|\phi(x,s)\|_{L^{\infty}(X)} \leq C$, and $-C\omega \leq \ddb\phi$, as discussed above.  Next, let $\phi_{\epsilon}(x,s)$ be $\epsilon$-geodesics joining $\phi_0, \phi_1$.  Since $\phi_{\epsilon} \rightarrow \phi$ the Bedford-Taylor theory \cite{BT} implies that
\[
CY_{\mathbb{C}}(\phi_{\epsilon}) \rightarrow CY_{\mathbb{C}}(\phi)
\]
as $\epsilon \rightarrow 0$.  The properties of $\mathcal{C}, \mathcal{J}, {\rm Re}(Z), {\rm Im}(Z)$ along $\phi(x,s)$ follow from the corresponding properties along $\phi_{\epsilon}(x,s)$; see Proposition~\ref{prop: convexityOfFuncs}, and Corollary~\ref{cor: Zconcave} 
\end{proof} 

\subsection{Applications to Homogeneous Monge-Amp\`ere}
Our techniques can be used to give a simplified proof of the existence of geodesics in the space of K\"ahler metrics \cite{Chen}, in particular avoiding B\l ocki's gradient estimate \cite{Bl, Bl1}.  We briefly describe how this is done.

  Let $(X,\omega)$ be a K\"ahler manifold, and 
  \[
  \mathcal{H} = \{ \phi \in C^{\infty}(X,\mathbb{R}) : \omega_{\phi} = \omega+ \ddb \phi >0\}
  \]
   be the space of K\"ahler metrics.  A geodesic in this space with respect to the Donaldson-Mabuchi-Semmes metric is equivalent to a solution of the homogeneous complex Monge-Amp\`ere equation on $\mathcal{X} = X\times \mathcal{A}$.  That is, a solution of
\[
(\pi_{X}^{*}\omega+ \DDb\phi)^{n+1}=0 \quad \text{ on } \mathcal{X}
\]
with boundary values $\phi_0, \phi_1 \in \mathcal{H}$.  As above, we approximate this equation by the degenerating Monge-Amp\`ere equations
\[
(\pi_{X}^{*}\omega + \epsilon^{2}\sqrt{-1}dt\wedge d\bar{t}+ \DDb\phi)^{n+1}=(n+1)\epsilon^{2}\sqrt{-1}dt\wedge d\bar{t} \wedge \pi_{X}^{*}\omega^{n}.
\]
Rescaling $t\mapsto \epsilon t$, we can view this as the non-degenerate Monge-Amp\`ere equation
\[
(\hat{\omega}+ \DDb\phi)^{n+1}=\hat{\omega}^{n+1}
\]
on $\mathcal{X}_{\epsilon}$ with boundary values $\phi_0, \phi_1 \in \mathcal{H}$.  Here, as before 
\[
\hat{\omega} = \pi_{X}^{*}\omega +\sqrt{-1}dt\wedge d\bar{t}
\]
is a product metric on $\mathcal{X}_{\epsilon}$.   The $C^0$ estimate follows from the maximum principle, comparing with sub and supersolutions as in Section~\ref{sec: AnalPrelim}.  For the spatial $C^{2}$ estimate, we argue as in Section~\ref{sec: C2est}, applying the maximum principle to $\lambda_1$, the largest eigenvalue of
\[
\left(\hat{\omega}+ \DDb\phi\right)\big|_{TX}
\]
measured with respect to $\hat{\omega}$.  Let $F(M)= \log(\det(M))$.  Fix a point $(p_0, t_0)$ and local holomorphic normal coordinates $(z_0,\ldots,z_n)$ so that $\hat{\omega}_{\bar{j}i} = \delta_{\bar{j}i}$ and $\omega_{\phi} := (\hat{\omega}+ \DDb\phi)$ is diagonal with entries $\mu_0 \leq \cdots \leq \mu_n$.  Let $V_1 \in T_{p_0}X$ be the unit spatial eigenvector achieving $\lambda_1$   Let $F^{i\bar{j}}$ be the linearized operator of $F$.  Then following the computation in Section~\ref{sec: C2est} we compute
\[
F^{i\bar{j}}\nabla_{i}\nabla_{\bar{j}}\log(\lambda_1) \geq -C +\frac{1}{\lambda_1}\sum_{i,j=0}^{n} \frac{1}{\mu_i\mu_j}|\nabla_{V_1}(\omega_{\phi})_{\bar{j}i}|^{2} - \frac{1}{\lambda_1^2}\sum_{i=0}^{n}\frac{1}{\mu_i}|\nabla_{i}(\omega_{\phi})_{\bar{V_1}V_1}|^{2}.
\]
for a uniform constant $C$.  Using Cauchy-Schwarz we now estimate
\[
\begin{aligned}
\frac{1}{\mu_i}|\nabla_{i}(\omega_{\phi})_{\bar{V_1}V_1}|^{2} &= \bigg|\sum_{j} \overline{V_{1}^{j}}\nabla_{V_1}(\omega_{\phi})_{\bar{j}i}\bigg|^2\\
&= \sum_{0\leq j, \ell \leq n} V_{1}^{\ell}\overline{V_{1}^{j}}\nabla_{V_1}(\omega_{\phi})_{\bar{j}i}\overline{\nabla_{V_1}(\omega_{\phi})_{\bar{\ell}i}}\\
&\leq \left(\sum_{0\leq j, \ell \leq n}| V_{1}^{\ell}|^2\frac{\mu_{\ell}}{\mu_{j}}|\nabla_{V_1}(\omega_{\phi})_{\bar{j}i}|^{2}\right) \\
&= \lambda_1\sum_{0\leq j \leq n}\frac{1}{\mu_j} |\nabla_{V_1}(\omega_{\phi})_{\bar{j}i}|^{2}.
\end{aligned}
\]
Thus we have that $F^{i\bar{j}}\nabla_{i}\nabla_{\bar{j}}\log(\lambda_1) \geq -C$, and arguing as in the second author's proof of the Calabi conjecture \cite{Y} we get
\[
{\rm Tr}_{\hat{\omega}}\left(\hat{\omega}+ \DDb\phi\right)\big|_{TX} \leq Ce^{-C(\phi-\inf_{\mathcal{X}}\phi)}.
\]
In other words, for every $t\in \mathcal{A}$ we have an $L^{\infty}$ bound for $\Tr_{\hat{\omega}}\ddb\phi(t)$ on $X$.  By the elliptic theory applied on $(X,\omega)$ we conclude that $|\nabla_{X}\phi(t)|_{\omega}$ is uniformly bounded.  The estimate for $\nabla_{t}\phi$ is easily obtained from the maximum principle, as in Section~\ref{sec: C1est}.  The remainder of the argument is the same, applying the boundary estimates for complex Monge-Amp\`ere \cite{Chen, Guan}.

\section{Applications to algebraic obstructions}\label{sec: AlgObstr}

The goal of this section is to use the existence of sufficiently regular geodesic segments to find algebraic obstructions to the existence of solutions to dHYM.  Let $\Delta = \{ |t|<1\}$ be the annulus, $\mathcal{X} = X\times \Delta$, and $\pi_{\Delta}$ be the projection to the disk.  Inspired by Mumford \cite{Mum}, and Ross-Thomas \cite{RoTh}, we recall the notion of a flag ideal $\mathfrak{I} \subset \mathcal{X}\otimes \mathbb{C}[t]$.  Fix ideals $\mathfrak{J}_{0}\subset \mathfrak{J}_{1} \subset \cdots \subset \mathfrak{J}_{r-1} \subset \mathcal{O}_{X}$.  The we define a {\em flag ideal} $\mathfrak{I}$ by
\[
\mathfrak{I} = \mathfrak{J}_{0} + t\mathfrak{J}_{1} + \cdots + t^{r-1}\mathfrak{J}_{r-1} + (t^r) \subset \mathcal{O}_{X}.
\]
The ideal $\mathfrak{I}$ defines a subscheme of $X\times \Delta$ which is supported in $\pi_{\Delta}^{-1}(0)$.  We are going to use this data to define an infinite ray $\phi(s) \in \mathcal{H}$ for $s\in [0,+\infty)$.  We need the following lemma
\begin{lem}[Demailly-Paun, \cite{DP}]
There is an $S^1$ invariant function $\psi: X\times \Delta \rightarrow \mathbb{R}$ satisfying the following properties
\begin{itemize}
\item $\DDb\psi \geq - A \hat{\omega}$ for some $A>0$.
\item $\psi$  is smooth on $\mathcal{X} \backslash {\rm supp}(\mathfrak{I})$, and 
\item Near ${\rm Supp}(\mathfrak{I})$ we have
\[
\psi =\frac{1}{2\pi} \log\left(\sum_{\ell=0}^{r}|t|^{2\ell} \sum_{k=1}^{N_{\ell}}|f_{\ell, k}|^2\right) + C^{\infty}
\]
where, for each $k$,  $(f_{\ell, k})_{k =1}^{N_{\ell}}$ are local generators for $\mathfrak{J}_{\ell}$.
\end{itemize}
\end{lem}
Since the statement we need is not exactly in Demailly-Paun, we quickly sketch the necessary ingredients.
\begin{proof}
Fix a cover of $X$ by open balls $B_i$ so that on each $B_i$ the ideal sheaves $\mathfrak{J}_{\ell}$ have local generators $(f_{\ell, k})_{k =1}^{N_{\ell}}$.  Choose a partition of unity $\theta_j$ subordinate to $B_j$, and let $\theta(x)$ be a smooth function positive, non-increasing function on $\mathbb{R}$ with $\theta(x) = 1 $ for $x \in [-1/2, 1/2]$ and $\theta(x)=0$ for $|x|>\frac{3}{4}$.  Now consider
\[
\psi = \frac{1}{2\pi}\log \left( \theta(|t|^2)\cdot\sum_{j} \theta_j \sum_{\ell=0}^{r-1}|t|^{2\ell} \sum_{k=1}^{N_{\ell}}|f_{\ell, k}|^2 + |t|^{2r}\right).
\]
Then $\psi$ is clearly $S^1$ invariant, and by the calculations in \cite{DP} the first two conditions are satisfied also.
\end{proof}

Note that from the construction we have
\[
C\log(\epsilon) \leq \psi\bigg|_{\{|t|=\epsilon\}} \leq C
\]
for some constant $C$ independent of $\epsilon$.

To construct an infinite ray in $\mathcal{H}$ we choose a function $\phi_0 \in \mathcal{H}$, and consider
\begin{equation}\label{eq: modelCurve}
\Phi(t) = \phi_0 + \delta\psi(t).
\end{equation}
We choose $\delta >0$ sufficiently small as follows. If $\omega^{-1}\alpha_{\phi_0}$ has eigenvalues $\lambda_1, \ldots, \lambda_n$ satisfying
\[
\sum_{i=1}^{n}\arctan(\lambda_i) > \hat{\theta} -\frac{\pi}{2}+\eta_1
\]
then choose $\delta$ small so that 
\[
n\frac{\pi}{2} >  \Theta_{\omega}(\alpha_{\phi(t)}) \geq \sum_{i=1}^{n}\arctan(\lambda_i-\delta A) > \hat{\theta} -\frac{\pi}{2}.
\]
Then for all $t$ the function $\phi_0 + \delta \psi(t) \in \mathcal{H}$.  Setting $s=-\log|t|$ gives an infinite ray $\Phi(s) \in\mathcal{H}$. 
\begin{defn}
The curve $\Phi(s)$ constructed above will be called a {\em model curve} for the ideal $\mathfrak{I}$ emanating from $\phi_0$.
\end{defn}

\begin{rk}
We have constructed the model curve $\Phi$ by hand, and hence the number $\delta >0$ depends on analytic data.  However, we expect that there is an algebraic characterization of $\delta$ analogous to the Seshadri constant for an ample line bundle.  This  will be discussed in Section~\ref{sec: StabCond}.
\end{rk}

The next proposition shows that the limit slope of complexified Calabi-Yau functional along the curve $\phi(s)$ exists.

\begin{prop}\label{prop: limSlopeAY}
Consider the curve $\Phi(s)$ constructed above.  Let $\mu: \tilde{\mathcal{X}} \rightarrow \mathcal{X}$ be a log-resolution of singularities of the ideal sheaf $\mathfrak{I}$, so that $\mu^{-1}\mathfrak{I} = \mathcal{O}_{\tilde{\mathcal{X}}}(-E)$ for a simple normal crossings divisor $E$.  Then we have
\begin{equation}\label{eq: limSlopeModel}
\lim_{s\rightarrow \infty} \frac{d}{ds}CY_{\mathbb{C}}(\Phi(s)) = -\frac{\delta}{\pi} E.\left[\left(\mu^{*}[\omega] + \sqrt{-1}\left(\mu^{*}[\alpha] -\delta E\right)\right)^{n}\right].
\end{equation}
In particular, the limit exists, and the quantity on the right hand side is independent of the choice of log-resolution.
\end{prop}
\begin{proof}
In order to avoid carrying around extra minus signs, we will evaluate
\[
-\lim_{s\rightarrow \infty} \int_{X} \frac{\del \Phi}{\del s}  (\omega+ \sqrt{-1}\alpha_{\Phi})^{n}.
\]
where $s=-\log|t|$, and $\Phi$ is the curve~\eqref{eq: modelCurve}.  To do this we change variables, using $t$ in place of $s$.  We note that if we write $t=re^{i\theta}$, then
\[
\frac{\del \Phi}{\del s} \bigg|_{s=-\log\epsilon} = -r \frac{\del \Phi}{\del r}|_{r=\epsilon} = -\epsilon \frac{\del \Phi}{\del r}|_{r=\epsilon} 
\]
At the same time, since $\Phi$ is $S^{1}$ invariant we have $\del_{\bar{t}}\Phi = \frac{1}{2}e^{i\theta}\frac{\del \Phi}{\del r}$, and so
\[
\del_{\bar{t}} \Phi d\bar{t} = \frac{1}{2} \frac{\del \Phi}{\del r} (dr -\sqrt{-1} rd\theta)
\]
and so
\[
-\frac{\del}{\del s} \Phi\bigg|_{s=-\log\epsilon} d\theta = 2\sqrt{-1}\del_{\bar{t}} \Phi d\bar{t}\,\bigg|_{r=\epsilon}.
\]
Let $Y_{\epsilon} = X \times \{|t|=\epsilon\} \cong X \times S^1$, and  let $\iota: Y_{\epsilon} \rightarrow \mathcal{X}$ be the inclusion, and $\pi_{X}: Y_{\epsilon}\rightarrow X$ be the projection.  Let $(x_1,\ldots,x_n)$ be coordinates on $X$.  Then for any complex $(1,1)$ form $\beta$ on $\mathcal{X}$ we can write
\[
\beta = \beta_{\bar{X}X} + \beta_{\bar{X}t} + \beta_{\bar{t}X} + \beta_{\bar{t}t}
\]
where 
\[
\begin{aligned}
\beta_{\bar{X}X} &= \sum_{1\leq i,j \leq n} \beta_{\bar{j}i}dx_{i} \wedge d\bar{x}_{j} &\beta_{\bar{t}t} = \beta_{\bar{t}t}dt\wedge d\bar{t}\\
\beta_{\bar{X}t} &= \sum_{1\leq j \leq n} \beta_{\bar{j}t}dt\wedge d\bar{x}_{j}  &\beta_{\bar{t}X} = \sum_{1\leq j \leq n} \beta_{\bar{t}j}dx_j\wedge d\bar{t}.
\end{aligned}
\]
In this notation we have
\[
\iota^{*}\beta = \beta_{\bar{X}X} + \sum_{1\leq j \leq n} \beta_{\bar{j}t}\sqrt{-1} \epsilon e^{\sqrt{-1}\theta}d\theta\wedge d\bar{x}_{j} - \sum_{1\leq j \leq n} \beta_{\bar{t}j}\sqrt{-1}\epsilon e^{-\sqrt{-1}\theta}dx_j\wedge d\theta
\]
where we view $\beta_{\bar{X}X}$ as a form on $Y_{\epsilon}$ in the obvious way using $\pi_{X}$.  Furthermore, note that $\dbar_{X}$ makes sense as an operator on forms on $Y_{\epsilon}$; indeed, this is just the $\dbar$ operator of the natural CR structure on $Y_{\epsilon}$.  However, even if $\beta_{\bar{X}X}$ is closed on the fibers of $\mathcal{X}\rightarrow \Delta$, and $\beta$ is closed on $\mathcal{X}$, it need not be the case that $\dbar_{X}\iota^{*}\beta =0$ on $Y_{\epsilon}$ due to contributions from the $\beta_{\bar{X},t}, \beta_{\bar{t}X}$ components.  In any event, we have
\[
\begin{aligned}
&-\int_{X}\frac{\del \Phi}{\del s} \left(\omega+ \sqrt{-1}\alpha_{\Phi})^{n}\right) \bigg|_{s=-\log \epsilon}\\
&  = \frac{\sqrt{-1}}{\pi} \int_{Y_{\epsilon}}\iota^{*}(\del_{\bar{t}} \Phi d\bar{t})\wedge \pi_{X}^{*}(\omega+ \sqrt{-1}\alpha_{\Phi})^{n}\\
&= \frac{\sqrt{-1}}{\pi} \int_{Y_{\epsilon}}\iota^{*}\left[\del_{\bar{t}} \Phi d\bar{t}\wedge \left(\pi_{X}^{*}\omega+ \sqrt{-1}(\pi_{X}^{*}\alpha + \DDb\Phi)^{n}\right)\right]
\end{aligned}
\]
Let $\mu:\tilde{\mathcal{X}}\rightarrow \mathcal{X}$ be a log resolution of the ideal $\mathfrak{I}$ with the property that $\mu$ is an isomorphism away from $t=0$.  Then
\[
\mu^{-1}\mathfrak{I} = \mathcal{O}_{\tilde{\mathcal{X}}}(-E)
\]
 for a simple normal crossings divisor $E$ on $\tilde{\mathcal{X}}$ supported over $t=0$.  By the Poincar\'e-Lelong formula \cite{DemBook} we have
 \[
 \mu^{*}\left(\pi_{X}^{*}\omega+ \sqrt{-1}\left(\pi_{X}^{*}\alpha + \DDb\Phi\right)\right)= \left(\mu^{*}\pi_{X}^{*}\omega+ \sqrt{-1}\left(\mu^{*}\pi_{X}^{*}\alpha +\delta[E] - \delta \gamma_{E}\right)\right)
 \]
 where $[E]$ denotes the current of integration over $E$, and $\gamma_{E} \in c_{1}(E)$ is an $S^{1}$-invariant smooth $(1,1)$ form on $\tilde{\mathcal{X}}$.   Let $\tilde{Y}_{\epsilon} = \mu^{-1}(Y_{\epsilon})$ which is a smooth submanifold of $\tilde{\mathcal{X}}$, and $\mu: \tilde{Y}\rightarrow Y$ is an isomorphism.  Let $\tilde{\iota} : \tilde{Y} \rightarrow \tilde{\mathcal{X}}$.
 \[
 \begin{aligned}
 &-\int_{X}\frac{\del \Phi}{\del s}(\omega+ \sqrt{-1}\alpha_{\Phi})^{n}\ \bigg|_{s=-\log \epsilon}\\
 &=  \frac{\sqrt{-1}}{\pi} \int_{\tilde{Y}_{\epsilon}}\mu^{*}\iota^{*}\left(\del_{\bar{t}} \Phi d\bar{t}\wedge \left(\pi_{X}^{*}\omega+ \sqrt{-1}(\pi_{X}^{*}\alpha +\DDb\Phi)\right)^{n}\right)\\
 &=\frac{\sqrt{-1}}{\pi} \int_{\tilde{Y}_{\epsilon}}\tilde{\iota}^{*}\left(\mu^{*}(\del_{\bar{t}} \Phi d\bar{t})\wedge \left(\mu^{*}\pi_{X}^{*}\omega+ \sqrt{-1}(\mu^{*}\pi_{X}^{*}\alpha +\delta[E] -\delta\gamma_{E})\right)^{n}\right)\\
 &= \frac{\sqrt{-1}}{\pi} \int_{\tilde{Y}_{\epsilon}}\tilde{\iota}^{*}\left(\mu^{*}(\del_{\bar{t}} \Phi d\bar{t})\wedge\left(\mu^{*}\pi_{X}^{*}\omega+ \sqrt{-1}(\mu^{*}\pi_{X}^{*}\alpha -\delta\gamma_{E})\right)^{n}\right),
 \end{aligned}
 \]
 where in the last line we used that  the support of $[E]$ is over $0$, and hence disjoint from $\tilde{Y}_{\epsilon}$.
 
Let us digress briefly to  consider the following integral
 \[
 \frac{\sqrt{-1}}{\pi} \int_{\tilde{Y}_{\epsilon}}\tilde{\iota}^{*}\mu^{*}\overline{D}(\Phi)\wedge \left(\mu^{*}\pi_{X}^{*}\omega+ \sqrt{-1}(\mu^{*}\pi_{X}^{*}\alpha -\delta \gamma_{E})\right)^n.
 \]
 In order to lighten notation, we will suppress the pull-backs, and take them as understood.  Fix a K\"ahler metric on $\tilde{\mathcal{X}}$, which we can take to be of the form $\tilde{\omega} = \mu^{*}(\omega+ \sqrt{-1}dt\wedge d\bar{t}) -\hat{\delta} \beta_E$ for $\beta_E$ some smooth representative of $c_1(E)$, and $\hat{\delta}>0$ sufficiently small \cite{DP}.  Since $\gamma_{E}$ is a smooth form on $\tilde{X}$ we can fix a constant $C_0$ so that
\begin{equation}
\label{eq: gammaEest}
\sup_{\tilde{\mathcal{X}}}|\bar{\nabla}\gamma_{E}|_{\tilde{\omega}} <C_0.
\end{equation}
By construction it is clear that $\tilde{\mathcal{X}}$ admits a fibration over $\Delta$ with fibers $\tilde{\mathcal{X}}_{t} = X$ for $t\ne 0$.  As above, we fix coordinates $(x_1,\ldots,x_n)$ on $X$, and $t$ on $\Delta$.  Since $\mu$ is an isomorphism away from $t=0$, pulling back by $\mu$ allows us to view $(x_1,\ldots,x_n)$ and $t$ is a local holomorphic coordinate system in a neighbourhood of (any point) of $\tilde{Y}_{\epsilon}$.  In these coordinates we write
\[
\overline{D}(\Phi) = \del_{\bar{t}}\Phi d\bar{t}+ \dbar_{X}\Phi
\]
Since $\tilde{Y}_{\epsilon} = X\times S^{1}$ is a product, we can integrate by parts over $X$ to get
\[
\begin{aligned}
& \frac{\sqrt{-1}}{\pi} \int_{\tilde{Y}_{\epsilon}}\dbar_{X}\Phi\wedge (\omega+ \sqrt{-1}(\alpha -\delta \gamma_{E}))^{n}\\
 &=  -\frac{\sqrt{-1}}{\pi} \int_{\tilde{Y}_{\epsilon}}\Phi\wedge \dbar_{X}(\omega+ \sqrt{-1}(\alpha -\delta \gamma_{E}))^{n}.\\
 \end{aligned}
 \]
 Since $\alpha, \omega$ are pulled back from $X$, and closed along the fibers we have
 \[
 \dbar_{X}\omega = 0 = \dbar_{X}\alpha.
 \]
 However, $\dbar_{X}\gamma_{E} \ne 0$ in general.  Instead, we use the estimate~\eqref{eq: gammaEest} to bound
 \[
| \dbar_{X}\left(\omega+ \sqrt{-1}(\alpha -\delta \gamma_{E})\right)^{n}| \leq C d{\rm Vol}_{(\tilde{Y}_{\epsilon},\tilde{\omega})}
\]
for a constant $C$ independent of $\epsilon$.  Here $d{\rm Vol}_{(\tilde{Y}_{\epsilon},\tilde{\omega})}$ denotes the volume form on $\tilde{Y}_{\epsilon}$ induced by the K\"ahler metric $\tilde{\omega}$.  By construction we have that
\[
C'\log(\epsilon) \leq \Phi\bigg|_{Y_{\epsilon}} \leq C'
\]
for a uniform constant $C'$.  Combining these estimates gives
\[
\begin{aligned}
& \frac{\sqrt{-1}}{\pi} \int_{\tilde{Y}_{\epsilon}}\overline{D}(\Phi)\wedge \left(\omega+ \sqrt{-1}(\alpha -\delta \gamma_{E})\right)^{n}\\
&= \frac{\sqrt{-1}}{\pi} \int_{\tilde{Y}_{\epsilon}}(\del_{\bar{t}}\Phi d\bar{t}\wedge \left(\omega+ \sqrt{-1}(\alpha -\delta \gamma_{E})\right)^{n} + O(\epsilon\log(\epsilon)).
\end{aligned}
 \]
 In particular, we have that
 \[
 \begin{aligned}
 &-\lim_{\epsilon \rightarrow 0} \int_{X}\frac{\del \Phi}{\del s} (\omega+ \sqrt{-1}\alpha_{\Phi})^{n} \bigg|_{s=-\log \epsilon}\\
&  = \lim_{\epsilon \rightarrow 0}\frac{\sqrt{-1}}{\pi} \int_{\tilde{Y}_{\epsilon}}\overline{D}(\Phi)\wedge \left(\omega+ \sqrt{-1}(\alpha -\delta \gamma_{E})\right)^n
\end{aligned}
\]
The latter integral can be evaluated using integration by parts and the Poincar\'e-Lelong formula. Integration by parts gives
\[
\begin{aligned}
&\frac{1}{\pi}\int_{\mu^{-1}(X\times \{|t|\leq \epsilon\})}\DDb\Phi \wedge \left(\omega+ \sqrt{-1}(\alpha -\delta \gamma_{E})\right)^{n}\\
&= \frac{\sqrt{-1}}{\pi} \int_{\tilde{Y}_{\epsilon}}\overline{D}(\Phi)\wedge \left(\omega+ \sqrt{-1}(\alpha -\delta \gamma_{E})\right)^{n}
\end{aligned}
\]
since  $\left(\omega+ \sqrt{-1}(\alpha -\delta \gamma_{E})\right)^n$ is closed on $\tilde{\mathcal{X}}$.  By the Poincar\'e-Lelong formula (recall that $\Phi$ is pulled back to $\tilde{\mathcal{X}}$) we also have
\[
\begin{aligned}
&\frac{1}{\pi}\int_{\mu^{-1}(X\times \{|t|\leq \epsilon\})}\DDb\Phi \wedge \left(\omega+ \sqrt{-1}(\alpha -\delta \gamma_{E})\right)^n\\
&= \frac{1}{\pi}\int_{\mu^{-1}(X\times \{|t|\leq \epsilon\})}(\delta[E]-\delta \gamma_{E}) \wedge \left(\omega+ \sqrt{-1}(\alpha -\delta \gamma_{E})\right)^n\\
&= \frac{\delta}{\pi} \int_{E}\left(\omega+ \sqrt{-1}(\alpha -\delta \gamma_{E})\right)^n + O(\epsilon)
 \end{aligned}
 \]
 where in the last line we used that $\alpha, \omega, \gamma_E$ are smooth on $\tilde{\mathcal{X}}$.  Taking the limit as $\epsilon\rightarrow 0$ we obtain that the limit slope of $CY_{\mathbb{C}}$ along the curve~\eqref{eq: modelCurve} is computed in terms of intersection numbers as
 \[
 \lim_{\epsilon \rightarrow 0} \int_{X}\frac{\del \Phi}{\del s} \left((\omega+ \sqrt{-1}\alpha_{\Phi})^{n}\right) \bigg|_{s=-\log \epsilon}= -\frac{1}{\pi}\delta E. \left([\omega]+ \sqrt{-1}([\alpha] -\delta E)\right)^{n}. 
 \]
 To see that this intersection number is independent of the choice of log resolution, one only needs to observe that it is equal to the limit slope of $CY_{\mathbb{C}}$ along the curve~\eqref{eq: modelCurve}.
\end{proof}

In certain cases we can simplify the leading term (in $\delta$) of the expression.

\begin{lem}\label{lem: Vsubvar}
Suppose $V\subset X$ is an irreducible analytic subvariety of pure dimension $0<p<n$.  Consider the model curve associated to $\mathfrak{I}  = \mathfrak{J}_{V} + (t)$, where $\mathfrak{J}_{V}$ is the ideal sheaf of $V$.  As above, let $\mu$ be a log resolution of $\mathfrak{I}$.  Then we have
\[
\begin{aligned}
 &\frac{\delta}{\pi} E.\left(\mu^{*}[\omega] + \sqrt{-1}\left(\mu^{*}[\alpha] -\delta E\right)\right)^{n}\\
 & = \delta^{n-p}\binom{n}{n-p}e^{\sqrt{-1}(n-p)\frac{\pi}{2}}\int_{V}(\omega+\sqrt{-1}\alpha)^{p} + O(\delta^{n-p+1})
 \end{aligned}
 \]
 \end{lem}
 \begin{proof}
 We only need to compute the leading order term of the left hand side. Expanding we have
 \[
 \begin{aligned}
&E.\left(\mu^{*}[\omega] + \sqrt{-1}\left(\mu^{*}[\alpha] -\delta E\right)\right)^{n} \\
&= E.\sum_{j=0}^{n}(\sqrt{-1})^{j}\binom{n}{j}\delta^{j}(\mu^{*}\omega+\sqrt{-1}\mu^{*}\alpha)^{n-j}.(-E)^{j}\\
\end{aligned}
\]
 If $k < {\rm codim}_{X\times \Delta}( V\times\{0\}) = n+1-p$ then we have $E^{k}=0$, and so the leading order term is the $j=n-p$ term in the above sum.  We get
 \[
\delta^{n-p}\binom{n}{n-p}(\sqrt{-1})^{n-p} E.(-E)^{n-p}.(\mu^{*}\omega+\sqrt{-1}\mu^{*}\alpha)^{p} + O(\delta^{n-p+1}).
\]
Now a standard computation in intersection theory \cite{Ful} shows that
\[
 E.(-E)^{n-p}.(\mu^{*}\omega+\sqrt{-1}\mu^{*}\alpha)^{p} = \int_{V}(\omega+\sqrt{-1}\alpha)^{p}.
 \]
 \end{proof}

In order to produce obstructions from this data we need to use the existence of sufficiently regular geodesics (or smooth $\epsilon$-regularized geodesics).  Suppose that $\phi_0$ is a solution of the dHYM equation, and  let $\Phi(s) \in \mathcal{H}$ be a model curve as constructed above~\eqref{eq: modelCurve} emanating from $\phi_0$.  For each $s \in (0,\infty)$ we let $\phi_s(t)$ be a $\epsilon$-geodesic in $\mathcal{H}$ with the property that
\[
\phi_s(0) = \phi_0 \qquad \phi_s(s) = \Phi(s).
\]
By Proposition~\ref{prop: convexityOfFuncs}, $\mathcal{J}(\phi_s(t))$ is strictly convex and has $\frac{d}{dt}|_{t=0}\mathcal{J}(\phi_s(t))=0$.  Therefore
\[
 \frac{d}{dt}\mathcal{J}(\phi_s(t))>0 \quad \text{ for all } t \in (0,s),
 \]
 and we conclude that
 \[
\frac{ \mathcal{J}(\Phi(s)) - \mathcal{J}(\phi_0)}{s} >0.
\]
 Note that $\mathcal{J}(\Phi(s))$ can only be bounded when the limit slope
 \[
 \frac{\delta}{\pi}E. {\rm Im}\left(e^{-\sqrt{-1}\hat{\theta}}\left[\left(\mu^{*}[\omega] + \sqrt{-1}\left(\mu^{*}[\alpha] -\delta E\right)\right)^{n}\right]\right)
 \] 
 as computed in Proposition~\ref{prop: limSlopeAY} is zero.  Otherwise by L'H\^opital's rule we have
 \[
 0 \leq \lim_{s\rightarrow \infty} \frac{ \mathcal{J}(\Phi(s)) - \mathcal{J}(\Phi_0)}{s} = \lim_{s\rightarrow \infty} \frac{d}{ds}\mathcal{J}(\Phi(s))
 \]
 and this limit is computed by Proposition~\ref{prop: limSlopeAY}.  Summarizing we have proved;
 \\

\begin{prop}~\label{prop: limSlopeJ}
 Suppose $[\alpha]$ admits a solution of the deformed Hermitian-Yang-Mills equation.  Let $\mathfrak{I}$ be a flag ideal over $X\times \Delta$ and $\mu$ a log-resolution $\mathfrak{I}$, as above.  Then we have 
 \[
  \frac{\delta}{\pi} E.{\rm Im}\left[e^{-\sqrt{-1}\hat{\theta}}\left(\mu^{*}[\omega] + \sqrt{-1}\left(\mu^{*}[\alpha] -\delta E\right)\right)^{n}\right] \geq 0
\]
for all $\delta$ sufficiently small.
\end{prop}
 
 Evidently, this produces obstructions to the existence of solutions to the deformed Hermitian-Yang-Mills equation.  We can also evaluate the limit of the $\mathcal{C}$ and $Z$ functionals.
 
 \begin{prop}\label{prop: limSlopeC}
Suppose $\mathcal{H}$ is not empty.  Then for every flag ideal over $X\times \Delta$ and $\mu$ a log-resolution $\mathfrak{I}$ as above, we have
\[
-E.{\rm Re}\left(e^{-\sqrt{-1}\hat{\theta}}\left(\mu^{*}[\omega] + \sqrt{-1}\left(\mu^{*}[\alpha] -\delta E\right)\right)^{n}\right) \leq 0
\]
and
\[
- E.{\rm Im}\left(e^{-\sqrt{-1}\frac{n\pi}{2}}\left(\mu^{*}[\omega] + \sqrt{-1}\left(\mu^{*}[\alpha] -\delta E\right)\right)^{n}\right) \geq 0
\]
for all $\delta>0$ sufficiently small.
\end{prop}
\begin{proof}
Fix $\phi_0 \in \mathcal{H}$, and let $\Phi(s)$ be the model curve associated to the flag ideal $\mathfrak{I}$.  Explicitly,
\[
\phi(s) = \phi_0 + \frac{\delta}{2\pi}\log \left( \theta(e^{-2s})\cdot\sum_{j} \theta_j \sum_{\ell=0}^{r-1}e^{-2\ell s} \sum_{k=1}^{N_{\ell}}|f_{\ell, k}|^2 + e^{-2rs}\right).
\]
We compute
\[
\begin{aligned}
\frac{d}{ds}e^{\frac{2\pi}{\delta}\phi(s)} &= -2e^{-2s}\theta'(e^{-2s})\cdot\sum_{j} \theta_j \sum_{\ell=0}^{r-1}e^{-2\ell s} \sum_{k=1}^{N_{\ell}}|f_{\ell, k}|^2\\
&+  \theta(e^{-2s})\cdot\sum_{j} \theta_j \sum_{\ell=0}^{r-1}(-2\ell)e^{-2\ell s} \sum_{k=1}^{N_{\ell}}|f_{\ell, k}|^2 -2re^{-2rs}
\end{aligned}
\]
The first line is non-negative, while the second line is strictly negative.  For $s>s_{*} =\frac{1}{2}\log(2)$, we have $\theta'(e^{-2s})=0$, and so $\frac{d}{ds}\phi(s) <0$.  From the definition of $\mathcal{C}$ we get
\[
\frac{d}{ds}\mathcal{C}(\phi(s)) <0 \quad \text { for all } s>\frac{1}{2}\log(2). 
\] 
In particular $\mathcal{C}(\phi(s)) < \mathcal{C}(\phi(s_*))$.  Let $\phi_s(t)$ be an $\epsilon$-geodesic joining $\phi(s_*)$ to $\phi(s)$.  Since $\mathcal{C}$ is affine along $\epsilon$-geodesics we have
\[
\frac{d}{ds}\mathcal{C}(\phi(s)) = \frac{\mathcal{C}(\phi(s)) - \mathcal{C}(\phi(s_*))}{s-s_*}<0.
\]
Therefore
\[
\lim_{s\rightarrow \infty}\frac{d}{ds}\mathcal{C}(\phi(s)) \leq 0.
\]
Substituting the expression for the limit slope of $CY_{\mathbb{C}}$ gives the result.  A similar analysis applies to the functional $Z$ of Definition~\ref{defn: centChargFunc}.  For $s>s_{*}$ we have
\[
\frac{d}{ds}{\rm Im}Z(\phi(s)) >0
\]
since the variation of ${\rm Im}(Z)$ is defined by integration against a negative measure.  Thus ${\rm Im}Z(\phi(s)) - {\rm Im}Z(\phi(s*)) >0$ and so
\[
\lim_{s\rightarrow \infty} \frac{{\rm Im}(Z(\phi(s)))}{s} \geq 0.
\]
Plugging in the formula for the limit slope of $CY_{\mathbb{C}}$ finishes the proof.
\end{proof}
\begin{rk}
Proposition~\ref{prop: limSlopeC} does not give particularly interesting obstructions to the existence of solutions of dHYM.  However, it does give interesting algebraic obstructions to the existence of functions in $\mathcal{H}$.  This result will play a role in our discussion of stability conditions in Section~\ref{sec: StabCond}.
\end{rk}

We can improve the inequalities in Propositions~\ref{prop: limSlopeJ} and~\ref{prop: limSlopeC} to strict inequalities using a perturbation argument.  Suppose for the sake of contradiction that $[\alpha]$ admits a solution of deformed Hermitian-Yang-Mills, but
\[
 E.{\rm Im}\left[e^{-\sqrt{-1}\hat{\theta}}\left(\mu^{*}[\omega] + \sqrt{-1}\left(\mu^{*}[\alpha] -E\right)\right)^{n}\right]=0,
\]
or that $\mathcal{H}$ is not empty, but
\[
 E.{\rm Re}\left[e^{-\sqrt{-1}\hat{\theta}}\left(\mu^{*}[\omega] + \sqrt{-1}\left(\mu^{*}[\alpha] -E\right)\right)^{n}\right]=0
 \]
 or
 \[
  E.{\rm Im}\left[e^{-\sqrt{-1}n\frac{\pi}{2}}\left(\mu^{*}[\omega] + \sqrt{-1}\left(\mu^{*}[\alpha] -E\right)\right)^{n}\right]=0.
\]
Here we have suppressed the dependence on $\delta$, and consider $E$ as a $\mathbb{R}$-divisor.   For $z\in \mathbb{C}$ close to $0$, consider the rational function
\[
 \frac{p(z)}{q(z)} = \frac{E.\left(\mu^{*}\omega+ \sqrt{-1}(\mu^{*}\alpha-E) +z\mu^{*}\omega\right)}{(\omega+\sqrt{-1}\alpha + z\omega)^{n}.[X]}
\]
Note that for $|z|$ sufficiently small we have that $(1+{\rm Re}(z))\omega$ is K\"ahler.  If $[\alpha]$ admits a solution of dHYM with $\hat{\theta}>(n-1)\frac{\pi}{2}$, then so does $[\alpha]+{\rm Im}(z)[\omega]$ for $|z|$ sufficiently small.  If $\mathcal{H}(0) = \mathcal{H}(\alpha, \omega)$ is non-empty, then this is also true of $\mathcal{H}(z) = \mathcal{H}(\alpha+{\rm Im}(z)\omega, (1+{\rm Re}(z))\omega)$.  Furthermore, $p(z)/q(z)$ is holomorphic for $|z|$ sufficiently small.    Thus, by Proposition~\ref{prop: limSlopeJ}, and/or Proposition~\ref{prop: limSlopeC} we must have
\[
{\rm Im}\left(\frac{p(z)}{q(z)}\right) \geq 0 \qquad \text{ for $|z|$ sufficiently close to } 0 \in \mathbb{C}
\]
if $[\alpha]$ admits a solution of dHYM, and if $\mathcal{H} \ne \emptyset$ then
\[
\begin{aligned}
{\rm Re}\left(\frac{p(z)}{q(z)}\right) &\geq 0 \qquad \text{ for $|z|$ sufficiently close to } 0 \in \mathbb{C}\\
{\rm Im}\left(e^{-\sqrt{-1}\frac{n\pi}{2}}p(z)\right) &\leq 0 \qquad \text{ for $|z|$ sufficiently close to } 0 \in \mathbb{C}.
\end{aligned}
\]
If any of these three functions vanishes when $z=0$ then we get a contradiction to the maximum principle unless $p(z) \equiv 0$.  It suffices to prove that this is not the case unless $\mathfrak{I} = (t^{r})$ for some $r$.  Assume that $E\ne 0$.  Writing out the numerator (suppressing $\mu^{*}$) we get
\begin{equation}\label{eq: p(z)expanded}
p(z) = \sum_{j=0}^{n}\binom{n}{j}z^{j}E.(\omega+\sqrt{-1}(\alpha-E))^{n-j}.\omega^{j}
\end{equation}
It suffices to prove that the coefficient $z^{j}$ is not zero for one $j$. For $0<\beta \ll 1$, the class $[\mu^{*}\hat{\omega}-\beta E]$ is K\"ahler on $\tilde{\mathcal{X}}$ (see e.g. \cite{DP}), and hence $[\mu^{*}\omega-\beta E]$ is positive on the fibers of $\tilde{\mathcal{X}}/\Delta$.  Since $E$ is an effective divisor in a fiber of $\tilde{\mathcal{X}}\rightarrow \Delta$
\[
(A\mu^{*}\omega-\beta E)^n.E >0 \text{ for all } A\geq 1.
\]
If $(\mu^{*}\omega)^{n}.E >0$, then there is an irreducible component $E^{*}$ of $E$ so that $\mu: E^{*} \rightarrow \mu(E^{*})\subset X\times \{0\}$ is an isomorphism at the generic point of $E^{*}$.  Since $X$ is connected this implies $\mu(E^{*})=X\times\{0\}$, and so $\mathfrak{I}$ is supported on $X\times\{0\}$.  Since $\mathfrak{I}$ is a flag ideal, this implies 
\[
\mathfrak{I} = (t^{r})
\]
for some $r$, and hence the model curve $\Phi(t)$ is trivial, in the sense the $\alpha+\ddb \phi(t) = \alpha+\ddb \phi(0)$ as $(1,1)$ forms on $X$.
  
If instead $(\mu^{*}\omega)^{n}.E =0$, then we can choose $0\leq k <n$ to be the largest number so that $(\mu^{*}\omega)^{k}.E^{n-k} \ne 0$.  We must have
\[
E.(\mu^{*}\omega)^{k}.(-E)^{n-k} >0.
\]
For all $C>0$ sufficiently large we have $2C\omega > C\omega+\alpha > \omega$.  It follows that $2CA\mu^{*}\omega-\beta E> A(C\mu^{*}\omega+\mu^{*}\alpha)-\beta E > A\mu^{*}\omega- \beta E$.  Since $E$ is effective we get
\[
\left(2CA\mu^{*}\omega-\beta E\right)^n.E >\big(A(C\mu^{*}\omega+\mu^{*}\alpha)-\beta E \big)^n.E > \left(A\mu^{*}\omega- \beta E\right)^n.E.
\]
By assumption the terms on the left and right grow like $A^{k}$ for $A\gg 1$, and $C\gg 1$.  By comparing terms we get
\[
(\mu^{*}\omega)^{\ell}.(\mu^{*}\alpha)^{j}.E^{n+1-(\ell+j)}=0\qquad \text{ whenever } n\geq \ell+j >k.
\]
We now expand the coefficient of $z^{k}$ in~\eqref{eq: p(z)expanded} to get
\[
\sum_{m=0}^{n-k}\binom{n-k}{m}(\sqrt{-1})^{n-k-m} E.(\mu^{*}\omega)^{m+k}(\mu^{*}\alpha-E)^{n-k-m}
\]
Note that for $m >0$, the intersection product contains $(\mu^{*}\omega)^{m+k}.(\mu^*{\alpha})^{j}$ for $j\geq 0$ and hence vanishes.  The only non-zero term in the sum is therefore the $m=0$ term.  Expanding this term and using that $(\mu^{*}\omega)^{k}.(\mu^{*}\alpha)^{j}.E^{n+1-(k+j)}$ vanishes whenever $j>0$, we get that the only non-zero term is $\omega^{k}.E.(-E)^{n-k} \ne 0$.  Thus $p(z)$ is not identically zero. 

\begin{thm}\label{thm: obstrThm}
Let $\mathfrak{I}$ be a flag ideal as above, and $\mu:\tilde{\mathcal{X}}\rightarrow X\times \Delta$ a log-resolution of singularities so that $\mu^{-1}\mathfrak{I} = \mathcal{O}_{\tilde{\mathcal{X}}}(E)$.  If the space $\mathcal{H}$ is non-empty, then 
\[
\begin{aligned}
-E.{\rm Re}\left[e^{-\sqrt{-1}\hat{\theta}}\left(\mu^{*}[\omega] + \sqrt{-1}\left(\mu^{*}[\alpha] -\delta E\right)\right)^{n}\right]&\leq0\\
-E.{\rm Im}\left[e^{-\sqrt{-1}\frac{n\pi}{2}}\left(\mu^{*}[\omega] + \sqrt{-1}\left(\mu^{*}[\alpha] -\delta E\right)\right)^{n}\right]& \geq0
\end{aligned}
\]
with equality if and only if $\tilde{\mathcal{X}} \cong X\times \Delta$.  Furthermore, if $[\alpha]$ admits a solution of the dHYM equation we must have
\[
E.{\rm Im}\left[e^{-\sqrt{-1}\hat{\theta}}\left(\mu^{*}[\omega] + \sqrt{-1}\left(\mu^{*}[\alpha] -\delta E\right)\right)^{n}\right]\geq 0
\]
with equality if and only if $\mathfrak{I} = (t^{r})$ for some $r>0$.
\end{thm}

Combining this theorem with Lemma~\ref{lem: Vsubvar} we have
\begin{cor}\label{cor: subVarObstr}
Let $V$ be an irreducible analytic subset $V\subset X$ with $\dim V=p<n$.  If the class $[\alpha]$ has $\mathcal{H} \ne \emptyset$, then
\[
\begin{aligned}
{\rm Re}\left[e^{-\sqrt{-1}(\hat{\theta}-(n-p)\frac{\pi}{2})}\int_{V}(\omega+\sqrt{-1}\alpha)^{p}\right]&>0\\
{\rm Im}\left[e^{-\sqrt{-1}\frac{p\pi}{2}}\int_{V}(\omega+\sqrt{-1}\alpha)^{p}\right]&<0.
\end{aligned}
\]
Furthermore, if $[\alpha]$ admits a solution of the deformed Hermitian-Yang-Mills equation, then
\[
{\rm Im}\left[e^{-\sqrt{-1}(\hat{\theta}-(n-p)\frac{\pi}{2})}\int_{V}(\omega+\sqrt{-1}\alpha)^{p}\right]>0.
\]
\end{cor}

\section{Stability conditions}\label{sec: StabCond}

In this section we will attempt to synthesize the obstructions from the previous section into a coherent algebraic framework.  This will inevitably lead us to discuss the relationship with categorical stability conditions \cite{Br}.  In order to simplify the discussion we will focus primarily on interpreting the rather simple set of obstructions obtained in Corollary~\ref{cor: subVarObstr}, and comment on the more general situation of flag ideals toward the end of the section.  Furthermore, in order to facilitate our discussion of Bridgeland stability we will now restrict to the case when $[\alpha]= c_1(L)$, but we caution the reader that this is purely aesthetic and everything can be carried over to general classes in $H^{1,1}(X,\mathbb{R})$.  This more general theory should be thought of as analogous to Bridgeland stability with non-zero $B$-field, which we will comment on at the end of the section.  Introduce the following notation; for an analytic subset $V\subset X$ define
\[
Z_{V}(L) = -\int_{V}e^{-\sqrt{-1}\omega}ch(L)
\]
and note that if the dimension of $V$ is $p$ then
\begin{equation}\label{eq: intVsCenChar}
\int_{V}(\omega+\sqrt{-1}c_1(L))^{p} = \sqrt{-1}^{p}\int_{V}e^{-\sqrt{-1}\omega}ch(L) = e^{\sqrt{-1}(p-2)\frac{\pi}{2}}Z_{V}(L).
\end{equation}
As a first step we need to discuss the problem of determining the lifted angle algebraically.  Recall that the angle on $X$ is defined by
\[
\int_{X}(\omega +\sqrt{-1}\alpha)^{n} \in \mathbb{R}_{>0}e^{\sqrt{-1}\hat{\theta}}
\]
and we noted that if $\mathcal{H} \ne \emptyset$, then $\hat{\theta}$ could be uniquely lifted to a $\mathbb{R}$-valued angle (see Lemma~\ref{lem: angLift}).  We now give a (conjectural) algebraic construction for determining the lifted angle of any irreducible analytic subvariety.  Fix such a $V$ of dimension $p$ and consider the path
\[
Z_{V}(t) = -\int_{V}e^{-\sqrt{-1}t\omega}ch(L), \qquad t\in [1,\infty],
\]
which interpolates with the large radius limit (see \cite{Bay} for related ideas on the algebraic side).
\begin{defn}
We define the algebraic lifted angle $\hat{\theta}_{V}(\alpha)$ to be the winding angle of the path $Z_{V}(t)$ as $t$ runs from $+\infty$ to $1$, provided $Z_{V}(t)$ does not pass through the origin. We define the slicing angle $\phi_{V}(L)$ by
\[
\phi_{V}(L) = \hat{\theta}_{V}(L) - \frac{\pi}{2}\left({\rm dim}V-2\right).
\]
\end{defn}

The use of the factor $({\rm dim}V-2)\frac{\pi}{2}$ appearing on the right hand side is motivated by~\eqref{eq: intVsCenChar}. 

  If  ${\rm dim} V=1$ it is trivial that the lifted angle is always well-defined.  If ${\rm dim} V=2$, the lifted angle is always well-defined by the Hodge Index Theorem.  As soon as we reach dimension $3$, however, it is easy to construct examples of classes $c_1(L)$ and K\"ahler manifolds $(X,\omega)$ for which the lifted angle is not defined.  Examples of this phenomenon occur already on ${\rm Bl}_{p} \mathbb{P}^{3}$; see Example~\ref{eq: blp3} below.  However, if we assume that $L$ admits a solution of dHYM with phase $\hat{\theta} \in (\frac{\pi}{2}, \frac{3\pi}{2})$, then the lifted angle is well-defined as follows from the  Chern number inequality of the authors and Xie.  

\begin{prop}[Collins-Xie-Yau, \cite{CXY}]\label{prop: chernNum}
Suppose $(X,\omega)$ is a K\"ahler 3-fold, and $L$ admits a metric $h$ solving the deformed Hermitian-Yang-Mills equation with $\hat{\theta} \in (\frac{\pi}{2}, \frac{3\pi}{2})$.  Then we have
\begin{equation}\label{eq: chernNum}
\left(\int_{X}\omega^{3}\right)\left(\int_{X} \frac{c_{1}(L)^{3}}{6}\right) < 3\left(\int_{X}\frac{c_1(L)^2\wedge \omega}{2}\right)\left(\int_{X}c_1(L)\wedge \omega^{2}\right).
\end{equation}
Furthermore, the algebraic lifted angle equals the lifted angle.
\end{prop}

Chern number inequalities of this type, involving the higher Chern classes $ch_3$ have played an important role in the study of Bridgeland stability conditions \cite{BMT}.

Returning to our discussion, and taking $V=X$ gives
\[
Z_{X}(L) = -e^{-\sqrt{-1}\frac{n\pi}{2}}e^{\sqrt{-1}\hat{\theta}}R_{X}
\]
where $0<R_{X} = |\int_{X}(\omega+\sqrt{-1}c_1(L))^n|$.  Since $\hat{\theta} \in ((n-1)\frac{\pi}{2}, n\frac{\pi}{2})$ we have that
\[
{\rm Re}(Z_{X}(L)) <0,\qquad Z_{X}(L) \in \mathbb{H} = \{z \in \mathbb{C} : {\rm Im}(z)>0\}
\]
and hence $Z_{X}(L)$ lies in the upper half-plane, with negative real component.

With this observation it is not difficult (though it does get slightly unwieldy) to write down the expected Chern number inequalities in any dimension.   Unfortunately, the techniques used to prove Proposition~\ref{prop: chernNum} do not easily carry over to the higher dimensional case.  For example, in dimension $4$ suppose that $L$ has $\hat{\theta}\in (\frac{3\pi}{2}, 2\pi)$.  The path $Z_{X}(t)$ is given by
\[
-\left(t^{4}\omega^{4}-6t^2\omega^{2}.c_1(L)^2 + c_{1}(L)^4\right) -4t\sqrt{-1}\left(t^2c_{1}(L).\omega^{3}-c_1(L)^3. \omega\right) 
\]
where we have written $a.b = \int_{X}a\wedge b$.  For $t \approx +\infty$, $Z_{X}(t)$ lies near the negative real axis.  Since $Z_{X}(1) \in \mathbb{H}$, we must have ${\rm Im}(Z_{X}(T_*) )=0$ for some $T_{*} \in (1,\infty)$, so
\[
\frac{c_{1}(L)^3. \omega}{c_1(L). \omega^{3}} >1.
\]
Furthermore, at $T_{*}$ we must have that ${\rm Re}(Z_{X}(T_{*})) >0$.  Solving for $T_{*}$ and plugging in yields
\begin{conj}
Suppose $L$ admits a solution of dHYM on the K\"ahler $4$-fold $(X,\omega)$, with angle $\hat{\theta} \in (\frac{3 \pi}{2}, 2\pi)$.  Then the following Chern number inequalities hold
\[
\frac{c_{1}(L)^3.\omega}{c_1(L). \omega^{3}} >1.
\]
and
\[
\frac{\left(c_{1}(L)^3. \omega\right)\left(\omega^{4}\right)}{c_1(L).\omega^{3}} - 6\left(c_{1}(L)^2.\omega^2\right) + \frac{\left(c_1(L).\omega^{3}\right)\left(c_1(L)^4\right)}{c_{1}(L)^3. \omega}<0
\]
\end{conj}
If $L$ has $\mathcal{H} \ne \emptyset$, and $\hat{\theta} \in ((n-1)\frac{\pi}{2}, n\frac{\pi}{2})$, Corollary~\ref{cor: subVarObstr} implies that $Z_{V}(L) \in \mathbb{H}$ for every irreducible analytic subset $V\subset X$. If, in addition, $L$ admits a hermitian metric solving the  dHYM equation, then we write
\[
\frac{Z_{V}(L)}{Z_{X}(L)} = \frac{1}{R_{X}} e^{-\sqrt{-1}(\hat{\theta} - (n-p)\frac{\pi}{2})}\int_{V}(\omega+\sqrt{-1}c_1(L))^{p},
\]
and so by Corollary~\ref{cor: subVarObstr} we must have
\[
0<{\rm Im}\left(\frac{Z_{V}(L)}{Z_{X}(L)}\right) =\frac{R_{V}}{R_{X}} \sin\left({\rm Arg}(Z_{V}(L)) - {\rm Arg}(Z_{X}(L))\right).
\]
We illustrate this situation in Figure~\ref{fig: stabPic}.

\begin{figure}
\begin{center}
\psset{unit=.007in}
\begin{pspicture}(-250,-260)(250,50)
	\newgray{llgray}{.90}
	\psset{linecolor=llgray}
	\pspolygon*(0,-120)(-200,-20)(-200,-120)
	\psset{linecolor=black}
	\psbezier[linewidth=0.3pt](60,-210)(100,-100)(0,-20)(-100,-70)
	\psline[linewidth=0.3pt, linestyle=dashed, dash=3pt 2pt](60,-210)(49, -240) 
	\psline(-200,-120)(200,-120)
	\psline(0,0)(0,-240)
	\psdot(-100,-70)
	\uput[110](-100,-50){$Z_{X}(L)$}
	\uput[90](80,-100){$Z_{X}(t)$}
	\psset{linestyle=dashed,dash=3pt 2pt}
	\psline(0,-120)(-200,-20)
\end{pspicture}
\caption{The path $Z_{X}(t)$, and its endpoint $Z_{X}(L)$.  If $L$ is stable, then $Z_{V}(L)$ must lie in the gray region for every irreducible analytic set $V\subset X$.} \label{fig: stabPic}
\end{center}
\end{figure}
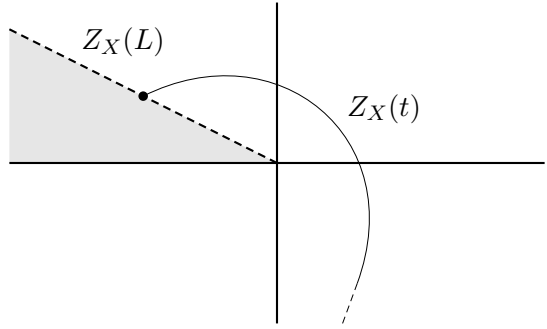

Assume now that we are in dimension $3$.  By assumption, we have $\hat{\theta}_{X}(L) \in(\pi, \frac{3\pi}{2})$, and so $\phi_{X}(L)\in (\frac{\pi}{2}, \pi)$.  If $V$ has dimension $1$, then
\[
Z_{V}(t) = -\int_{V}c_1(L) + \sqrt{-1}t\int_{V}\omega.
\]
Since $Z_{V}(1)$ must lie in the shaded region in Figure~\ref{fig: stabPic}, we see that the lifted angle $\hat{\theta}_{V}(L) \in (\phi_{V}(L)-\frac{\pi}{2},\frac{\pi}{2})$, and so
\[
\phi_{V}(L) \in(\phi_{X}(L), \pi).
\]
If $V$ has dimension $2$ then we have
\[
2Z_{V}(t) = \int_{V}t^2\omega^2-c_{1}(L)^2 + \sqrt{-1}2t\int_{V}c_{1}(L)\wedge \omega
\]
In this case ${\rm Im}(Z_{V}(1)) >0$, and so $Z_{V}(t)$ must lie in $\mathbb{H}$ for all $t\in[1,+\infty)$.  It follows that the lifted angle must satisfy $\hat{\theta}_{V}(L) \in (\phi_{X}(L), \pi)$.  Since ${\rm dim}V =2$ we get
\[
\phi_{V}(L) \in (\phi_{X}(L), \pi).
\]
We summarize this in the following proposition,
\begin{prop}\label{prop: BridgStabInt}
Suppose that $(X,\omega)$ is a K\"ahler 3-fold and $L\rightarrow X$ is a holomorphic line bundle.  If $L$ admits a solution of dHYM with lifted angle $\hat{\theta}\in(\pi, \frac{3\pi}{2})$.  Then
\begin{itemize}
\item[(i)] The Chern number inequality~\eqref{eq: chernNum} holds, and so the lifted angle is well-defined.
\item[(ii)] $Z_{X}(L) \in \mathbb{H}$, and the slicing angle $\phi_{X}(L) \in (\frac{\pi}{2}, \pi)$.
\item[(iii)]  For every irreducible analytic subset $V\subset X$, $Z_{V}(L) \in \mathbb{H}$.
\item[(iv)] The slicing angle $\phi_{V}(L)$ satisfies
\begin{equation}\label{eq: BridStabIneq}
\phi_{V}(L) > \phi_{X}(L).
\end{equation}
\end{itemize}
\end{prop}

The following conjecture updates a conjecture from \cite{CJY}.
\begin{conj}
The converse of Proposition~\ref{prop: BridgStabInt} holds.
\end{conj}

Note that the ``small radius limit" of this conjecture for ample line bundles was proven by the first author and Sz\'ekelyhidi \cite{CoSz}; see \cite{CXY} for a discussion.  One can formulate analogous conjectures in higher dimensions.  However, we expect that in dimension $4$, for example, one also needs to impose Chern number inequalities on $3$ dimensional subvarieties $V$ ensuring that the phase $\phi_{V}(L)$ is well-defined. We also note

\begin{cor}\label{cor: numChar}
Suppose $L$ is a holomorphic line bundle over $(X,\omega)$ with $Z_{X}(L) \in \mathbb{H}$.  If there exists an irreducible analytic subvariety $V \subset X$ with $Z_{V}(L) \notin \mathbb{H}$, then $\mathcal{H} = \emptyset$.
\end{cor}

This corollary should be compared with the Nakai-Moishezon criterion of ampleness, and Demailly-P\u{a}un's \cite{DP} numerical criterion for the existence of K\"ahler metrics in the class $[\alpha]$. We conjecture a converse to Corollary~\ref{cor: numChar}

\begin{conj}
A line bundle $L$ (or more generally a real $(1,1)$ class $[\alpha]$) with $Z_{X}(L) \in \mathbb{H}$ has $\mathcal{H}\ne \emptyset$ if and only if $Z_{V}(L) \in \mathbb{H}$ for all $V\subset X$.
\end{conj}

In dimensions greater than $3$ it seems possible that imposing some further Chern number inequalities may also be necessary.  These extra conditions should be thought of as analogous to the conditions appearing in \cite[Theorem 4.3]{DP}.  A solution to this conjecture, or more generally an algebraic characterization of the non-emptiness of $\mathcal{H}$, would also give an algebraic characterization of parameter $\delta$ appearing in the construction of the model curves in Section~\ref{sec: AlgObstr} as a non-linear Seshadri constant.

It is instructive to consider a simple example.

\begin{ex}\label{eq: blp3}
Consider $X = {\rm Bl}_{p}\mathbb{P}^3$.  Let $H$ be the pull-back of hyperplane, and $E$ be the exceptional divisor of the blow-up, and note that $H^3=E^3=1$ and $E.H^2=H.E^2=0$.  We take $\omega= 2H-E$, and consider $L_{a,b}= aH-bE$.  Consider the paths
\[
\begin{aligned}
Z_{X}(t) &=( b^3 + 12at^3-3bt^2-a^3) + \sqrt{-1}(6at^2-7t^3-3bt^2)\\
Z_E(t) &= (t^2-b^2) + \sqrt{-1}(2bt)
\end{aligned}
\]
We are interested in studying when $L$ admits a solution of dHYM with phase $\hat{\theta} \in (\pi, \frac{3\pi}{2})$.  In this family, four different phenomena occur
\begin{itemize}
\item When $a=5$, $b=3$, then $Z_{X}(L_{5,3}), Z_{E}(L_{5,3}) \in \mathbb{H}$, the lifted angles are well-defined, and $\phi_{X}(L)<\phi_{X}(E)$. 
\item When $a=5$, $b=-2$ the lifted angles are well-defined, $Z_{X}(L_{5,-2}) \in \mathbb{H}$, but $Z_{E}(L_{5,-2}) \notin \mathbb{H}$, and so by Corollary~\ref{cor: subVarObstr} the space $\mathcal{H}(L_{5,-2})$ is empty.
\item When $a=5, b=1$ $Z_{X}(L_{5,1}), Z_{E}(L_{5,1}) \in \mathbb{H}$, the lifted angles are well-defined, but $\phi_{E}(L_{5,1}) < \phi_{X}(L_{5,1})$ and so by Proposition~\ref{prop: BridgStabInt} no solution of dHYM exists.
\item When $b=-3$, then there exists $a\in(0.3, 0.6)$ so that the path $Z_{X}(t)$ passes through the origin.  The lifted angled is not defined and by Proposition~\ref{prop: chernNum} no solution of dHYM exists.
\end{itemize}
\end{ex}

We now recall the definition of a Bridgeland stability condition, focusing specifically on the case of interest to the $B$-model of mirror symmetry, so that the triangulated category is $D^{b}Coh(X)$.
\begin{defn}\label{defn: slicing}
A {\em slicing} $\cP$ of $D^{b}Coh(X)$ is a collection of subcategories $\cP(\varphi) \subset D^{b}Coh(X)$ for all $\varphi \in \mathbb{R}$ such that
\begin{enumerate}
\item $\cP(\varphi)[1] = \cP(\varphi+1)$ where $[1]$ denotes the ``shift" functor,
\item if $\varphi_1 > \varphi_2$ and $A\in \cP(\varphi_1)$, $B \in \cP(\varphi_2)$, then ${\rm Hom}(A,B) =0$,
\item every $E\in D^{b}Coh(X)$ admits a Harder-Narasimhan filtration by objects in $\cP(\phi_i)$ for some $1 \leq i \leq m$.
\end{enumerate}
\end{defn}

We refer to \cite{Br} for a precise definition of the Harder-Narasimhan property.  A Bridgeland stability condition on $D^{b}Coh(X)$ consists of a slicing together with a {\em central charge} (see below).  For BPS $D$-branes in the B-model, the relevant central charge was first proposed by Douglas (see, for example, \cite{Doug, DFR, BMT, AB}).  We take
\[
D^{b}Coh(X) \ni E \longmapsto Z_{D}(E):= -\int_{X}e^{-\sqrt{-1}\omega}ch(E).
\]

\begin{defn}\label{defn: BrStab}
A Bridgeland stability condition on $D^{b}Coh(X)$ with central charge $Z_{D}$ is a slicing $\cP$ satisfying the following properties
\begin{enumerate}
\item For any non-zero $E\in \cP(\varphi)$ we have
\[
Z_{D}(E) \in \mathbb{R}_{>0} e^{\sqrt{-1}\varphi},
\]
\item
\[
C := \inf \left\{ \frac{|Z_{D}(E)|}{\|ch(E)\|} : 0 \ne E \in \cP(\varphi), \varphi \in \mathbb{R} \right\} >0
\]
where $\| \cdot \|$ is any norm on the finite dimensional vector space $H^{even}(X, \mathbb{R})$.
\end{enumerate}
\end{defn}

Given a Bridgeland stability condition we define the {\em heart} to be $\mathcal{A} := \cP((0,\pi])$.  An object $A \in \mathcal{A}$ is semistable (resp. stable) if, for every surjection $A\twoheadrightarrow B$, $B\in\mathcal{A}$ we have
\[
 \varphi(A) \leq (\text { resp.} <)\,\,\varphi(B).  
\]

It seems readily apparent that the algebraic structures which predict the existence or non-existence of solutions to dHYM are closely related to Bridgeland stability.  For example, in dimension $3$, if $V$ is an irreducible analytic subvariety and $\mathcal{O}_{V}$ is the skyscraper sheaf supported on $V$, then the ideal dictionary would be
\[
\begin{aligned}
\text{ Proposition~\ref{prop: BridgStabInt} (i)-(iii) } &\Longleftrightarrow L, L\otimes \mathcal{O}_{V} \in \mathcal{A}\\
\text{ Proposition~\ref{prop: BridgStabInt} (iv) } &\Longleftrightarrow L \text{ is not destabilized by } L\twoheadrightarrow L\otimes \mathcal{O}_{V}
\end{aligned}
\]
Note that $Z_{V}(L)\ne Z_{D}(L\otimes \mathcal{O}_{V})$ in general.  Nevertheless, solutions of dHYM share two further important properties with Bridgeland stable objects.  First, recall that by the result of Jacob-Yau \cite{JY} (see Lemma~\ref{lem: BPS}), line bundles admitting solutions of dHYM have property (2) of Definition~\ref{defn: BrStab}.  Secondly, we note the following lemma, which should be compared with Definition~\ref{defn: slicing}, (2).

\begin{lem}
Suppose $L, M$ are two line bundles on $(X,\omega)$ admitting solutions of the deformed Hermitian-Yang-Mills equation with
\[ 
0<\phi_{X}(M)<\phi_{X}(L) < \pi.
\]
Then $\Hom(L,M) = 0$.
\end{lem}
\begin{proof}
Let $\alpha_L  \in c_1(L)$, $\alpha_M \in c_1(M)$ be the solutions of the deformed Hermitian-Yang-Mills equation.  Suppose for the sake of contradiction that there exists a non-zero section $\sigma \in H^{0}(X, M\otimes L^{\vee})$.  We fix a Hermitian metric $h$ on $M \otimes L^{\vee}$ with
\[
-\ddb \log h = \alpha_M- \alpha_L.
\]
The $(1,1)$ current $T:=\alpha_{M} + \ddb \log |\sigma|^{2}_{h}$ satisfies $T = \alpha_{L}$ on $X \backslash \{\sigma=0\}$.  The function $\log |\sigma|^2_h$ achieves its maximum at some point $x^{*}\in X\backslash\{\sigma=0\}$, and at $x^{*}$ we have
\[
\alpha_L(x^{*}) = T(x^{*}) \leq \alpha_{M}(x^{*})
\]
which implies that
\[
\hat{\theta}(L) = \Theta_{\omega}(\alpha_{L}(x^{*})) \leq \Theta_{\omega}(\alpha_{M}(x^*))  = \hat{\theta}(M).
\]
Since $\hat{\theta}(L) = \phi_{X}(L)-(n-2)\frac{\pi}{2}$, and similarly for $\phi_{X}(M)$, this is a contradiction.
\end{proof}

For a general flag ideal we can write the obstructions in Theorem~\ref{thm: obstrThm} in terms of the numbers
\[
Z_{E}(L) :=  \int_{E}e^{-\sqrt{-1}\mu^{*}\omega}ch(\mu^{*}L-\delta E).
\]
These invariants are more difficult to interpret in terms of Bridgeland stability in that they are computed on a birational model of $X\times \Delta$.  For example, it's unclear how one show interpret the flag ideal $\mathfrak{I} \subset \mathcal{O}_{X}\times \mathbb{C}[t]$ as an object in $D^{b}Coh(X)$.  It would be particularly enlightening to express $Z_E(L)$ in terms of data on $X$, for example the Chern class of the ideals $\mathfrak{J}_{k}$.  This seems difficult however, as the Grothendieck-Riemann-Roch Theorem includes the Todd class of the relative tangent bundle of the log-resolution $\mu$.   For example, for applications to Bridgeland stability we would hope for a positive answer to the following question
\begin{que}\label{que: gradeObj}
Does the quantity appearing on the right hand side of~\eqref{eq: limSlopeModel} depend only on the quasi-isomorphism class of
\[
\mathfrak{J}_{0}\hookrightarrow \mathfrak{J}_{1} \hookrightarrow \cdots \hookrightarrow \mathfrak{J}_{r-1} \hookrightarrow \mathcal{O}_{X}
\]
in $D^{b}Coh(X)$?
\end{que}

Finally, note that our results here apply to {\em any} $(1,1)$ class $[\alpha] \in H^{1,1}(X,\mathbb{R})$.  The relation with Bridgeland stability is via the $B$-field.  In mirror symmetry it is natural to consider not just K\"ahler forms, but complex forms $\omega+ \sqrt{-1}\beta$, where $\beta \in H^{1,1}(X,\mathbb{R})$.  In fact, it is usually assumed that $\beta \in H^{1,1}(X,\mathbb{R})/H^{1,1}(X,\mathbb{Z})$.  If $[\alpha] \in H^{1,1}(X,\mathbb{R})$, then we write it as $[\alpha] = [L] + [\beta]$, where $[L] \in H^{1,1}(X,\mathbb{Z})$, and $[\beta] \in H^{1,1}(X,\mathbb{R})/H^{1,1}(X,\mathbb{Z})$.  Then the stability of $[\alpha]$ in the sense discussed above is equivalent to the stability of $[L]$ with respect to the complexified K\"ahler form $\omega + \sqrt{-1}\beta$.  In each equation in this section one just replaces $\omega \mapsto \omega +\sqrt{-1}\beta$.  When $[L]=0$, we are asking that the structure sheaf $\mathcal{O}_{X}$ be stable with respect to the complexified form $\omega+\sqrt{-1}\beta$.  Under mirror symmetry this is related to the fact that the zero section of the SYZ fibration should be a special Lagrangian; see \cite{Leung, Gr} and the references therein for a discussion.

Note that in restricting our attention to line bundles, and more generally classes $[\alpha]\in H^{1,1}(X,\mathbb{R})$, we do need to understand the existence of a stability condition on all of $D^{b}Coh(X)$.  In particular, this allows us to avoid addressing the stability of higher rank bundles, and the existence of Harder-Narasimhan filtrations, two issues which are at the heart of constructing Bridgeland stability conditions.  It would be interesting to understand, even in examples, whether Harder-Narasimhan filtrations for unstable line bundles appear analytically. For example, one could study singular solutions of the dHYM equation, limits of the flow proposed by Jacob-Yau \cite{JY}, or the gradient flow of the $\mathcal{J}$ functional with respect to the Riemannian structure on $\mathcal{H}$, in analogy with the work of the first author with Hisamoto and Takahashi in the setting of K\"ahler-Einstein metrics \cite{CHT}.

\subsection{Higher Rank, Lower Phase}
We conclude this section  with some remarks about what one might expect in the case of line bundles with lower phase, and general vector bundles.  In the case of line bundles with lower phase, the foremost analytic difficulty is to prove the existence of regular geodesics, or even $\epsilon$-regularized geodesics, when the phase $|\hat{\theta}|\leq  (n-1)\frac{\pi}{2}$.  In this lower phase range, the Lagrangian phase operator on the product manifold $\mathcal{X}_{\epsilon}$ fails to be concave, or even have concave level sets.  Even in the local case in $\mathbb{R}^{n}$ there are examples of viscosity solutions to the constant Lagrangian phase equation which fail to have even $C^{1,1}$ regularity \cite{NV, WY1}.

One could nevertheless optimistically assume sufficiently regular $\epsilon$-geodesics exist and proceed to study the algebro-geometric consequences as in Section~\ref{sec: AlgObstr}.  However, even here new phenomena appear.  Recall that the space $\mathcal{H}$ is defined to be the space of smooth potentials $\phi$ so that
\[
|\Theta_{\omega}(\alpha_{\phi}) - \hat{\theta}| <\frac{\pi}{2}.
\]
When $\hat{\theta} > (n-1)\frac{\pi}{2}$ this reduces to the one-sided bound $\hat{\theta}-\frac{\pi}{2} < \Theta_{\omega}(\alpha_{\phi})$.  This fact was used crucially in the construction of the model curve~\eqref{eq: modelCurve}.  On the other hand, when $\hat{\theta}< (n-1)\frac{\pi}{2}$, the construction of a model curve becomes more subtle.  For example, when $|\hat{\theta}|<(n-1)\frac{\pi}{2}$ there is no model curve in $\mathcal{H}$ for the ideal $\mathfrak{I} = \mathfrak{J}_{p} + (t)$, where $\mathfrak{J}_p$ is the ideal sheaf of a point.  In particular, any attempt to understand the obstructions to existence of solutions to dHYM in lower phase must come to terms with understanding the model curves in $\mathcal{H}$.  

In higher rank one can construct a similar theory, but with formidable new difficulties.  The first caveat is that, for holomorphic bundles with non-abelian gauge group the analogue of the deformed Hermitian-Yang-Mills equation is not known.  Mathematically, the dHYM equation for vector bundles should arise from the Fourier-Mukai transform of a special Lagrangian multi-section \cite{LYZ}.  However, new difficulties arise due to the presence of holomorphic disks.  On the physics side, the dHYM equations have not been derived as the correct formulation of the non-abelian Dirac-Born-Infeld theory is not completely solved \cite{MMMS}.  Nevertheless, there is a natural guess \cite{MMMS, GM, MT} for the higher rank dHYM equation.  Fix a holomorphic bundle $E\rightarrow (X,\omega)$.  We say that a hermitian metric $H$ on $E$ solves the deformed Hermitian-Yang-Mills equation if the associated Chern connection of $H$ satisfies 
\[
{\rm Im}\left(e^{-\sqrt{-1}\hat{\theta}}(\omega\otimes I_{E} - F)^{n}\right)=0
\]
where 
\[
\int_{X}{\rm Tr}(\omega\otimes I_{E} - F)^{n} \in \mathbb{R}_{>0}e^{\sqrt{-1}\hat{\theta}}
\]
and the imaginary part is defined using the metric $H$.

Then we can consider
\[
\mathcal{H} : =\{ H: {\rm Re}\left({\rm Tr}\left((\omega\otimes I_{E} - F(H))^{n}\right)\right) >0 \}
\]
The tangent space to $\mathcal{H}$ at a point $H$ is just the space of smooth $H$-hermitian bilinear forms on $E$.  Given two sections $\psi_1,\psi_2$ we have a natural $L^2$ inner-product
\[
\langle\langle \psi_1, \psi_2\rangle\rangle_{H} = \int_{X} \langle \psi_1, \psi_2\rangle_{H} {\rm Re}\left({\rm Tr}\left((\omega\otimes I_{E} - F(H))^{n}\right)\right).
\]
One can then proceed as before, computing the geodesic equation, introducing the analogue of the $CY_{\mathbb{C}}$ functional and so on.  The resulting equations are fully nonlinear systems, and the analytic difficulties in addressing them are formidable, to say the least.  Nevertheless, one can study the consequences of existence, including higher rank versions of the Chern number inequality~\eqref{eq: chernNum}; see for example \cite{Ping} where a vector bundle version of the Monge-Amp\`ere equation is studied.

\section{The A-model}\label{sec: Amod}

Let us discuss the mirror picture to our results.  We are going to focus on the setting in which the $B$-model is a toric K\"ahler manifold.  Mirror symmetry in this setting has been extensively studied; see for example \cite{Gr1, Sei3, AKO, Abou, Abou1, FLTZ, FLTZ1, FLTZ2, FLTZ3, KChan, ChOh, HoVa} and the references there in.  In the toric setting we can apply the semi-flat SYZ proposal and the real Fourier-Mukai transform \cite{LYZ} to transform results about dHYM into results about special Lagrangians in the $A$-model.

 Suppose $(X,\omega)$ is toric, $\omega$ is an $(S^1)^n$-invariant symplectic form.  Then $X$ contains a dense open orbit of $(\mathbb{C}^{*})^n$, and $X\backslash (\mathbb{C}^{*})^n = D=  \bigcup_{i}D_i$ is a simple normal crossings divisor with $D \in -K_{X}$.  Each $D_i$ is itself a toric manifold of dimension $n-1$.  Fix coordinates $(z_1,\ldots, z_n)$ on $(\mathbb{C}^{*})^n$.  Since $D$ is an anticanonical divisor $X_0 = X\backslash D$ is Calabi-Yau with holomorphic volume form
\[
\Omega = \frac{dz_1}{z_1}\wedge \ldots \wedge \frac{dz_n}{z_n}.
\]
We are therefore in a position to discuss mirror symmetry.  The K\"ahler form $\omega|_{X_0} = \ddb \phi$ for a function $\phi$, and since $\omega$ is $(S^1)^n$-invariant we have
\[
\phi = \phi(|z_1|,\ldots,|z_n|).
\]
Let $(x_1,\ldots, x_n)$ be coordinates on $\mathbb{R}^{n}$ and let $(\theta_1,\ldots, \theta_n)$ be coordinates on $(S^{1})^n$, and define $w_i = x_i +\sqrt{-1}\theta_i$ to be holomorphic coordinates on $\mathbb{R}^{n}\times (S^1)^{n}$, which we identify with $(\mathbb{C}^{*})^{n}$ by the map $w_i\mapsto e^{w_i} = z_i$.  In particular, we can view $X\backslash D =  (\mathbb{C}^{*})^{n}$ as a Lagrangian torus fibration over $\mathbb{R}^{n}$.  Writing the K\"ahler metric out in these coordinates gives
\[
g = \sum_{i,j}\phi_{ij} (dx^i\otimes dx^j + d\theta_i \otimes d\theta_j)
\]
where $\phi_{ij} : = \frac{\del^2 \phi}{\del x_i \del x_j}$.  In particular, the metric induced on $\mathbb{R}^n$ is the hessian metric of the convex function $\phi$.

To construct the mirror manifold, we use the moment map $\mu: X \rightarrow \Delta$,  where $\Delta$ is a convex polyhedron, which we view as a subset of $\mathbb{R}^{n}$.  The divisor $D$ is mapped to the boundary of $\Delta$, and the lower dimensional toric varieties are mapped to the faces of $\Delta$.  Restricting to the interior $\Delta^{o}$ the moment map gives a smooth Lagrangian torus fibration
\[
\mu: X\backslash D \rightarrow \Delta^{0}.
\]
It is a standard fact that $\nabla \phi : \mathbb{R}^{n} \rightarrow \Delta^{o}$ is a diffeomorphism.  Namely, we let
\[
y_i = \frac{\del \phi}{\del x_i}
\]
and these define coordinates on $\Delta$.  Using $(\nabla \phi)^{-1}$ we can view the moment map as inducing an $(S^{1})^n$-fibration over $\mathbb{R}^{n}$, which comes naturally equipped with a symplectic form, and Riemannian metric given by
\[
\tilde{g} =  \sum_{i,j}\phi_{ij} dx^i\otimes dx^j +  \phi^{ij} d\tilde{\theta}_i \otimes d\tilde{\theta}_j, \qquad \tilde{\omega} = \sum_{ij} \phi^{ij} dx_i\wedge d\tilde{\theta}_j
\]
where $\phi^{ij}$ is the inverse of the metric $\phi_{ij}$, and $\tilde{\theta}_i$ are coordinates in the $(S^{1})^n$ fibers of $\mu$.  This induces a complex structure $\tilde{J}$  by defining the coordinates
\[
\tilde{w}_i = y_i+\sqrt{-1}\tilde{\theta}_i
\]
to be holomorphic, and the resulting triple defines a K\"ahler structure which is precisely the underlying manifold of the mirror of $X$.  This can all be written succinctly on $\Delta^{o}$ as follows.  Let $u = \phi^{*}$ be the Legendre transform of $\phi$, and let $y_i$ be the coordinates on $\Delta^{o}$ induced by $\nabla \phi$.  Then $\tilde{w}_i = y_i+\sqrt{-1}\tilde{\theta}_i$ are holomorphic coordinates on $Y:= (S^1)^{n}\times \Delta^{o}$, which is equipped with the K\"ahler metric
\[
\tilde{g} = u_{ij}(dy^i\otimes dy^j + d\tilde{\theta}_i \otimes d\tilde{\theta}_j)
\]
where $u_{ij} = \frac{\del^2 u}{\del y_i \del y_j} = \phi^{ij}(x(y))$.  There is a holomorphic volume form on $Y$ given by $\check{\Omega} = d\tilde{w}_{1} \wedge \ldots \wedge d\tilde{w}_n$.  We view $Y$ as a bounded domain inside of $(\mathbb{C}^{*})^n$ by defining $\tilde{z}_i = e^{\tilde{w}_i}$, in which case the holomorphic volume form is
\[
\check{\Omega} = \frac{d\tilde{z}_{1}}{\tilde{z}_1} \wedge \ldots \wedge \frac{d\tilde{z}_{n}}{\tilde{z}_n}.
\]
The mirror to $X$ is then the Landau-Ginzburg model $(Y,W)$, where $W$ is the super-potential.  At least when $X$ is toric Fano, the superpotential $W$ can be computed directly from the polytope \cite{ChOh, KChan}.

In this setting, Leung-Yau-Zaslow use the real Fourier-Mukai transform to prove
\begin{thm}[Leung-Yau-Zaslow \cite{LYZ}]
Let $L \rightarrow (X,\omega)$ be a torus invariant line bundle, equipped with a torus invariant hermitian metric $h$.  The metric $h$ induces a function $f = \log(h): \Delta^{0}\rightarrow \mathbb{R}$ whose graph $y\mapsto(y,\tilde{\theta}(y))$ defined by
\[
\tilde{\theta}_i(y) = u^{ij}\frac{\del f}{\del y_j}
\]
is a Lagrangian submanifold $\tilde{L}$.  Furthermore, $\tilde{L}$ is special Lagrangian if and only if $h$ solves the deformed Hermitian-Yang-Mills equation~\eqref{eq: introDHYM}
\end{thm}

In the statement of this theorem we are abusively viewing $\tilde{\theta}_i$ as coordinates on the universal cover of $(S^1)^n$.  Under this correspondence, changes in the metric $h \mapsto he^{-f}$ correspond to Hamiltonian deformations of $\tilde{L}$ (coupled with a flat $U(1)$ bundle).

Our goal is to describe, in rather general terms, how our results combined with the above mirror construction, imply results about the existence of special Lagrangians.  Naturally, much of what follows is somewhat speculative, owing to the complicated nature of the derived Fukaya-Seidel category $D\mathcal{F}(Y,W)$ of the Landau-Ginzburg model.  The precise details of the correspondence will be addressed in future work.

The first ingredient of the correspondence is an equivalence between torus invariant line bundles $L\rightarrow X$ and admissible Lagrangians $\tilde{L}$ in $(Y,W)$, defining objects $[\tilde{L}]$ in  $D\mathcal{F}(Y,W)$; more precisely, a correspondence between smooth, torus invariant metrics $h$ on $L$ and Lagrangians in $(Y,W)$ representing $[\tilde{L}]$.  This correspondence is based on a notion of admissibility for Lagrangians in $(Y,W)$, about which much has been written; see for instance \cite{Abou, Abou1, Sei, Sei1}.  For our purposes the notion of monomial admissibility recently introduced by Hanlon \cite{Han} seems to capture the necessary structure.

Fix a torus invariant line bundle $L\rightarrow X$, and toric metrics $h_1, h_2$ on $L$ which define points in $\mathcal{H}$.  Suppose that $\hat{\theta}(L)= \hat{\theta} > (n-1)\frac{\pi}{2}$.  Then the mirror Lagrangians $L_1, L_2$ are almost calibrated in the sense of \cite{Th}, or positive in the sense of \cite{Sol}, with lifted angle $\tilde{\theta} = -\hat{\theta}$ (since our convention introduces an extra minus sign under the mirror map).  Furthermore, the geodesic connecting $h_1, h_2$ in $\mathcal{H}$ produced by Theorem~\ref{thm: geoExistence} corresponds exactly to a geodesic, in the sense of Solomon \cite{Sol}, connecting the Lagrangians $\tilde{L}_1, \tilde{L}_2$ in the class $[\tilde{L}]$.  Suppose that 
\begin{equation}\label{eq: nest}
\mathfrak{J}_{0} \subset \mathfrak{J}_{1} \cdots \subset \mathfrak{J}_{r} = \mathcal{O}_{X}
\end{equation}
are torus invariant ideal sheaves on $X$, and let $h_{\delta}(s)$ be the associated model curve (for $\delta>0$ sufficiently small) (see Section~\ref{sec: AlgObstr}).  The Fourier-Mukai transform associates to $h_{\delta}(s)$ a family of positive (or almost calibrated) Lagrangians $\tilde{L}_{\delta}(s)$.  Suppose that
\[
\Delta = \{ y \in \mathbb{R}^{n} :  \bigcap_{i=1}^{M} \ell_i(y) \geq 0 \}
\]
where $\ell_i, 1 \leq \ell \leq M$ are linear functions with $\ell_i=0$ defining the faces of $\Delta$.  Model curves correspond under the Fourier-Mukai transform to potentials of the form
\[
f_{\delta}(s) = -\frac{\delta}{2\pi}\log\left(\sum_{\ell=0}^{r} e^{-2s\ell}\sum_{k=0}^{r-1} P_{k}(\ell_1,\ldots, \ell_M) + e^{-2sr}\right)
\]
where $P_{k}$ are polynomials having $ P_{k}(\ell_1,\ldots, \ell_M)>0$ on $\Delta^{o}$, and satisfying certain nesting conditions for their zero sets, corresponding to~\eqref{eq: nest}.  The $f_{\delta}(s)$ give rise to the positive Lagrangians $\tilde{L}_{\delta}(s)$ by Hamiltonian deformation from some initial positive Lagrangian $\tilde{L}_0$.  As $s\rightarrow \infty$, the Lagrangians $\tilde{L}_{\delta}(s)$ degenerate to a limit
\[
\lim_{s\rightarrow \infty} L_{\delta}(s) = \tilde{L}_{\delta}(\infty).
\]
In general, $\tilde{L}_{\delta}(\infty)$ will not define an element of $D\mathcal{F}(Y,W)$.  However, for rational choices of $\delta$, we can scale the Lagrangian, and the symplectic form by a sufficiently large, and divisible integer  to get a well-defined element of $D\mathcal{F}(Y,W)$.  This is akin (and indeed mirror) to the fact that $\mathbb{Q}$-divisors define elements of $D^{b}Coh(X)$ only after clearing denominators.  For now let us ignore this technical issue, with the understanding that one needs to scale appropriately to make the following discussion meaningful.

Since $\tilde{L}_{\delta}(\infty)$ is obtained by a family of Hamiltonian deformations from $\tilde{L}_0$, it should come equipped with a map  $\tilde{L}_{\delta}(\infty)\rightarrow \tilde{L}_0$.  Let $C = {\rm Cone}(\tilde{L}_{\delta}(\infty)\rightarrow \tilde{L}_0)$ be the cone in $D\mathcal{F}(Y,W)$.  Geometrically, \cite{FOOO, Sei2} one can think of $C$ as the Lagrangian connect sum $\tilde{L}_0\#\tilde{L}_{\delta}(\infty)$, although this is not rigorous in general and it is unclear whether $C$ is a ``geometric" object of $D\mathcal{F}(Y,W)$.  At a purely formal level, it is natural to expect that the limit of the symplectic Kempf-Ness functional along the family $\tilde{L}_{\delta}(s)$ is
\begin{equation}\label{eq: sympKN}
\lim_{s\rightarrow \infty} - \int_{\tilde{L}_s} \frac{\del f_{\delta}}{\del s} {\rm Im}\left(e^{-\sqrt{-1}\tilde{\theta}}\Omega\right)= - \int_{\tilde{L}_{\delta}(\infty)}{\rm Im}\left(e^{-\sqrt{-1}\tilde{\theta}}\Omega\right).
\end{equation}
The limit on the left-hand side is of course calculated on the $B$-model $X$ by the results in Section~\ref{sec: AlgObstr}.  However, evaluating the limit in terms of data on the $A$-model $(Y,W)$ appears slightly complicated.  Note that, by direct computation, one can show that for all $x\in \Delta^o$ we have
\[
- \frac{\del f_{\delta}(s)}{\del s}(x) \rightarrow 0,
\]
 and so the integral in~\eqref{eq: sympKN} should localize to an integral near the boundary of $\Delta$, where $-\frac{\del f_{\delta}(s)}{\del s}(x)$ should converge to a piecewise constant function on each face of $\del \Delta$.  This observation already seems to suggest that~\eqref{eq: sympKN} is incorrect.  However, the cone $C$ can thought of loosely as representing certain torus fibers over the boundary (see Figure~\ref{fig: conePic}).  These boundary torus fibers are not objects in $Y$, since $Y$ does not include $\del \Delta$.  Nevertheless we have $[C] =[\tilde{L}_0]- [\tilde{L}_{\delta}(\infty)]$ in the Grothendieck group of $D\mathcal{F}(Y,W)$.  Since $[{\rm Im}(e^{-\sqrt{-1}\tilde{\theta}}\Omega)].[\tilde{L}_0]=0$ by definition of $\tilde{\theta}$, the boundary integral over the torus fibers formally represented by $C$ is equal to the integral on the right hand side of~\eqref{eq: sympKN}.
 
To illustrate this, let us consider the simple example of $(\mathbb{P}^1, \omega_{FS})$.  As usual let $[Z_1:Z_2]$ be homogeneous coordinates on $\mathbb{P}^{1}$ and let $\mathbb{C}^{*}$ act by $\lambda\cdot[Z_1:Z_2] =  [\lambda\cdot Z_1:Z_2]$.  We can identify $Z_1Z_2 \ne 0$ with $\mathbb{C}^{*}$ equipped with the coordinate $z= Z_1/Z_2$.  Write $z= e^{x+\sqrt{-1}\theta}$.  The potential of the Fubini-Study metric is
 \[
 \phi = \log(1+e^{2x}),
 \]
 which is a convex function on $\mathbb{R}$.   Consider the line bundle $\mathcal{O}_{\mathbb{P}^{1}}(k)$ equipped with the model curve of metrics
 \[
h_{k,\delta}(s):= \frac{h_{FS}^{\otimes k}}{\left(|Z_1^2Z_2|_{h_{FS}^{\otimes 3}}^{2} + e^{-2s}\right)^\delta}
 \]
 which corresponds to the toric flag ideal $\mathfrak{I} = (Z_1^2Z_2) + (t) \subset \mathcal{O}_{\mathbb{P}^1}\otimes \mathbb{C}[t]$.  Recall that, by our convention, these metrics induce Lagrangians mirror to $\mathcal{O}_{\mathbb{P}^1}(-k)$ (see the beginning of Section~\ref{sec: varFrame}).  Since $\nabla \phi : \mathbb{R} \rightarrow [0,2]$, the resulting section of $T^{*}(0,2)$ under the Fourier-Mukai transform is
 \[
 (0,2) \ni y \longmapsto -ky - \delta\frac{y^2(2-y)(4-3y)}{y^2(2-y)+8e^{-2s}}.
 \]
 The Lagrangian section of the mirror is obtained by passing to the quotient $T^{*}(0,2)/\Lambda$ for a lattice $\Lambda$.  For simplicity, we take this lattice to be $\mathbb{Z}$.  We plot the resulting Lagrangians for a few values of $s$ as $s\rightarrow \infty$ in Figure~\ref{fig: degLags}
 
 \begin{figure}
\begin{center}
\psset{unit=.007in}
\begin{pspicture}(-250,-260)(250,50)
	\psset{linecolor=black}
	\psline(-250,50)(250,50)
	\psline[linestyle=dashed](-250,-270)(-250,50)
	\psline[linestyle=dashed](250,-270)(250,50)
	\psline[linewidth=0.7pt](-250,50)(250,-250)
	\pscurve[linewidth=0.3pt](-250,50)(-248,49)(-230,5)(100,-160)(230,-222)(248,-249)(250,-250)
	\psline[linewidth=0.7pt](-250,-30)(250,-216)
	\psdot(100,-160)
	\uput[90](100,-160){$p$}
	\psline[linewidth=0.3pt]{|-|}(-270,50)(-270,-30)
	\uput[180](-270,10){4}
	\psline[linewidth=0.3pt]{|-|}(270,-216)(270,-250)
	\uput[0](270,-233){2}
	\uput[90](200,-180){$\tilde{L}_{\infty}$}
	\uput[-90](-180,50){$\tilde{L}_{0}$}
\end{pspicture}
\caption{The family of degenerating Lagrangians in $T^{*}(0,2)$ induced by the metrics $h_{k,1}^{-1}(s)$ on $\mathcal{O}_{\mathbb{P}^1}(-k)$ for $k\gg1$.  The marked point $p$ corresponds to an element of ${\rm Hom}(\tilde{L}_{\infty}, \tilde{L}_0)$ in $D{\mathcal{F}}(Y,W)$.} 
\label{fig: degLags}
\end{center}
\end{figure}
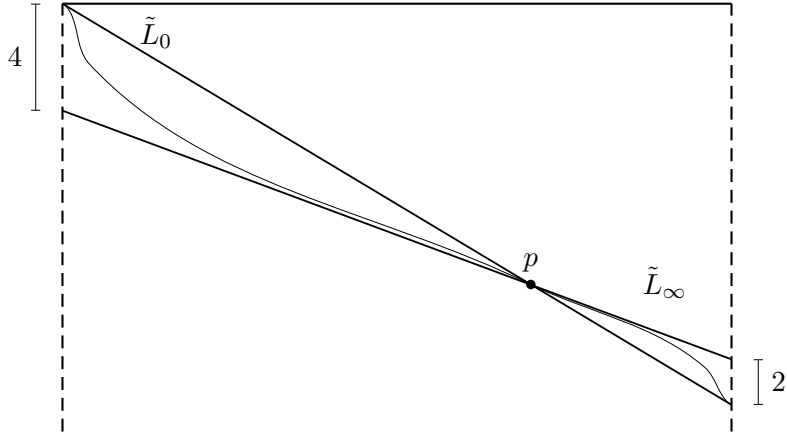
 
Recall that in this case we can think of the degeneration as corresponding to the exact sequence
\[
0 \rightarrow\mathcal{O}_{\mathbb{P}^{1}}(-3) \xrightarrow{\cdot Z_1^2Z_2} \mathcal{O}_{\mathbb{P}^{1}}\rightarrow \mathcal{O}_{\mathbb{P}^{1}}/(Z_1^2Z_2) \rightarrow 0
\]
Recalling our convention, we tensor with $\mathcal{O}_{\mathbb{P}^{1}}(-k)$ to get
\begin{equation}\label{eq: homInj}
0\rightarrow \mathcal{O}_{\mathbb{P}^{1}}(-(k+3)) \xrightarrow{\cdot Z_1^2Z_2} \mathcal{O}_{\mathbb{P}^{1}}(-k) \rightarrow \mathcal{O}_{\mathbb{P}^{1}}(-k)/(Z_1^2Z_2) \rightarrow 0.
\end{equation}
We think of $\tilde{L}_{\infty}\rightarrow \tilde{L}_{0}$ as mirror to the first map in~\eqref{eq: homInj}, and ${\rm Cone}(\tilde{L}_{\infty} \rightarrow \tilde{L}_0)$ as mirror to the cone over this map, which is quasi-isomorphic to $ \mathcal{O}_{\mathbb{P}^{1}}(-k)/(Z_1^2Z_2)$.  Geometrically, taking the Lagrangian connect sum can be visualized as depicted in Figure~\ref{fig: conePic}
 
 Of course the example of $\mathbb{P}^1$ is not particularly interesting geometrically, since, for example, all line bundles are stable.  Perhaps the first interesting case would be the case of ${\rm Bl}_{p}\mathbb{P}^2$, for which unstable bundles exist.  These unstable bundles are mirror to unstable classes in $D\mathcal{F}(Y,W)$.  Furthermore, in dimension $2$, the authors and Jacob showed that the conditions in Corollary~\ref{cor: subVarObstr} are both necessary and sufficient for the existence of solutions to dHYM \cite{CJY}.  This gives necessary and sufficient algebraic conditions for the existence of special Lagrangian sections in two dimensional Landau-Ginzburg models.  It would be enlightening to understand this correspondence in more detail.
 
 Let us briefly explain how this fits with the Thomas-Yau \cite{ThY} proposal.  Recall that in $D\mathcal{F}(Y,W)$ we have a distinguished triangle
\[
\tilde{L}_{\infty} \rightarrow \tilde{L}_0 \rightarrow C \rightarrow \tilde{L}_{\infty}[1].
\]

   \begin{figure}
\begin{center}
\psset{unit=.007in}
\begin{pspicture}(-250,-260)(250,50)
	\psset{linecolor=black} 
	\psline(-250,50)(250,50)
	\psline[linestyle=dashed](-250,-270)(-250,50)
	\psline[linestyle=dashed](250,-270)(250,50)
	\psline[linewidth=0.7pt](-250,50)(250,-250)
	\psline[linewidth=0.7pt](-250,-30)(250,-216)
	\psdot(100,-160)
	\psline[linewidth=0.3pt]{|-|}(-270,50)(-270,-30)
	\uput[180](-270,10){4}
	\psline[linewidth=0.3pt]{|-|}(270,-216)(270,-250)
	\uput[0](270,-233){2}
	\uput[90](200,-180){$\tilde{L}_{\infty}$}
	\uput[-90](-180,50){$\tilde{L}_{0}$}
	
	\psline[linewidth=0.5pt, linearc=2](-250,50)(50,-135)(-250,-30)
	\psline[linewidth=0.5pt, linearc=2](250,-250)(150,-185)(250,-216)
	\uput[-90](40,-180){${\rm Cone}(\tilde{L}_{\infty}\rightarrow \tilde{L}_{0})$}
	\psline[linewidth=0.3]{->}(60,-180)(40, -130)
	\psline[linewidth=0.3]{->}(60,-180)(150, -185)
\end{pspicture}
\caption{A geometric representation of ${\rm Cone}(\tilde{L}_{\infty}\rightarrow \tilde{L}_{0})$} 
\label{fig: conePic}
\end{center}
\end{figure}
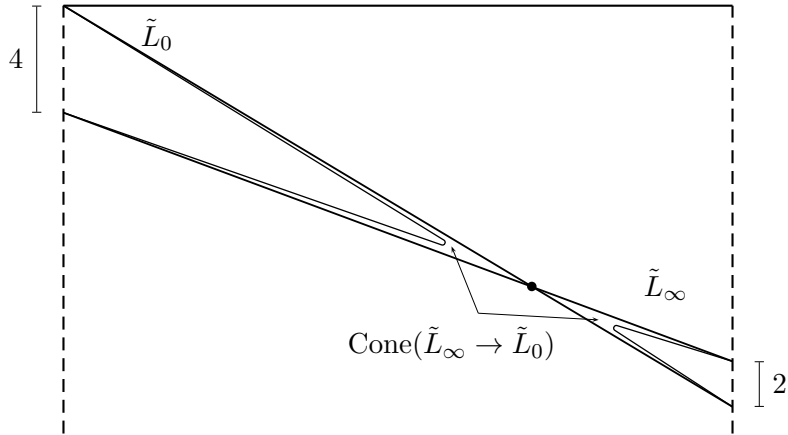

%Let us briefly explain how this fits with the Thomas-Yau \cite{ThY} proposal.  Recall that in $D\mathcal{F}(Y,W)$ we have a distinguished triangle
%\[
%\tilde{L}_{\infty} \rightarrow \tilde{L}_0 \rightarrow C \rightarrow \tilde{L}_{\infty}[1].
%\]
This gives rise to a distinguished triangle
\[
C[-1] \rightarrow \tilde{L}_{\infty} \rightarrow \tilde{L}_0\rightarrow C
\]
showing that $\tilde{L}_0 = {\rm Cone}(C[-1] \rightarrow \tilde{L}_{\infty})$.  In particular, viewing the mapping cone formally as Lagrangian connect sum leads to
\[
\tilde{L}_0 = \tilde{L}_{\infty} \# C[-1],
\]
and so $\tilde{L}_{\infty}$ can be viewed as a ``sub-object" of $L_0$ in the language of Thomas-Yau \cite{ThY}.  If the class of $\tilde{L}_0$ in $D\mathcal{F}(Y,W)$ contains a special Lagrangian, then, assuming the conjectural limit~\eqref{eq: sympKN}, using the existence of geodesics in the space of positive Lagrangians and arguing as in Section~\ref{sec: AlgObstr} we obtain
\[
{\rm Im}\left(\frac{\int_{\tilde{L}_{\infty}}\Omega}{\int_{\tilde{L}_0}\Omega}\right) <0.
\]
If we write
\[
\int_{\tilde{L}_{\infty}}\Omega \in \mathbb{R}_{>0}e^{\sqrt{-1}\phi(\tilde{L}_{\infty})}, \qquad \int_{\tilde{L}_{0}}\Omega \in \mathbb{R}_{>0}e^{\sqrt{-1}\phi(\tilde{L}_{0})}
\]
this is equivalent to $\sin(\phi(\tilde{L}_{\infty})-\phi(\tilde{L}_{0}))<0$, which, modulo lifting the phases to $\mathbb{R}$, is precisely the expected Bridgeland stability condition.

\end{document}